\documentclass[twoside,centertags                  %,draft
]{amsart}

\usepackage{amssymb}

\usepackage[cmtip,color,all,curve,
arc,poly
]{xy}
\usepackage{graphicx}
\usepackage{mathrsfs}

\usepackage{tikz}
\usetikzlibrary{matrix,arrows,calc,decorations.pathmorphing,fit,shapes.geometric,decorations.pathreplacing}
\usepackage{chngcntr}

\title{Corrections to ``Homotopy theory of nonsymmetric operads, I, II''}

\author{Fernando Muro}
\address{Universidad de Sevilla,
Facultad de Matem\'aticas,
Departamento de \'Algebra,
Avda. Reina Mercedes s/n,
41012 Sevilla, Spain}
\email{fmuro@us.es}
\urladdr{http://personal.us.es/fmuro}

\newtheorem{theorem}{Theorem}
\newtheorem{corollary}[theorem]{Corollary}
\newtheorem{lemma}[theorem]{Lemma}
\newtheorem{proposition}[theorem]{Proposition}
\newtheorem*{nh}{New Hypothesis}

\theoremstyle{definition}
\newtheorem{definition}[theorem]{Definition}

\newtheorem{remark}[theorem]{Remark}
\newtheorem{exm}[theorem]{Example}

\numberwithin{equation}{section}

\newcommand{\C}[1]{\mathscr{#1}}

\newcommand{\operad}[1]{\mathrm{Op}(#1)}
\newcommand{\oper}{\operatorname{oper}}

\newcommand{\algebra}[2]{\mathrm{Alg}_{#2}(#1)}

\newcommand{\graphs}[2]{\mathrm{Graph}_{#2}(#1)}

\newcommand{\mor}[1]{\operatorname{Mor}(#1)}

\def\r{\rightarrow} % flecha -->
\def\into{\rightarrowtail}
\def\onto{\twoheadrightarrow}

\newcommand{\id}[1]{\mathrm{id}_{#1}}

\def\dos{\mathbf{2}}

\def\st{\stackrel} % abreviatura de \stackrel

\def\unit{\mathbb{I}} % abreviatura de \underline

\def\To{\longrightarrow}

\def\unit{\mathbb{I}} % abreviatura de \underline

\def\colim{\mathop{\operatorname{colim}}}

\newdir{ >}{{}*!/-7pt/@{>}}
\newdir{> }{{}*!/10pt/@{>}}
\newdir{>> }{{}*!/10pt/@{>>}}

\begin{document}

\begin{abstract}    % type your abstract below
We correct a mistake in the construction of push-outs along free morphisms of algebras over a nonsymmetric operad  in \cite{htnso}, and we fix the affected results in \cite{htnso,htnso2}.
\end{abstract}

\maketitle

%%%%%%%%%%%%%%%%%%%%   Start of main body of article

%\renewcommand{\thesection}{\Alph{section}}
\renewcommand{\theequation}{\arabic{section}-\arabic{equation}}

\section*{Introduction}

In \cite[\S8]{htnso} we give a wrong construction of push-outs along free maps in the category of algebras over an operad (nonsymmetric). A counterexample provided by Donald Yau is presented below. Our operads live in a symmetric monoidal category $\C V$, but algebras live in a possibly nonsymmetric monoidal category $\C C$ with a central action of $\C V$ via a strong monoidal functor $z\colon\C V\r\C C$ equipped with some extra structure \cite[\S7]{htnso}. This situation generalizes the usual \emph{symmetric case}\footnote{This does not refer to symmetric operads, which are not considered here. The symmetric case would be, more generally, when $\C C$ and $z$ are symmetric monoidal, but this apparently more general case can be reduced to the former.}, where $\C C=\C V$ and $z$ is the identity. However, contrary to what we intended and claimed in the introduction of \cite{htnso}, \cite[\S8]{htnso} does not generalize Harper's construction \cite[Proposition 7.32]{htmommc}, which is the correct 
one. We here fix this mistake and its consequences in \cite{htnso,htnso2}. The main results of these papers, presented in their introductions, remain true as stated modulo a modification in the nonsymmetric monoid axiom \cite[Definition 9.1]{htnso} and a strengthening in the hypotheses of \cite[Theorem 1.13 and Corollary 1.14]{htnso2}. These 
changes do not affect applications. Moreover, the results which are purely on operads, not on  algebras, remain completely unaffected.

The corrections are presented in Section 1. The required proofs, in Section 3, are based in the categorical constructions of Section 2. As a byproduct, given a cofibrant operad $\mathcal O$ in $\C V$, we prove homotopy invariance for enveloping operads of $\mathcal O$-algebras in $\C V$ with underlying cofibrant object. This result, which is of indepedent interest, has been obtained in \cite[\S17.4]{mof} for symmetric operads in the category $\C V$ of chain complexes over a commutative ring.

%In Sections \ref{I} and \ref{II} we present  results which should replace those in  \cite{htnso} and \cite{htnso2} with the same name and number. All references in these sections are to  \cite{htnso} and \cite{htnso2}, respectively, unless indicated. Proofs are in Sections \ref{Iproofs} and \ref{IIproofs}. They are based in the categorical computations of Section \ref{2po}. As a byproduct, we obtain new homotopy invariance results for enveloping algebras and operads.

\section{Corrections}\label{s1}

Consider the following push-out diagram in the category of algebras in $\C C$ over a certain operad $\mathcal O$ in $\C V$, where the top arrow is a free $\mathcal O$-algebra map, 
\begin{equation}\label{algebrapo}
\xymatrix{
\mathcal F_{\mathcal O}(Y)\ar[r]^-{\mathcal F_{\mathcal O}(f)}\ar[d]_g\ar@{}[rd]|{\text{push}}&\mathcal F_{\mathcal O}(Z)\ar[d]^{g'}\\
A\ar[r]_-{f'}&B
} 
\end{equation}
The adjoints of $g$ and $g'$ are maps $\bar g\colon Y\r A$ and $\bar g'\colon Z\r A$ in $\C C$, respectively. 

The constructions in \cite[Lemmas 8.1 and 8.2]{htnso}, contrary to what we claimed in \cite[Theorem 8.3]{htnso}, do not yield an $\mathcal O$-algebra. The origin of the mistake is in the map \cite[(15) ]{htnso}, which must be replaced with a quotient in the category $\mor{\C C}$ of morphisms in $\C C$ \cite[\S4]{htnso}. This quotient is related to the \emph{enveloping functor-operad} $\mathcal O_A$ of the $\mathcal O$-algebra $A$, see Section \ref{s2}. A (nonsymmetric) \emph{functor-operad} $F=\{F(n)\}_{n\geq 0}$ or \emph{multitensor} \cite{ahoec} in $\C C$ is a sequence of functors $F(n)\colon\C C^n\r\C C$ equipped with composition and unit natural transformations
\begin{align*}
\circ_i\colon F (p)(\st{i-1}{\dots\dots}, F (q),\st{p-i}{\dots\dots})&\longrightarrow F (p+q-1),\quad 1\leq i\leq p,\; q\geq 0,\\
%F (n)( F (p_1),\dots, F (p_n))&\longrightarrow  F (p_1+\cdots+p_n),\qquad n\geq1,\quad p_1,\dots, p_n\geq 0,\\
u\colon \id{\C C}&\longrightarrow F (1),
\end{align*}
satisfying relations similar to operads. %Notice that $F(1)$ is a monad in $\C C$.
In the following corrected statement of \cite[Lemma 8.1]{htnso} we simply use the functor-operad structure and the fact that $\mathcal O_A(0)=A$ (a functor $\C C^0\r\C C$ is a plain object in $\C C$ since $\C C^0$ is the final category).

Recall from \cite[\S4]{htnso} that a map $f\colon Y\r Z$ in $\C C$ is the same as a functor $f\colon\dos\r\C C$ from the poset $\dos=\{0<1\}$. Given a functor $F\colon \C C^{n}\r \C C$ and maps $f_i\colon Y_i\r Z_i$, $1\leq i\leq n$, in $\C C$, $F(f_1,\dots,f_n)$ usually denotes the induced map $F(Y_1,\dots,Y_n)\r F(Z_1,\dots,Z_n)$. However, we also denote by $F(f_1,\dots,f_n)$ the composite functor $$\xymatrix@C=20pt{\dos^n\ar[rr]^-{f_1\times\cdots\times f_n}&& \C C^n\ar[r]^-{F}&\C C}.$$ Moreover, unless otherwise indicated, in this paper $F(f_1,\dots,f_n)$ denotes its latching map at $(1,\dots,1)\in\dos^n$.

\counterwithin{theorem}{section}
\renewcommand{\thesection}{\arabic{section}}
\setcounter{section}{8}
\setcounter{theorem}{0}

\begin{lemma}\label{pind2}
	There is a sequence in $\C C$
	\[A=B_0\st{\varphi_1}\To B_1\r \cdots \r B_{t-1}\st{\varphi_t}\To B_t\r \cdots,\]
	such that the morphism $\varphi_t$, $t\geq 1$, is given by the push-out square
	$$\xymatrix@C=40pt{
		\bullet\ar[r]^-{\mathcal O_A(t)(f,\dots,f)}\ar[d]_{\psi_t}\ar@{}[rd]|{\text{push}}&\bullet\ar[d]^{\bar\psi_t}\\
		B_{t-1}\ar[r]_-{\varphi_{t}}&B_t
		}$$
	where the attaching map $\psi_t$ is defined by the following maps, $1\leq i\leq t$,
	$$\xymatrix@R=15pt{
		\mathcal O_A(t)(Z,\st{i-1}{\dots\dots},Z,Y,Z,\st{t-i}{\dots\dots},Z)\ar[d]^{\text{defined by }\bar g\colon Y\r A}\\
		\mathcal O_A(t)(Z,\st{i-1}{\dots\dots},Z,A,Z,\st{t-i}{\dots\dots},Z)\ar[d]^{\circ_i}\\
		\mathcal O_A(t-1)(Z,\st{t-1}{\dots\dots},Z)\ar[d]_{\bar\psi_{t-1}\text{ if }t>1}^{\text{the identity if }t=1}\\
		B_{t-1}
		}$$
\end{lemma}

There is a canonical map $\text{\cite[(15)]{htnso}}\r\mathcal O_A(t)(f,\dots, f)$ in $\mor{\C C}$, actually a projection onto a coequalizer, see Section \ref{s2}. The map \cite[(15)]{htnso} is a coproduct indexed by the integers $n\geq 1$ and the subsets $S\subset\{1,\dots,n\}$ of cardinality $t$. Hence, the vertical composites in the following diagram
	$$\xymatrix@C=40pt{
		\bullet\ar[r]^-{\text{\cite[(15)]{htnso}}}\ar[d]\ar@/_20pt/[dd]_-{(\psi_t^{n,S})_{n,S}}
		\ar@{}[rd]|{\text{canonical map}}
		&\bullet\ar[d]\ar@/^20pt/[dd]^-{(\bar \psi_t^{n,S})_{n,S}}\\
		\bullet\ar[r]^-{\mathcal O_A(t)(f,\dots,f)}\ar[d]^-{\psi_t}\ar@{}[rd]|{\text{push}}&\bullet\ar[d]_-{\bar\psi_t}\\
		B_{t-1}\ar[r]&B_t
	}$$
are defined by their restrictions. With this notation, the rest of results in \cite[\S8]{htnso} are correct as stated. The same proof as in the symmetric case \cite[Proposition 7.32]{htmommc} works here, modulo slight changes in notation (Harper works in the more general context of left $\mathcal O$-modules instead of $\mathcal O$-algebras). In that case, the enveloping functor-operad is simply given by $\mathcal O_{A}(n)(X_{1},\dots,X_{n})=\mathcal O_{A}(n)\otimes\bigotimes_{i=1}^{n}X_{i}$, where $\mathcal O_{A}$ on the right denotes the enveloping operad of $A$. The order of factors is also important in our case, but fortunately Harper's notation, using symmetric groups, takes care of this. We think this is because the natural setting for this kind of result is precisely our nonsymmetric context. In the nonsymmetric case, $\Sigma_{p+q}/\Sigma_{p}\times \Sigma_{q}$ should be interpreted as the set of $(p,q)$-shuffles (each coset is represented by a unique $(p,q)$-shuffle).

\setcounter{section}{1}
\setcounter{theorem}{0}

\begin{exm}\label{counter0}
The following counterexample to the constructions in \cite[\S8]{htnso} is due to Donald Yau, to whom we are very grateful. We place ourselves in the symmetric  case, taking $\C V$ to be the category of complexes over a commutative ring and $\mathcal O=\mathtt{uAss}^{\C V}$ the unital associative operad. Let $A=0$ be the final  $\mathtt{uAss}^{\C V}$-algebra and $f\colon 0=Y\r Z$ a map from the zero complex. According to \cite[\S8]{htnso}, the push-out of $A$ along the free $\mathtt{uAss}^{\C V}$-algebra map spanned by $f$ is  $\bigoplus_{t\geq 1}Z^{\otimes t}$. The real push-out is $A$ again, since this $\mathtt{uAss}^{\C V}$-algebra has the unusual property that any morphism $A\r B$ must be an isomorphism.

Actually, the statements in \cite[\S8]{htnso} do not yield Schwede--Shipley's construction \cite{ammmc} of push-outs along free maps when $\mathcal O=\mathtt{uAss}^{\C V}$, but the corrected version does, see Proposition \ref{uass} below. 
\end{exm}

%Symmetry is not used in any crucial way in \cite{htmommc}, just to simplify the appearance of certain equations.

In \cite{htnso, htnso2}, we derived some results from the homotopical properties of  \cite[(15)]{htnso}, which are nice and follow easily from the push-out product axiom. We should use $\mathcal O_{A}(t)(f,\dots,f)$ instead, whose  homotopical properties are unfortunately worse and more difficult to establish. The first affected result is \cite[Proposition 9.2]{htnso}, where we must modify the statement of (2). The full correct statement is the following.

\setcounter{section}{9}
\setcounter{theorem}{1}

\begin{proposition}\label{dura}
Consider the push-out diagram \eqref{algebrapo} in $\algebra{\mathcal O}{\C C}$. 
\begin{enumerate}
\item If $f$ is a trivial cofibration in $\C C$ then the underlying morphism $f'\colon A\r B$ in $\C C$ is a relative $K'$-cell complex, where $K'$ is the class in Definition \ref{mnax}.
\item Suppose that $f$ is a cofibration in $\C C$ and that one of the following statements holds:
\begin{enumerate}
\item $A$ is a cofibrant $\mathcal O$-algebra and $\mathcal O(n)$ is cofibrant in $\C V$ for $n\geq 0$.
\item $\C V$ has cofibrant tensor unit, $\mathcal O$ is a cofibrant operad or there is a cofibration $\mathtt{Ass}^{\C V}\into \mathcal O$ or $\mathtt{uAss}^{\C V}\into \mathcal O$ in $\operad{\C V}$, and the $\mathcal O$-algebra $A$ in~$\C C$ has underlying cofibrant object.
\end{enumerate}
Then the morphism $f'\colon A\r B$ is a cofibration in $\C C$.
\end{enumerate}
\end{proposition}

Although the statement of Proposition \ref{dura} (1) remains as in \cite{htnso}, we must modify \cite[Definition 9.1]{htnso} in the following way, in order to obtain it as a trivial consequence of \cite[\S8]{htnso} (corrected version).

\setcounter{section}{9}
\setcounter{theorem}{0}

\begin{definition}\label{mnax}
The \emph{monoid axiom} in the $\C V$-algebra $\C C$ says that relative $K'$-cell complexes are weak equivalences, where $K'$ is the following class of morphisms,
\begin{align*}
K'=\{&f\otimes X, X\otimes f, \mathcal O_A(t)(f,\dots,f)\,;\,X\text{ is an object in }\C C,
f\text{ is a trivial}\\[-2pt]
&\text{cofibration in }\C C, \,
\mathcal O\text{ is an operad in }\C V,\,A\text{ is an $\mathcal O$-algebra in }\C C,\, t\geq 1\}.
\end{align*}
\end{definition}

This nonsymmetric monoid axiom is apparently stronger than the one in \cite{htnso}, but still equivalent to Schwede--Shipley's in the symmetric case, since $X\otimes f\cong f\otimes X$ and $\mathcal O_A(t)(f,\dots,f)=\mathcal O_{A}(t)\otimes f^{\odot t}$, $t\geq 1$, where $\mathcal O_{A}(t)$ on the right is an object and $f^{\odot t}$ is a trivial cofibration by the push-out product axiom. Moreover, it holds in our main nonsymmetric example, the category $\graphs{\C V}{S}$ of $\C V$-graphs with object set $S$, i.e.~\cite[Proposition 10.3]{htnso} remains valid. A new proof is required though, but the argument is very similar. It is easy to check (using the symmetry of $\C V$) that any map in $K'$ is componentwise a coproduct of maps, each of which is the tensor product of a single object in $\C V$ with a push-out product of components of $f$, which are trivial cofibrations in $\C V$. Such a push-out product is 
again a trivial 
cofibration 
by the push-out product axiom. Hence, any $K'$-cell complex is componentwise a $K$-cell complex in the sense of \cite[Definition 6.1]{htnso}, and therefore a weak equivalence by the monoid axiom for $\C V$.

Proposition \ref{dura} (2) (a) is the arity $0$ part of Proposition \ref{mensgen}, and (b) follows from Propositions \ref{sonex1}, \ref{sonex2}, and \ref{marrown}. Conditions (a) and (b) are however not necessary for $f'$ to be a cofibration. For instance, in the symmetric case, suppose that all objects in $\C V$ are cofibrant.  Then, for any operad $\mathcal O $ and any $\mathcal{O}$-algebra $A$, $\mathcal O_{A}(t)\otimes f^{\odot t}$ is a cofibration for all $t\geq 1$ by the push-out product axiom. Hence $f'$ is a transfinite composition of cofibrations in $\C V$ by \cite[\S8]{htnso} (corrected version). The same holds for $\C C=\graphs{\C V}{S}$. Nevertheless, we can give the following counterexample to the original statement of \cite[Proposition 9.2 (2)]{htnso}.

\setcounter{section}{1}
\setcounter{theorem}{1}

\begin{exm}\label{counter1}
Assume we are in the symmetric case with $\C V$ the category of chain complexes over a commutative ring $\Bbbk$. Let $\mathcal O$ be the operad whose algebras are non-unital DG-algebras $A$ with $A^3=0$. This operad has a presentation with a degree $0$ generator  $\mu\in\mathcal O(2)$ and two arity $3$ relations $\mu\circ_{1}\mu=0=\mu\circ_{2}\mu$. Moreover, $\mathcal O(1)=\Bbbk\cdot u$, $\mathcal O(2)=\Bbbk\cdot\mu$, and $\mathcal O(n)=0$ otherwise, so $\mathcal O$ is aritywise cofibrant. The enveloping operad of an $\mathcal O$-algebra $A$ is $\mathcal O_A(0)=A$, $\mathcal O_A(1)=\Bbbk\oplus A/A^{2}\oplus A/A^{2}$, $\mathcal O_A(2)=\Bbbk$, and $0$ otherwise, see Proposition \ref{cubo0} below. In general, $A\amalg\mathcal F_{\mathcal O}(\Bbbk)=\bigoplus_{n\geq 0}\mathcal O_{A}(n)$ is the direct sum of the components of the enveloping operad, see \cite[Proposition 7.28]{htmommc}, and the inclusion of the first factor $f'\colon A\hookrightarrow A\amalg\mathcal F_{\mathcal O}(\Bbbk)$ is the inclusion of the 
first
direct summand $\mathcal O_{A}(0)=A$. This inclusion is an $\mathcal O$-algebra cofibration since it is $f'$ in \eqref{algebrapo} for $f\colon 0\into \Bbbk$.

Let $\Bbbk=\mathbb Z$ and $A$ be the $\mathcal O$-algebra with two degree $0$ generators $x,y$ satisfying the relations $x^{2}=2y, xy=yx=y^{2}=0$. This algebra is $A=\mathbb Z\cdot x\oplus \mathbb Z\cdot y$ concentrated in degree $0$, so it is cofibrant as a complex. The cokernel of $f'$ in $\C V$ contains $\mathcal O_{A}(1)$, which is not free since $A/A^{2}=\mathbb Z\oplus\mathbb Z/2$, so $f'$ is not a cofibration in $\C V$.
\end{exm}

The modifications made in \cite[Proposition  \ref{dura}]{htnso} have no impact on \cite[Lemma 9.4 and Corollary 9.5]{htnso}, however \cite[Lemma 9.6 and Theorem 1.3]{htnso} require a new proof. They follow from the arity $0$ part of Proposition \ref{masgen}. Actually, \cite[Theorem 6.7]{htnso2}, which is \cite[Theorem 1.3]{htnso} with one less hypothesis, follows in this way. Similarly \cite[Theorem D.4]{htnso2}, which generalizes \cite[Theorem 6.7]{htnso2}, see Remark \ref{ende}. The rest of results in \cite{htnso} are unaffected by the corrections. We now move to \cite{htnso2}.

% \setcounter{section}{2}
% 
% \section{Corrections to \cite{htnso2}}\label{II}

\addtocounter{theorem}{1}

In \cite[Remark 6.4]{htnso2}  we give a short description of the results in \cite[\S8]{htnso}. Since we have corrected \cite[Lemma 8.1]{htnso}, we must modify \cite[Remark 6.4]{htnso2} accordingly. More precisely \cite[(6-2)]{htnso2} should be replaced with $\mathcal O_A(t)(f,\dots,f)$. This modification  forces us to give a new proof of \cite[Proposition 7.3]{htnso2}. It is the arity $0$ part of Proposition \ref{masgen1/2}. We must replace the monoid axiom in \cite[Definition 2.3 (3)]{htnso2} with the new version of \cite[Definition \ref{mnax}]{htnso} above. In addition, the following new hypothesis should be added to the indicated results.

\begin{nh}[{\cite[Theorems 1.13, 8.1, and D.13, Corollaries 1.14 and 8.2, and Propositions 8.3 and D.14]{htnso2}}]
	The operad $\mathcal O$ is cofibrant in $\operad{\C V}$ or admits a cofibration from the associative operad $\mathtt{Ass}^{\C V}\into\mathcal O$ or from the unital associative operad $\mathtt{uAss}^{\C V}\into\mathcal O$.
\end{nh} 

The most general of these results is \cite[Theorem D.13]{htnso2}, which follows from Theorem \ref{cuandoex}, Propositions \ref{sonex1} and \ref{sonex2}, and Remark \ref{ende} below. The remark is not required for the results which are not in \cite[Appendix D]{htnso2}, provided we also assume that the tensor unit of $\C V$ and $\mathcal O(n)$, $n\geq 0$, are cofibrant. Note that \cite[Lemmas 6.6 and D.1]{htnso2} are not useful any more, since the map \cite[(6-2)]{htnso2} plays no role after the corrections. Finally, \cite[Corollary D.2]{htnso2}, which generalizes \cite[Proposition 9.2 (2)]{htnso}, should be similarly corrected.

\setcounter{section}{4}
\setcounter{theorem}{1}
\renewcommand{\thetheorem}{\Alph{section}.\arabic{theorem}}

\begin{corollary}
	Suppose that $\C C$ satisfies the strong unit axiom and that one of the following statements holds:
	\begin{enumerate}
		\item $A$ is a cofibrant $\mathcal O$-algebra and $z(\mathcal O(n))$ is pseudo-cofibrant in $\C C$ for $n\geq 0$.
		\item $\mathcal O$ is a cofibrant operad or there is a cofibration $\mathtt{Ass}^{\C V}\into \mathcal O$ or $\mathtt{uAss}^{\C V}\into \mathcal O$ in $\operad{\C V}$, and the $\mathcal O$-algebra $A$ in~$\C C$ has underlying pseudo-cofibrant object.
	\end{enumerate}
	Then any cofibration $\phi\colon A\into B$ in $\algebra{\mathcal O}{\C C}$ is also a  cofibration in $\C C$.
\end{corollary}

The proof follows the same steps as the proof of Proposition 9.2 (2) above, considering also Remark \ref{ende}, and \cite[Corollary D.3]{htnso2} is an immediate consequence. There are no more affected results in \cite{htnso2}.

\setcounter{section}{1}
\setcounter{theorem}{2}
\renewcommand{\thetheorem}{\arabic{section}.\arabic{theorem}}

\begin{exm}\label{counter2}
We here give a counterexample to the original \cite[Theorem 1.13]{htnso2} for $\C V$ the category of complexes over a field $\Bbbk$. It can also be used to disprove the other results where the new hypothesis is required. Let us place ourselves in the context of Example \ref{counter1}. Let $B=\Bbbk\cdot x$ be the $\mathcal O$-algebra concentrated in degree $0$ with $x^{2}=0$, and $A$ the $\mathcal O$-algebra concentrated in degrees $1$ and $0$
$$\cdots\r 0\r \Bbbk\cdot z\To \Bbbk\cdot y\oplus \Bbbk\cdot y^{2}\r0\r\cdots$$
with $yz=zy=z^{2}=0$ and $d(z)=y^{2}$. There is an obvious weak equivalence of $\mathcal O$-algebras $\varphi \colon A\st{\sim}\r B$ defined by $\varphi(y)=x$. If $\psi\colon A\rightarrowtail A\amalg \mathcal F_{\mathcal O}(\Bbbk)=C=\bigoplus_{n\geq 0}\mathcal O_{A}(n)$ is the inclusion of the first factor of the coproduct then so is $\psi'\colon B\rightarrowtail B\amalg \mathcal F_{\mathcal O}(\Bbbk)=B\cup_AC=\bigoplus_{n\geq 0}\mathcal O_{B}(n)$. However, $\varphi'\colon C\r B\cup_AC$ cannot be a weak equivalence since $\mathcal O_{B}$ is concentrated in degree $0$ but $\mathcal O_{A}(1)=\Bbbk\oplus A/A^{2}\oplus A/A^{2}$ and $H_{1}(A/A^{2})= \Bbbk\cdot z$.
\end{exm}

\setcounter{section}{1}

\section{Enveloping functor-operads}\label{2po}\label{s2}

Given arbitrary categories $\C M$ and $\C N$, we will work with the big categories $\C M^{\C N^n}$ of functors $\C N^n\r\C M$ and natural transformations between them, $n\geq 0$. All invoked categorical notions in $\C M^{\C N^n}$ will tacitly have the pointwise meaning. In this way, we will not incur in any contradiction derived from the fact that morphism classes in $\C M^{\C N^n}$ may not be sets. We will also work with big categories of sequences of functors $\C M^{\C N^{(\mathbb N)}}=\prod_{n\geq 0}\C M^{\C N^n}$. %Layers of such sequences will be called \emph{arities}.

A functor $F\colon \C C^{n}\r \C C$, $n\geq 0$, will be denoted by a corolla with $n$ leaves with inner vertex labeled with $F$,
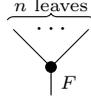
\begin{center}
	\begin{tikzpicture}[level distance=5mm, sibling distance=5mm] 
	\tikzset{every node/.style={ execute at begin node=$\scriptstyle ,%
			execute at end node=$, }}
	\node {} [grow'=up] 
	child { [fill] circle (2pt)
		child {{} node [above] (A) {}}
		child {node {\displaystyle\cdots}  edge from parent [draw=none] {}}
		child {{} node [above] (O) {}}
		node [below right] {F}
	};
	\draw [decoration={brace}, decorate] ($(A)+(-.08,0)$) -- node [above] {\text{\scriptsize $n$ leaves}} ($(O)+(.08,0)$) ;
	\end{tikzpicture}
\end{center}
The evaluation $F(X_1,\dots, X_n)$ of this functor at $n$ objects $X_1,\dots, X_n$ in $\C C$ will be designated by labeling the leaves with these objects 
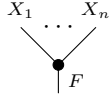
\begin{center}
	\begin{tikzpicture}[level distance=5mm, sibling distance=5mm] 
	\tikzset{every node/.style={ execute at begin node=$\scriptstyle ,%
			execute at end node=$, }}
	\node {} [grow'=up] 
	child { [fill] circle (2pt)
		child {{} node [above] {X_1}}
		child {node {\displaystyle\cdots}  edge from parent [draw=none] {}}
		child {{} node [above] {X_n}}
		node [below right] {F}};
	\end{tikzpicture}
\end{center}
If $G\colon \C C^{m}\r \C C$, $m\geq 0$, is another functor and $1\leq i\leq n$,  composition at the $i^{\text{th}}$ slot $F\circ_iG=F(\st{i-1}{\cdots\cdots},G,\st{n-i}{\cdots\cdots})\colon\C C^{n+m-1}\r \C C$ will be denoted by tree grafting
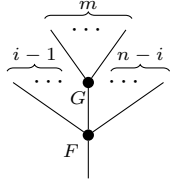
\begin{center}
	\begin{tikzpicture}[level distance=7mm, sibling distance=5mm] 
	\node {} [grow'=up]
	child {node [fill,circle, inner sep=0,outer sep =0, minimum size = 1.5mm] (L) {} 
		child{coordinate (A0)}
		child{node {$\cdots$} edge from parent[draw=none]}
		child{coordinate (M) node [fill,circle, inner sep=0,outer sep =0, minimum size = 1.5mm] {} 
			child{coordinate (A)}
			child{node {$\cdots$} edge from parent[draw=none]}
			child{coordinate (O)}
			node [below left] {$\scriptstyle G\!\!$}}
		child{node {$\cdots$} edge from parent[draw=none]}
		child{coordinate (O0)}
		node [below left] {$\scriptstyle F$}};   
	\draw [decoration={brace}, decorate] ($(A)+(-.08,.15)$) -- node [above] {\scriptsize $m$} ($(O)+(.08,.15)$) ;
	\draw [decoration={brace}, decorate] ($(A0)+(-.08,.15)$) -- node [above] {\scriptsize $i-1$} ($(M)+(-.35,.15)$) ;
	\draw [decoration={brace}, decorate] ($(M)+(.3,.15)$) -- node [above] {\scriptsize $n-i$} ($(O0)+(.08,.15)$) ;
	\end{tikzpicture}
\end{center}
In this way, a \emph{tree} (planted, planar and with leaves) in the sense of \cite[\S3]{htnso}, where each vertex $v$ is labeled with a functor $\C C^{\widetilde{v}}\r \C C$ ($\widetilde{v}$ is the \emph{arity} of $v$), denotes a functor $\C C^n\r \C C$, where $n$ is the number of leaves, and evaluation of this functor is designated by labeling the leaves with objects in $\C C$.

We say that a set of inner vertices in a tree is labeled with an object $F=\{F(n)\}_{n\geq 0}$ in $\C C^{\C C^{(\mathbb N)}}$ if each vertex $v$ in the set is labeled with $F(\widetilde{v})$. A sequence $\{V(n)\}_{n\geq 0}$ of objects in $\C V$ is regarded as the sequence of functors $(X_1,\dots, X_n)\mapsto z(V(n))\otimes\bigotimes_{i=1}^nX_i$. In this way, we can regard any operad in $\C V$ as a functor-operad.

Some labeled trees below have leaves decorated with shapes. This does not change the functorial meaning of labeled trees, it is only used to define appropriate indexing sets and to indicate the labels.

Let $\mathcal O$ be an operad in $\C V$ and $A$ an $\mathcal O$-algebra in $\C C$. Let $\mathcal O_A^0$ be the sequence of functors such that $\mathcal O_{A}^0(n)$ is the coproduct of all corollas equipped with $n$ distinguished snaky leaves, inner vertex labeled with $\mathcal O$, and straight leaves with $A$, e.g.~($n=2$)
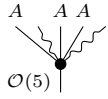
\begin{center}
	\begin{tikzpicture}[level distance=5mm, sibling distance=3mm] 
	\node {} [grow'=up]
	child {[fill] circle (2pt) 
		child {{} node [above] {$\scriptstyle A$}}
		child {edge from parent [decorate,decoration={snake,amplitude=.4mm,segment length=2mm,post length=.5mm,pre length=1.5mm}] {}}
		child {{} node [above] {$\scriptstyle A$}}
		child {{} node [above] {$\scriptstyle A$}}
		child {edge from parent [decorate,decoration={snake,amplitude=.4mm,segment length=2mm,post length=.5mm,pre length=1.5mm}] {}}
		node [below left] {$\scriptstyle \mathcal O(5)$}
	};
	\end{tikzpicture}
\end{center}
The map \cite[(15)]{htnso} is $\mathcal O_A^0(t)(f,\dots,f)$. We must replace it with $\mathcal O_A(t)(f,\dots,f)$, where $\mathcal O_A$ is the reflexive coequalizer in $\C C^{\C C^{(\mathbb N)}}$ of a diagram
\begin{equation}\label{rcoeq}
\xymatrix{\mathcal O_A^1 \ar@<.5ex>[r] \ar@<-.5ex>[r] &\mathcal O_A^0\ar@/_10pt/[l]}.
\end{equation}
Here, $\mathcal O_{A}^1(n)$ is the coproduct of trees of height $\leq 3$, $n$ leaves in level $2$, all of them snaky, and such that all level $3$ edges (if any) are straight leaves (the level of an edge is the level of the top vertex). Labels are as above, e.g.~$(n=2)$
$$
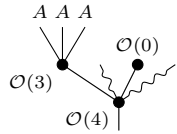
\begin{tikzpicture}[level distance=5mm, sibling distance=5mm] 
\tikzstyle{level 3}=[sibling distance=3mm] 
\node {} [grow'=up]
child {[fill] circle (2pt) 
	child {[fill] circle (2pt) 
		child {{} node [above] (A4) {$\scriptstyle A$}}
		child {{} node [above] (A5) {$\scriptstyle A$}}
		child {{} node [above] (A6) {$\scriptstyle A$}}
		node [below left] (O3) {$\scriptstyle \mathcal O(3)$}
	}
	child {edge from parent [decorate,decoration={snake,amplitude=.4mm,segment length=2mm,post length=.5mm,pre length=1.5mm}] {}}
	child {[fill] circle (2pt) node [above] (O0) {$\scriptstyle \mathcal O(0)$}}
	child {edge from parent [decorate,decoration={snake,amplitude=.4mm,segment length=2mm,post length=.5mm,pre length=1.5mm}] {}}
	node [below left] (O5) {$\scriptstyle \mathcal O(4)$}
};   
\end{tikzpicture}
$$
The arrows in \eqref{rcoeq} are defined in terms of the following three basic operations with labeled trees, $n,q\geq 0$, $1\leq i\leq p$,
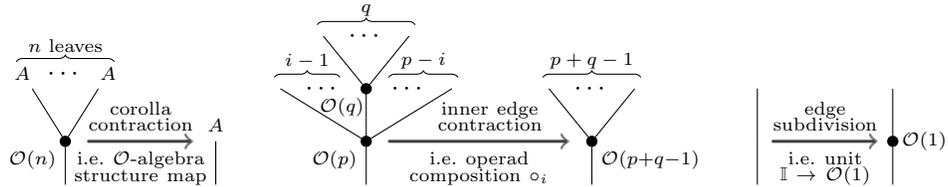
\begin{equation}\label{treeops}
\begin{array}{c}
\begin{tikzpicture}[level distance=7mm, sibling distance=5.7mm] 
\draw (-4,0) node {} [grow'=up]
child {node [fill,circle, inner sep=0,outer sep =0, minimum size = 1.5mm] (L) {} 
	child{node [above] (A) {$\scriptstyle A$}}
	child{node [above] {$\cdots$} edge from parent[draw=none]}
	child{node [above] (O) {$\scriptstyle A$}}
	node [below left] {$\scriptstyle \mathcal O(n)$}};   
\draw [decoration={brace}, decorate] ($(A)+(-.08,.15)$) -- node [above] {\scriptsize $n$ leaves} ($(O)+(.08,.15)$) ;
\draw (-2,0) node {} [grow'=up] child{coordinate (R) node [above] {$\scriptstyle A$}};
\path[->] ($(L)+(3mm,0)$) edge  [very thick, opacity=.5] node [above, opacity=1, text width=1.5cm, text centered] {\scriptsize corolla \\[-2mm] contraction} node [below, opacity=1, text width=1.8cm, text centered] {\scriptsize i.e.~$\mathcal O$-algebra\\[-2mm] structure map} ($(R)+(-3mm,0)$);
\draw (0,0) node {} [grow'=up]
child {node [fill,circle, inner sep=0,outer sep =0, minimum size = 1.5mm] (L) {} 
	child{coordinate (A0)}
	child{node {$\cdots$} edge from parent[draw=none]}
	child{coordinate (M) node [fill,circle, inner sep=0,outer sep =0, minimum size = 1.5mm] {} 
		child{coordinate (A)}
		child{node {$\cdots$} edge from parent[draw=none]}
		child{coordinate (O)}
		node [below left] {$\scriptstyle \mathcal O(q)\!\!$}}
	child{node {$\cdots$} edge from parent[draw=none]}
	child{coordinate (O0)}
	node [below left] {$\scriptstyle \mathcal O(p)$}};   
\draw [decoration={brace}, decorate] ($(A)+(-.08,.15)$) -- node [above] {\scriptsize $q$} ($(O)+(.08,.15)$) ;
\draw [decoration={brace}, decorate] ($(A0)+(-.08,.15)$) -- node [above] {\scriptsize $i-1$} ($(M)+(-.35,.15)$) ;
\draw [decoration={brace}, decorate] ($(M)+(.3,.15)$) -- node [above] {\scriptsize $p-i$} ($(O0)+(.08,.15)$) ;
\draw (3,0) node {} [grow'=up]
child{node [fill,circle, inner sep=0,outer sep =0, minimum size = 1.5mm] (R) {} 
	child{coordinate (A)}
	child{node {$\cdots$} edge from parent[draw=none]}
	child{coordinate (O)}
	node [below right] {$\scriptstyle \mathcal O(p+q-1)$}};
\draw [decoration={brace}, decorate] ($(A)+(-.08,.15)$) -- node [above] {\scriptsize $p+q-1$} ($(O)+(.08,.15)$) ;
\path[->] ($(L)+(3mm,0)$) edge  [very thick, opacity=.5] node [above, opacity=1, text width=1.9cm, text centered] {\scriptsize \quad inner edge\\[-2mm] \quad contraction} node [below, opacity=1, text width=2cm, text centered] {\scriptsize i.e.~operad\\[-2mm] composition $\circ_i$} ($(R)+(-3mm,0)$);
\draw (5.2,0) node {} [grow'=up] child{coordinate (LL) child{}};
\draw (7,0) node {} [grow'=up] child{node [fill,circle, inner sep=0,outer sep =0, minimum size = 1.5mm] (RR) {} child{} node [right] {$\scriptstyle \mathcal O(1)$}};
\path[->] ($(LL)+(2mm,0)$) edge  [very thick, opacity=.5] node [above, opacity=1, text width=1.9cm, text centered] {\scriptsize edge\\[-2mm] subdivision} node [below, opacity=1, text width=1.8cm, text centered] {\scriptsize i.e.~unit\\[-2mm] $\unit\r\mathcal O(1)$} ($(RR)+(-2mm,0)$);
\end{tikzpicture}
\end{array}
\end{equation}
Inner egde contraction and edge subdivision also make sense for functor-operads. 
The two parallel arrows in \eqref{rcoeq} are given by corolla and inner edge contraction, respectively, e.g. 
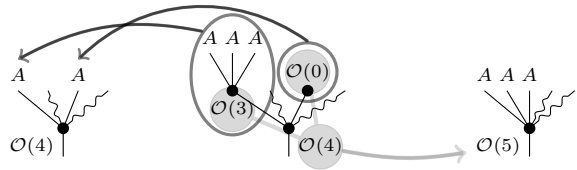
$$
\begin{tikzpicture}[level distance=5mm, sibling distance=5mm] 
\tikzstyle{level 2}=[sibling distance=4mm] 
\tikzstyle{level 3}=[sibling distance=5mm] 
\node {} [grow'=up]
child {[fill] circle (2pt) 
	child {{} node [above] (A1) {$\scriptstyle A$}}
	child {edge from parent [decorate,decoration={snake,amplitude=.4mm,segment length=2mm,post length=.5mm,pre length=1.5mm}] {}}
	child {{} node [above] (A2) {$\scriptstyle A$}}
	child {edge from parent [decorate,decoration={snake,amplitude=.4mm,segment length=2mm,post length=.5mm,pre length=1.5mm}] {}}
	node [below left] {$\scriptstyle \mathcal O(4)$}
};  
\tikzstyle{level 3}=[sibling distance=3mm] 
\tikzstyle{level 2}=[sibling distance=5mm] 
\draw (3cm,0) node {} [grow'=up]
child {[fill] circle (2pt) 
	child {[fill] circle (2pt) 
		child {{} node [above] (A4) {$\scriptstyle A$}}
		child {{} node [above] (A5) {$\scriptstyle A$}}
		child {{} node [above] (A6) {$\scriptstyle A$}}
		node [below] (O3) {$\scriptstyle \mathcal O(3)$}
	}
	child {edge from parent [decorate,decoration={snake,amplitude=.4mm,segment length=2mm,post length=.5mm,pre length=1.5mm}] {}}
	child {[fill] circle (2pt) node [above] (O0) {$\scriptstyle \mathcal O(0)$}}
	child {edge from parent [decorate,decoration={snake,amplitude=.4mm,segment length=2mm,post length=.5mm,pre length=1.5mm}] {}}
	node [below right] (O5) {$\scriptstyle \mathcal O(4)$}
};   
\node[ellipse, fit= (A4) (A5) (A6) (O3), draw, opacity=.5, very thick, inner sep=-4] (E1) {};
\path[opacity=.5, very thick, bend right = 20] (E1.north west) edge [->] (A1.north);
\node[circle, fit=(O0), draw, opacity=.5, very thick, inner sep=-2] (E2) {};
\path[->] (E2.north) edge [opacity=.5, very thick, bend right = 40] (A2.north);
\tikzstyle{level 2}=[sibling distance=3mm] 
\draw (6.2cm,0) node {} [grow'=up]
child {[fill] circle (2pt) 
	child {{} node [above] {$\scriptstyle A$}}
	child {{} node [above] {$\scriptstyle A$}}
	child {{} node [above] {$\scriptstyle A$}}
	child {edge from parent [decorate,decoration={snake,amplitude=.4mm,segment length=2mm,post length=.5mm,pre length=1.5mm}] {}}
	child {edge from parent [decorate,decoration={snake,amplitude=.4mm,segment length=2mm,post length=.5mm,pre length=1.5mm}] {}}
	node [below left] (O6) {$\scriptstyle \mathcal O(5)$}
};   
\node[circle, fill, fit=(O3), draw, opacity=.15, very thin, inner sep=-4] (OO3) {};
\node[circle, fill, fit=(O0), draw, opacity=.15, very thin, inner sep=-4] (OO0) {};
\node[circle, fill, fit=(O5), draw, opacity=.15, very thin, inner sep=-4] (OO5) {};
\path[-] (OO5) edge [opacity=.15, ultra thick] (OO3);
\path[-] (OO5) edge [opacity=.15, ultra thick] (OO0);
\path[->] (OO5) edge [opacity=.15, ultra thick, bend right =10] (O6);
\end{tikzpicture}
$$
The arrow pointing backwards in \eqref{rcoeq} is given by subdividing straight leaves, e.g. 
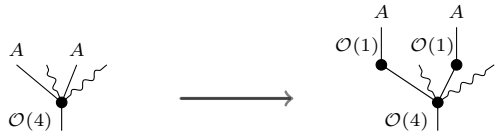
$$
\begin{tikzpicture}[level distance=5mm, sibling distance=5mm] 
\tikzstyle{level 2}=[sibling distance=4mm] 
\node {} [grow'=up]
child {[fill] circle (2pt) 
	child {{} node [above] {$\scriptstyle A$}}
	child {edge from parent [decorate,decoration={snake,amplitude=.4mm,segment length=2mm,post length=.5mm,pre length=1.5mm}] {}}
	child {{} node [above] {$\scriptstyle A$}}
	child {edge from parent [decorate,decoration={snake,amplitude=.4mm,segment length=2mm,post length=.5mm,pre length=1.5mm}] {}}
	node [below left] (O51) {$\scriptstyle \mathcal O(4)$}
};  
\tikzstyle{level 2}=[sibling distance=5mm] 
\draw (5cm,0) node {} [grow'=up]
child {[fill] circle (2pt) 
	child {[fill] circle (2pt) child{{} node [above] {$\scriptstyle A$}} node [above left] {$\scriptstyle \mathcal O(1)\!\!$}}
	child {edge from parent [decorate,decoration={snake,amplitude=.4mm,segment length=2mm,post length=.5mm,pre length=1.5mm}] {}}
	child {[fill] circle (2pt) child{{} node [above] {$\scriptstyle A$}} node [above left] {$\scriptstyle\mathcal O(1)\!\!$}}
	child {edge from parent [decorate,decoration={snake,amplitude=.4mm,segment length=2mm,post length=.5mm,pre length=1.5mm}] {}}
	node [below left] (O52) {$\scriptstyle \mathcal O(4)$}
};  
\path[->] ($(O51)+(2,.3)$) edge [very thick, opacity=.5] ($(O52)+(-1.5,.3)$);
\end{tikzpicture}
$$
Note that $\mathcal O_A(0)=A$, since \eqref{rcoeq} in arity $0$ is precisely the final part of the cotriple resolution $\mathcal F_{\mathcal O}^{2}(A)\rightrightarrows \mathcal F_{\mathcal O}(A)$. 

The functor-operad structure on $\mathcal O_{A}$ is as follows. Compositions in $\mathcal O_{A}$ are defined as the (reflexive) coequalizer of
\[\xymatrix{
	\mathcal O^1_A(p)\circ_i\mathcal O^1_A(q)\ar[r]^-{\circ_i^1}\ar@<-.5ex>[d]\ar@<.5ex>[d]&\mathcal O^1_A(p+q-1)\ar@<-.5ex>[d]\ar@<.5ex>[d]\\
	\mathcal O^0_A(p)\circ_i\mathcal O^0_A(q)\ar[r]^-{\circ_i^0}&\mathcal O^0_A(p+q-1)
	}\]
Here, horizontal maps are defined by contracting the inner edge created by grafting, e.g.~$\circ_2^0$ for $p=q=2$ contains
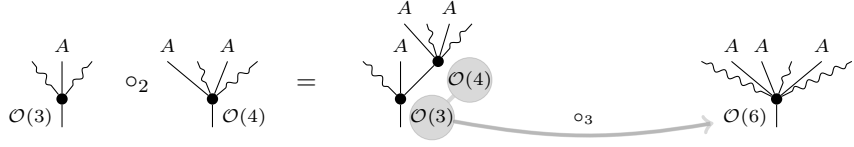
\[
\begin{tikzpicture}[level distance=5mm, sibling distance=4mm] 
  \draw (-4.5cm,0) node {} [grow'=up]
  child {node [fill,circle, inner sep=0,outer sep =0, minimum size = 1.5mm] {} %[fill] circle (2pt) 
  	child {edge from parent [decorate,decoration={snake,amplitude=.4mm,segment length=2mm,post length=.5mm,pre length=1.5mm}] {}}
  	child {{} node [above] {$\scriptstyle A$}}
  	child {edge from parent [decorate,decoration={snake,amplitude=.4mm,segment length=2mm,post length=.5mm,pre length=1.5mm}] {}}
  	node [below left] {$\scriptstyle \mathcal O(3)$} };
   \draw (-3.5cm,.7cm) node {$ \circ_2$} ;
    \tikzstyle{level 2}=[sibling distance=4mm] 
\draw (-2.5cm,0) node {} [grow'=up]
  child {node [fill,circle, inner sep=0,outer sep =0, minimum size = 1.5mm] {} %[fill] circle (2pt) 
  	child {{} node [above] {$\scriptstyle A$}}
  	child {edge from parent [decorate,decoration={snake,amplitude=.4mm,segment length=2mm,post length=.5mm,pre length=1.5mm}] {}}
  	child {{} node [above] {$\scriptstyle A$}}
  	child {edge from parent [decorate,decoration={snake,amplitude=.4mm,segment length=2mm,post length=.5mm,pre length=1.5mm}] {}}
  	node [below right] {$\scriptstyle \mathcal O(4)$} };
     \draw (-1.25cm,.7cm) node {$ =$} ;
  \tikzstyle{level 2}=[sibling distance=5mm] 
  \tikzstyle{level 3}=[sibling distance=3mm] 
\node {} [grow'=up]
child {node [fill,circle, inner sep=0,outer sep =0, minimum size = 1.5mm] {} %[fill] circle (2pt) 
  child {edge from parent [decorate,decoration={snake,amplitude=.4mm,segment length=2mm,post length=.5mm,pre length=1.5mm}] {}}
  child {{} node [above] {$\scriptstyle A$}}
  child {node [fill,circle, inner sep=0,outer sep =0, minimum size = 1.5mm] {} %[fill] circle (2pt) 
  child {{} node [above] {$\scriptstyle A$}}
  child {edge from parent [decorate,decoration={snake,amplitude=.4mm,segment length=2mm,post length=.5mm,pre length=1.5mm}] {}}
  child {{} node [above] {$\scriptstyle A$}}
  child {edge from parent [decorate,decoration={snake,amplitude=.4mm,segment length=2mm,post length=.5mm,pre length=1.5mm}] {}}
node [below right] (O5) {$\scriptstyle \mathcal O(4)$} }
 node [below right] (O4) {$\scriptstyle \mathcal O(3)$}
};  
\node [circle, fill, fit=(O4), draw, opacity=.15, very thin, inner sep=-4] (E1) {};
\node [circle, fill, fit=(O5), draw, opacity=.15, very thin, inner sep=-4] (E2) {};
\path[-] (E1) edge [opacity=.15, ultra thick] (E2);
\tikzstyle{level 2}=[sibling distance=4mm] 
\draw (5cm,0) node {} [grow'=up]
child {node [fill,circle, inner sep=0,outer sep =0, minimum size = 1.5mm] {} %[fill] circle (2pt) 
  child {edge from parent [decorate,decoration={snake,amplitude=.4mm,segment length=2mm,post length=.5mm,pre length=1.5mm}] {}}
  child {{} node [above] {$\scriptstyle A$}}
  child {{} node [above] {$\scriptstyle A$}}
  child {edge from parent [decorate,decoration={snake,amplitude=.4mm,segment length=2mm,post length=.5mm,pre length=1.5mm}] {}}
  child {{} node [above] {$\scriptstyle A$}}
  child {edge from parent [decorate,decoration={snake,amplitude=.4mm,segment length=2mm,post length=.5mm,pre length=1.5mm}] {}}
node [below left] (O8) {$\scriptstyle \mathcal O(6)$}
};
\path[->] (E1) edge [opacity=.15, ultra thick, bend right =10] node [above, opacity=1] {$\scriptstyle \circ_3$} (O8) ;
\end{tikzpicture}
\]
The unit is the composite $\id{\C C}\r\mathcal O_A^0(1)\r\mathcal O_A(1)$, where the first arrow is given by edge subdivision
$$
\begin{tikzpicture}[level distance=5mm, sibling distance=7mm] 
\tikzstyle{level 1}=[level distance=10mm] 
\tikzstyle{level 2}=[level distance=10mm] 
\draw (-.7,.5) node {$\id{\C C}=$};
\node {} [grow'=up] child{edge from parent [decorate,decoration={snake,amplitude=.4mm,segment length=2mm,post length=.5mm,pre length=1.5mm}] {}};
\tikzstyle{level 1}=[level distance=5mm] 
\tikzstyle{level 2}=[level distance=5mm] 
\draw (4,0) node {} [grow'=up] child{node [fill,circle, inner sep=0,outer sep =0, minimum size = 1.5mm] (RR) {} child{edge from parent [decorate,decoration={snake,amplitude=.4mm,segment length=2mm,post length=.5mm}] {}} node [right] {$\scriptstyle \mathcal O(1)$}};
\path[->] (2mm,5mm) edge  [very thick, opacity=.5] ($(RR)+(-2mm,0)$);
\end{tikzpicture}$$
and inclusion, an the second arrow is the natural projection.
%The $\mathcal O$-algebra structure of $A$ is not used to define composition laws in $\mathcal O_{A}$, but it must be used to check that they are well defined on coequalizers.

The inclusion of corollas with no straight leaves induces a map $\mathcal O\r\mathcal O_A^0$ such that the composite $\mathcal O\r\mathcal O_A^0\r\mathcal O_A$ is a natural morphism of functor-operads.

We now compute some enveloping functor-operads. Recall that a \emph{split coequalizer} in a category is a diagram
$$\xymatrix{U\ar@<1ex>[r]^{f}\ar[r]|{g}&V\ar@<1ex>[l]^{t}\ar@<.5ex>[r]^{e}&W\ar@<.5ex>[l]^{s}}$$
such that $ef=eg$, $es=\id{W}$, $ft=\id{V}$, and $se=gt$. The arrows pointing $\r$ are a coequalizer in the usual sense.

\begin{proposition}\label{envinitial}
	For any operad $\mathcal O$ in $\C V$, if $A=z(\mathcal O(0))$ is the initial $\mathcal O$-algebra in $\C C$ then $\mathcal O_A(n)(X_1,\dots, X_n)=z(\mathcal O(n))\otimes \bigotimes_{i=1}^nX_i$, $n\geq 0$. %The functor-operad structure is defined by the operad structure of $\mathcal O$ in the obvious way.
\end{proposition}

\begin{proof}
	Since $A=z(\mathcal O(0))$ we can cork the straight leaves (a \emph{cork} is an arity $0$ inner vertex). The natural projection is then defined by inner edge contraction, e.g.~$(n=2)$
	$$
	\begin{array}{c}
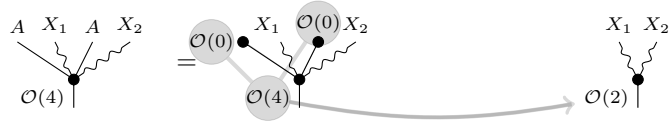
	\begin{tikzpicture}[level distance=5mm, sibling distance=5mm] 
	\node {} [grow'=up]
	child {[fill] circle (2pt) 
		child {{} node [above] (A1) {$\scriptstyle A$}}
		child {{} node [above] {$\scriptstyle X_1$} edge from parent [decorate,decoration={snake,amplitude=.4mm,segment length=2mm,post length=.5mm,pre length=1.5mm}] {}}
		child {{} node [above] (A2) {$\scriptstyle A$}}
		child {{} node [above] {$\scriptstyle X_2$} edge from parent [decorate,decoration={snake,amplitude=.4mm,segment length=2mm,post length=.5mm,pre length=1.5mm}] {}}
		node [below left] {$\scriptstyle \mathcal O(4)$}
	};  
	\draw (1.5cm,7mm) node {$=$};
	\draw (3cm,0) node {} [grow'=up]
	child {[fill] circle (2pt) 
		child {[fill] circle (2pt) node [left] (O3) {$\scriptstyle \mathcal O(0)$}}
		child {{} node [above] {$\scriptstyle X_1\quad$} edge from parent [decorate,decoration={snake,amplitude=.4mm,segment length=2mm,post length=.5mm,pre length=1.5mm}] {}}
		child {[fill] circle (2pt) node [above] (O0) {$\scriptstyle \mathcal O(0)$}}
		child {{} node [above] {$\scriptstyle X_2$} edge from parent [decorate,decoration={snake,amplitude=.4mm,segment length=2mm,post length=.5mm,pre length=1.5mm}] {}}
		node [below left] (O5) {$\scriptstyle \mathcal O(4)$}
	};   
	\draw (7.5cm,0) node {} [grow'=up]
	child {[fill] circle (2pt) 
		child {{} node [above] {$\scriptstyle X_1$} edge from parent [decorate,decoration={snake,amplitude=.4mm,segment length=2mm,post length=.5mm,pre length=1.5mm}] {}}
		child {{} node [above] {$\scriptstyle X_2$} edge from parent [decorate,decoration={snake,amplitude=.4mm,segment length=2mm,post length=.5mm,pre length=1.5mm}] {}}
		node [below left] (O6) {$\scriptstyle \mathcal O(2)$}
	};   
	\node[circle, fill, fit=(O3), draw, opacity=.15, very thin, inner sep=-4] (OO3) {};
	\node[circle, fill, fit=(O0), draw, opacity=.15, very thin, inner sep=-4] (OO0) {};
	\node[circle, fill, fit=(O5), draw, opacity=.15, very thin, inner sep=-4] (OO5) {};
	\path[-] (OO5) edge [opacity=.15, ultra thick] (OO3);
	\path[-] (OO5) edge [opacity=.15, ultra thick] (OO0);
	\path[->] (OO5) edge [opacity=.15, ultra thick, bend right =10] (O6);
	\end{tikzpicture}
	\end{array}
	$$
	This is a split coequalizer. The two morphisms going backwards are plain inclusions of coproduct factors.
\end{proof}

\begin{proposition}\label{uass}
	For $A$ a unital associative algebra in $\C C$, $\mathtt{uAss}^{\C V}_{A}(n)(X_{1},\dots,X_{n})= A\otimes\bigotimes_{i=1}^n(X_i\otimes A)$, $n\geq 0$. %Composition laws are given by multiplying consecutive copies of $A$, and the unit is the natural map $X_{1}\r A\otimes X_{1}\otimes A$ obtained by tensoring $X_{1}$ with the unit $\unit\r A$ on both sides.
\end{proposition}

\begin{proof}
	The natural projection onto the coequalizer is defined by multiplying consecutive copies of $A$ and introducing an $A$ between consecutive $X_i$ by using the unit $\unit\r A$ (also before $X_1$ or after $X_n$ if necessary), e.g.~$(n=2)$
	$$
	
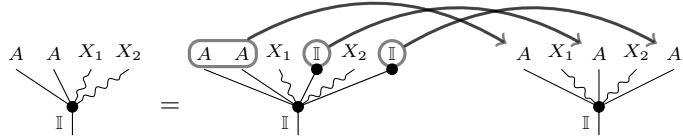
\begin{tikzpicture}[level distance=5mm, sibling distance=5mm] 
	\tikzstyle{level 2}=[sibling distance=5mm] 
	\node {} [grow'=up]
	child {[fill] circle (2pt) 
		child {{} node [above] {$\scriptstyle A$}}
		child {{} node [above] {$\scriptstyle A$}}
		child {{} node [above] {$\scriptstyle X_1$} edge from parent [decorate,decoration={snake,amplitude=.4mm,segment length=2mm,post length=.5mm,pre length=1.5mm}] {}}
		child {{} node [above] {$\scriptstyle X_2$} edge from parent [decorate,decoration={snake,amplitude=.4mm,segment length=2mm,post length=.5mm,pre length=1.5mm}] {}}
		node [below left] (O51) {$\scriptstyle \unit$}}; 
	\draw (1.3,.5) node {$=$}; 
	\draw (3cm,0) node {} [grow'=up]
	child {[fill] circle (2pt) 
		child {{} node [above] (Al) {$\scriptstyle A$}}
		child {{} node [above] (Ar) {$\scriptstyle A$}}
		child {{} node [above] {$\scriptstyle X_1$} edge from parent [decorate,decoration={snake,amplitude=.4mm,segment length=2mm,post length=.5mm,pre length=1.5mm}] {}}
		child {[fill] circle (2pt)  node [above] (I1) {$\scriptstyle \unit$}}
		child {{} node [above] {$\scriptstyle X_2$} edge from parent [decorate,decoration={snake,amplitude=.4mm,segment length=2mm,post length=.5mm,pre length=1.5mm}] {}}
		child {[fill] circle (2pt)  node [above] (I2) {$\scriptstyle \unit$}}
		node [below left] (O51) {$\scriptstyle \unit$}};  
	\draw (7cm,0) node {} [grow'=up]
	child {[fill] circle (2pt) 
		child {{} node [above] (A1) {$\scriptstyle A$}}
		child {{} node [above] {$\scriptstyle X_1$} edge from parent [decorate,decoration={snake,amplitude=.4mm,segment length=2mm,post length=.5mm,pre length=1.5mm}] {}}
		child {{} node [above] (A2) {$\scriptstyle A$}}
		child {{} node [above] {$\scriptstyle X_2$} edge from parent [decorate,decoration={snake,amplitude=.4mm,segment length=2mm,post length=.5mm,pre length=1.5mm}] {}}
		child {{} node [above] (A3) {$\scriptstyle A$}}
		node [below left] (O52) {$\scriptstyle \unit$}};
	%\path[->] ($(O51)+(2,.3)$) edge [very thick, opacity=.5] ($(O52)+(-1.5,.3)$);
	\node [rectangle, rounded corners, fit= (Al) (Ar), draw, opacity=.5, very thick, inner sep=-1] (E1) {};
	\node[circle, fit=(I1), draw, opacity=.5, very thick, inner sep=-2] (E2) {};
	\node[circle, fit=(I2), draw, opacity=.5, very thick, inner sep=-2] (E3) {};
	\path[->] (E1) edge [opacity=.5, very thick, bend left = 30] (A1);
	\path[->] (E2) edge [opacity=.5, very thick, bend left = 30] (A2);
	\path[->] (E3) edge [opacity=.5, very thick, bend left = 30] (A3);
	\end{tikzpicture}
	$$
	This is again a split coequalizer. The two maps going backwards are inclusions of factors, after doing the following kind of identifications, 
	$$
	\begin{array}{c}
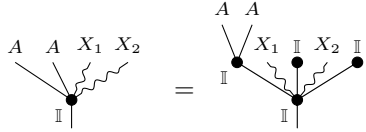
	\begin{tikzpicture}[level distance=5mm, sibling distance=7mm] 
	\tikzstyle{level 2}=[sibling distance=5mm] 
	\tikzstyle{level 3}=[sibling distance=4mm] 
	\node {} [grow'=up]
	child {[fill] circle (2pt) 
		child {{} node [above] {$\scriptstyle A$}}
		child {{} node [above] {$\scriptstyle A$}}
		child {{} node [above] {$\scriptstyle X_1$} edge from parent [decorate,decoration={snake,amplitude=.4mm,segment length=2mm,post length=.5mm,pre length=1.5mm}] {}}
		child {{} node [above] {$\scriptstyle X_2$} edge from parent [decorate,decoration={snake,amplitude=.4mm,segment length=2mm,post length=.5mm,pre length=1.5mm}] {}}
		node [below left] (O51) {$\scriptstyle \unit$}};
	\tikzstyle{level 2}=[sibling distance=4mm] 
	\draw (3cm,0) node {} [grow'=up]
	child {[fill] circle (2pt) 
		child{[fill] circle (2pt) 
			child {{} node [above] {$\scriptstyle A$}}
			child {{} node [above] {$\scriptstyle A$}}
			node [below left] {$\scriptstyle \unit$}}
		child {{} node [above] {$\scriptstyle X_1$} edge from parent [decorate,decoration={snake,amplitude=.4mm,segment length=2mm,post length=.5mm,pre length=1.5mm}] {}}
		child {[fill] circle (2pt)  node [above] {$\scriptstyle \unit$}}
		child {{} node [above] {$\scriptstyle X_2$} edge from parent [decorate,decoration={snake,amplitude=.4mm,segment length=2mm,post length=.5mm,pre length=1.5mm}] {}}
		child {[fill] circle (2pt) node [above] {$\scriptstyle \unit$}}
		node [below left] (O52) {$\scriptstyle \unit$}}; 
	\draw (1.5cm,6mm) node {$=$};
	\end{tikzpicture}
	\end{array}
	$$
\end{proof}

We leave the following similar computation for the reader.

\begin{proposition}\label{ass}
	For any non-unital associative algebra $A$ in $\C C$, we have that $\mathtt{Ass}^{\C V}_{A}(n)(X_{1},\dots,X_{n})= (A\amalg\unit)\otimes\bigotimes_{i=1}^n(X_i\otimes (A\amalg\unit))$, $n\geq 1$. %Composition laws are given by multiplying consecutive copies of the unitalization $A\amalg\unit$ and the unit is the inclusion 
	%$X_{1}\cong \unit\otimes X_{1}\otimes \unit\hookrightarrow (A\amalg\unit)\otimes X_{1}\otimes (A\amalg\unit)$.
\end{proposition}

If $\mathcal O$ is concentrated in arity $1$, we have the following result, which is an obvious consequence of the coequalizer definition.

%Enveloping functor-operads are very easy to compute when $\mathcal O(n)$ is the initial object for almost all $n\geq 0$.

\begin{proposition}\label{anillos}
	If $\mathcal O(n)=\varnothing$ for $n\neq 1$, then for any $\mathcal O$-algebra (i.e.~left $z(\mathcal O(1))$-module) $A$ in $\C C$, $\mathcal O_{A}(1)(X_{1})=z(\mathcal O(1))\otimes X_{1}$ and $\mathcal O_{A}(n)(X_{1},\dots X_{n})=\varnothing$ for $n\geq 2$.
\end{proposition}

Notice that $\mathcal O_A^1(n)$ contains a corolla with $n$ leaves, all of them snaky. The two maps $\mathcal O_A^1(n)\rightrightarrows\mathcal O_A^0(n)$ are the identity on this corolla, hence we can neglect it in coequalizer computations. Similarly, if the operad unit is an isomorphism $u\colon\unit\cong \mathcal O(1)$ we can neglect the trees in $\mathcal O_A^1$ all whose level $2$ inner vertices have arity $1$. We use this obervation in the proof of the following result.

\begin{proposition}\label{cubo0}
	In the situation of Example \ref{counter1}, for any $\mathcal O$-algebra $A$, 
	$\mathcal O_{A}(n)=0$ for $n\geq 3$, $\mathcal O_{A}(2)(X_{1}, X_{2})=X_{1}\otimes X_{2}$, and $\mathcal O_{A}(1)(X_{1})=X_{1}\oplus (A/A^{2})\otimes X_{1}\oplus X_{1}\otimes(A/A^{2})$.
\end{proposition}

\begin{proof}
	The only point which deserves mention is the fact that $\mathcal O_{A}(1)(X_{1})$ is by definition the coequalizer of 
	$$
	
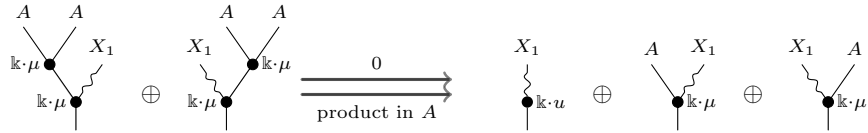
\begin{tikzpicture}[level distance=5mm, sibling distance=7mm] 
	\node {} [grow'=up]
	child {[fill] circle (2pt) 
		child {[fill] circle (2pt) 
			child {{} node [above] {$\scriptstyle A$}}
			child {{} node [above] {$\scriptstyle A$}}
			node [left] {$\scriptstyle \Bbbk\cdot\mu$}}
		child {{} node [above] {$\scriptstyle X_1$} edge from parent [decorate,decoration={snake,amplitude=.4mm,segment length=2mm,post length=.5mm,pre length=1.5mm}] {}}
		node [left] {$\scriptstyle \Bbbk\cdot\mu$}};
	\draw (1cm,.6cm) node {$\oplus$};
	\draw (2cm,0) node {} [grow'=up]
	child {[fill] circle (2pt) 
		child {{} node [above] {$\scriptstyle X_1$} edge from parent [decorate,decoration={snake,amplitude=.4mm,segment length=2mm,post length=.5mm,pre length=1.5mm}] {}}
		child {[fill] circle (2pt) 
			child {{} node [above] {$\scriptstyle A$}}
			child {{} node [above] {$\scriptstyle A$}}
			node [right] {$\scriptstyle \Bbbk\cdot\mu$}}
		node [left] {$\scriptstyle \Bbbk\cdot\mu$}};
	\draw (6,0) node {} [grow'=up] child{node [fill,circle, inner sep=0,outer sep =0, minimum size = 1.5mm] (RR) {} child{node [above] {$\scriptstyle X_1$} edge from parent [decorate,decoration={snake,amplitude=.4mm,segment length=2mm,post length=.5mm}] {}} node [right] {$\scriptstyle \Bbbk\cdot u$}};
	\draw (7cm,.6cm) node {$\oplus$};
	\draw (8cm,0) node {} [grow'=up]
	child {[fill] circle (2pt) 
		child {{} node [above] {$\scriptstyle A$}}
		child {{} node [above] {$\scriptstyle X_1$} edge from parent [decorate,decoration={snake,amplitude=.4mm,segment length=2mm,post length=.5mm,pre length=1.5mm}] {}}
		node [right] {$\scriptstyle \Bbbk\cdot\mu$}};
	\draw (9cm,.6cm) node {$\oplus$};
	\draw (10cm,0) node {} [grow'=up]
	child {[fill] circle (2pt) 
		child {{} node [above] {$\scriptstyle X_1$} edge from parent [decorate,decoration={snake,amplitude=.4mm,segment length=2mm,post length=.5mm,pre length=1.5mm}] {}}
		child {{} node [above] {$\scriptstyle A$}}
		node [right] {$\scriptstyle \Bbbk\cdot\mu$}};   
	\path[->] (3cm,.8cm) edge [very thick, opacity=.5] node [above, opacity=1] {$\scriptstyle 0$} (5cm,.8cm);
	\path[->] (3cm,.6cm) edge [very thick, opacity=.5] node [below, opacity=1] {\scriptsize product in $A$} (5cm,.6cm);
	\end{tikzpicture}
	$$
\end{proof}

We now compute enveloping functor-operads for algebras over free operads.

\begin{proposition}\label{free}
	If $\mathcal O=\mathcal F(V)$ is the free operad on a sequence $V=\{V(n)\}_{n\geq 0}$ of objects in $\C V$, an $\mathcal O$-algebra in $\C C$ is just an object $A$ equipped with structure maps $z(V(n))\otimes A^{\otimes n}\r A$, $n\geq 0$, that we regard as corolla contractions,
	\begin{equation}\label{newcc}
	\begin{array}{c}
	\begin{tikzpicture}[level distance=5mm, sibling distance=5mm] 
	\draw (-4,0) node {} [grow'=up]
	child {node [fill,circle, inner sep=0,outer sep =0, minimum size = 1.5mm] (L) {} 
		child{node [above] (A) {$\scriptstyle A$}}
		child{node [above] {$\cdots$} edge from parent[draw=none]}
		child{node [above] (O) {$\scriptstyle A$}}
		node [below left] {$\scriptstyle V(n)$}};   
	\draw [decoration={brace}, decorate] ($(A)+(-.08,.15)$) -- node [above] {\scriptsize $n$ leaves} ($(O)+(.08,.15)$) ;
	\draw (-2,0) node {} [grow'=up] child{coordinate (R) node [above] {$\scriptstyle A$}};
	\path[->] ($(L)+(3mm,0)$) edge  [very thick, opacity=.5] ($(R)+(-3mm,0)$);
	\end{tikzpicture}
	\end{array}
	\end{equation}
	and $\mathcal O_A(n)$ is the coproduct of all trees equipped with a distinguished subset of $n$ snaky leaves, inner vertices labelled with $V$ and straight leaves with $A$,  which do not contain corollas as in \eqref{newcc}, e.g.
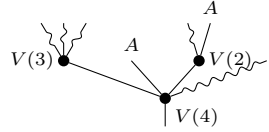
	\begin{equation*}\label{corks2}
	\begin{tikzpicture}[level distance=5mm, sibling distance=9mm] 
	\tikzstyle{level 3}=[sibling distance=3mm]
		\node {} [grow'=up]
		child{node [fill,circle, inner sep=0,outer sep =0, minimum size = 1.5mm] {} 
			child{node [fill,circle, inner sep=0,outer sep =0, minimum size = 1.5mm] {}
					child{edge from parent [decorate,decoration={snake,amplitude=.4mm,segment length=2mm,post length=.5mm,pre length=1.5mm}] {}}
					child{edge from parent [decorate,decoration={snake,amplitude=.4mm,segment length=2mm,post length=.5mm,pre length=1.5mm}] {}}
					child{edge from parent [decorate,decoration={snake,amplitude=.4mm,segment length=2mm,post length=.5mm,pre length=1.5mm}] {}}
			node [left] {$\scriptstyle V(3)$}}
            child{{} node [above] {$\scriptstyle A$}}
			  child{node [fill,circle, inner sep=0,outer sep =0, minimum size = 1.5mm] {}
				child{edge from parent [decorate,decoration={snake,amplitude=.4mm,segment length=2mm,post length=.5mm,pre length=1.5mm}] {}}
				child{{} node [above] {$\scriptstyle A$}}				
			  node [right] {$\scriptstyle V(2)$}}
           child{edge from parent [decorate,decoration={snake,amplitude=.4mm,segment length=2mm,post length=.5mm,pre length=1.5mm}] {}}
		node [below right] {$\scriptstyle V(4)$}};
	\end{tikzpicture}
	\end{equation*}
	Composition is simply given by grafting and, if necessary, iterated corolla contractions until no corolla as in \eqref{newcc} remains, e.g.~$\circ_{4}\colon\mathcal O_{A}(5)\circ_{4}\mathcal O_{A}(0)\r\mathcal O_{A}(4)$ contains
	
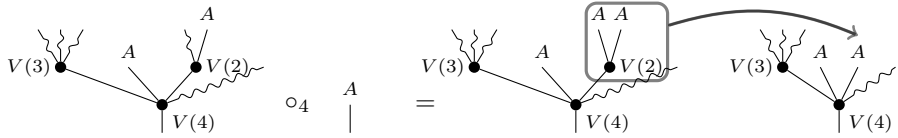
\[
	\begin{tikzpicture}[level distance=5mm, sibling distance=10mm] 
	\tikzstyle{level 3}=[sibling distance=3mm]
	\tikzstyle{level 2}=[sibling distance=9mm]
	\node {} [grow'=up]
	child{node [fill,circle, inner sep=0,outer sep =0, minimum size = 1.5mm] {} 
		child{node [fill,circle, inner sep=0,outer sep =0, minimum size = 1.5mm] {}
			child{edge from parent [decorate,decoration={snake,amplitude=.4mm,segment length=2mm,post length=.5mm,pre length=1.5mm}] {}}
			child{edge from parent [decorate,decoration={snake,amplitude=.4mm,segment length=2mm,post length=.5mm,pre length=1.5mm}] {}}
			child{edge from parent [decorate,decoration={snake,amplitude=.4mm,segment length=2mm,post length=.5mm,pre length=1.5mm}] {}}
			node [left] {$\scriptstyle V(3)$}}
		child{{} node [above] {$\scriptstyle A$}}
		child{node [fill,circle, inner sep=0,outer sep =0, minimum size = 1.5mm] {}
			child{edge from parent [decorate,decoration={snake,amplitude=.4mm,segment length=2mm,post length=.5mm,pre length=1.5mm}] {}}
			child{{} node [above] {$\scriptstyle A$}}				
			node [right] {$\scriptstyle V(2)$}}
		child{edge from parent [decorate,decoration={snake,amplitude=.4mm,segment length=2mm,post length=.5mm,pre length=1.5mm}] {}}
		node [below right] {$\scriptstyle V(4)$}};
		\tikzstyle{level 2}=[sibling distance=5mm]
		\draw (1.8cm,.5cm) node {$\circ_4$};
		\draw (2.5cm,0) node {} [grow'=up]
		child{{} node [above] {$\scriptstyle A$}};
		\draw (3.5cm,.5cm) node {$=$};
			\tikzstyle{level 2}=[sibling distance=9mm]
		\draw (5.5cm,0) node {} [grow'=up]
		child{node [fill,circle, inner sep=0,outer sep =0, minimum size = 1.5mm] {} 
			child{node [fill,circle, inner sep=0,outer sep =0, minimum size = 1.5mm] {}
					child{edge from parent [decorate,decoration={snake,amplitude=.4mm,segment length=2mm,post length=.5mm,pre length=1.5mm}] {}}
					child{edge from parent [decorate,decoration={snake,amplitude=.4mm,segment length=2mm,post length=.5mm,pre length=1.5mm}] {}}
					child{edge from parent [decorate,decoration={snake,amplitude=.4mm,segment length=2mm,post length=.5mm,pre length=1.5mm}] {}}
			node [left] {$\scriptstyle V(3)$}}
            child{{} node [above] {$\scriptstyle A$}}
			  child{node [fill,circle, inner sep=0,outer sep =0, minimum size = 1.5mm] {}
				child{{} node [above] (A1) {$\scriptstyle A$}}
				child{{} node [above] (A2) {$\scriptstyle A$}}				
			  node [right] (V2) {$\scriptstyle V(2)$}}
           child{edge from parent [decorate,decoration={snake,amplitude=.4mm,segment length=2mm,post length=.5mm,pre length=1.5mm}] {}}
		node [below right] {$\scriptstyle V(4)$}};
  \tikzstyle{level 2}=[sibling distance=5mm]
  \tikzstyle{level 3}=[sibling distance=3mm]
\draw (9cm,0) node {} [grow'=up]
    child{node [fill,circle, inner sep=0,outer sep =0, minimum size = 1.5mm] {} 
      child{node [fill,circle, inner sep=0,outer sep =0, minimum size = 1.5mm] {}
          child{edge from parent [decorate,decoration={snake,amplitude=.4mm,segment length=2mm,post length=.5mm,pre length=1.5mm}] {}}
          child{edge from parent [decorate,decoration={snake,amplitude=.4mm,segment length=2mm,post length=.5mm,pre length=1.5mm}] {}}
          child{edge from parent [decorate,decoration={snake,amplitude=.4mm,segment length=2mm,post length=.5mm,pre length=1.5mm}] {}}
      node [left] {$\scriptstyle V(3)$}}
            child{{} node [above] {$\scriptstyle A$}}
        child{{}
        node [above] (A4) {$\scriptstyle A$}}
           child{edge from parent [decorate,decoration={snake,amplitude=.4mm,segment length=2mm,post length=.5mm,pre length=1.5mm}] {}}
    node [below right] {$\scriptstyle V(4)$}};
\node[rectangle, rounded corners, fit= (A1) (A2) (V2), draw, opacity=.5, very thick, inner sep=-2] (E1) {};
\path[opacity=.5, very thick, bend left = 20] (E1) edge [->] (A4.north);
	\end{tikzpicture}
	\]
	The unit is the inclusion in arity $1$ of the tree consisting of just one snaky leaf (which represents the identity functor).
\end{proposition}

\begin{proof}
Recall from \cite[\S5]{htnso} the structure of the free operad. The functor $\mathcal O_A^0(n)$ is the coproduct of all trees with $n$ snaky leaves where inner vertices are labelled with $V$ and straight leaves with $A$, e.g.~$(n=5)$
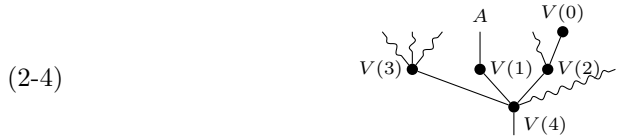
\begin{equation}\label{ooo}
\begin{array}{c}
\begin{tikzpicture}[level distance=5mm, sibling distance=9mm] 
\tikzstyle{level 3}=[sibling distance=4mm]
\node {} [grow'=up]
child{node [fill,circle, inner sep=0,outer sep =0, minimum size = 1.5mm] {} 
	child{node [fill,circle, inner sep=0,outer sep =0, minimum size = 1.5mm] {}
		child{edge from parent [decorate,decoration={snake,amplitude=.4mm,segment length=2mm,post length=.5mm,pre length=1.5mm}] {}}
		child{edge from parent [decorate,decoration={snake,amplitude=.4mm,segment length=2mm,post length=.5mm,pre length=1.5mm}] {}}
		child{edge from parent [decorate,decoration={snake,amplitude=.4mm,segment length=2mm,post length=.5mm,pre length=1.5mm}] {}}
		node [left] {$\scriptstyle V(3)$}}
	child{node [fill,circle, inner sep=0,outer sep =0, minimum size = 1.5mm] {}
		child{{} node [above] {$\scriptstyle A$}}
	node [right] {$\scriptstyle V(1)$}}
	child{node [fill,circle, inner sep=0,outer sep =0, minimum size = 1.5mm] {}
		child{edge from parent [decorate,decoration={snake,amplitude=.4mm,segment length=2mm,post length=.5mm,pre length=1.5mm}] {}}
		child{node [fill,circle, inner sep=0,outer sep =0, minimum size = 1.5mm] {} node [above] {$\scriptstyle V(0)$}}				
		node [right] {$\scriptstyle V(2)$}}
	child{edge from parent [decorate,decoration={snake,amplitude=.4mm,segment length=2mm,post length=.5mm,pre length=1.5mm}] {}}
	node [below right] {$\scriptstyle V(4)$}};
\end{tikzpicture}
\end{array}
\end{equation}
Similarly, $\mathcal O_A^1(n)$ is a coproduct indexed by the collections of trees $(T_0;T_1,\dots T_t)$ such that $T_0$ has $t+n$ leaves, $n$ of them snaky, and the leaves of $T_1,\dots, T_t$, if any, are straight. The factor indexed by $(T_0;T_1,\dots T_t)$ is obtained by labelling inner vertices with $V$, the leaves of $T_1,\dots, T_t$ with $A$, and finally grafting each $T_i$, $1\leq i\leq t$, into the $i^{\text{th}}$ straight leaf of $T_0$, e.g.~$(t=1, n=5)$
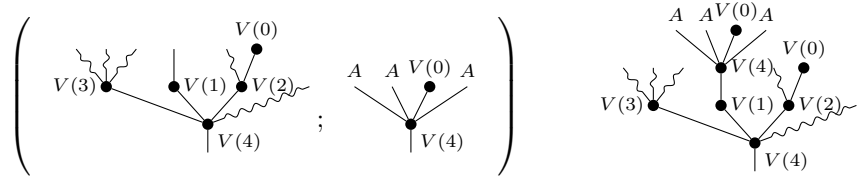
\[
\left(
\begin{array}{c}
\begin{tikzpicture}[level distance=5mm, sibling distance=9mm] 
\tikzstyle{level 3}=[sibling distance=4mm]
\node {} [grow'=up]
child{node [fill,circle, inner sep=0,outer sep =0, minimum size = 1.5mm] {} 
	child{node [fill,circle, inner sep=0,outer sep =0, minimum size = 1.5mm] {}
		child{edge from parent [decorate,decoration={snake,amplitude=.4mm,segment length=2mm,post length=.5mm,pre length=1.5mm}] {}}
		child{edge from parent [decorate,decoration={snake,amplitude=.4mm,segment length=2mm,post length=.5mm,pre length=1.5mm}] {}}
		child{edge from parent [decorate,decoration={snake,amplitude=.4mm,segment length=2mm,post length=.5mm,pre length=1.5mm}] {}}
		node [left] {$\scriptstyle V(3)$}}
	child{node [fill,circle, inner sep=0,outer sep =0, minimum size = 1.5mm] {}
		child{}
		node [right] {$\scriptstyle V(1)$}}
	child{node [fill,circle, inner sep=0,outer sep =0, minimum size = 1.5mm] {}
		child{edge from parent [decorate,decoration={snake,amplitude=.4mm,segment length=2mm,post length=.5mm,pre length=1.5mm}] {}}
		child{node [fill,circle, inner sep=0,outer sep =0, minimum size = 1.5mm] {} node [above] {$\scriptstyle V(0)$}}				
		node [right] {$\scriptstyle V(2)$}}
	child{edge from parent [decorate,decoration={snake,amplitude=.4mm,segment length=2mm,post length=.5mm,pre length=1.5mm}] {}}
	node [below right] {$\scriptstyle V(4)$}};
\draw (1.5cm,.5cm) node {;};
\tikzstyle{level 2}=[sibling distance=5mm]
\draw (2.7cm,0) node {} [grow'=up]
child{node [fill,circle, inner sep=0,outer sep =0, minimum size = 1.5mm] {} 
	child{{} node [above] {$\scriptstyle A$}}
	child{{} node [above] {$\scriptstyle A$}}
	child{node [fill,circle, inner sep=0,outer sep =0, minimum size = 1.5mm] {} node [above] {$\scriptstyle V(0)$}}
	child{{} node [above] {$\scriptstyle A$}}
	node [below right] {$\scriptstyle V(4)$}};
\end{tikzpicture}
\end{array}\right)\qquad
\begin{array}{c}
\begin{tikzpicture}[level distance=5mm, sibling distance=9mm] 
\tikzstyle{level 3}=[sibling distance=4mm]
\tikzstyle{level 4}=[sibling distance=4mm]
\node {} [grow'=up]
child{node [fill,circle, inner sep=0,outer sep =0, minimum size = 1.5mm] {} 
	child{node [fill,circle, inner sep=0,outer sep =0, minimum size = 1.5mm] {}
		child{edge from parent [decorate,decoration={snake,amplitude=.4mm,segment length=2mm,post length=.5mm,pre length=1.5mm}] {}}
		child{edge from parent [decorate,decoration={snake,amplitude=.4mm,segment length=2mm,post length=.5mm,pre length=1.5mm}] {}}
		child{edge from parent [decorate,decoration={snake,amplitude=.4mm,segment length=2mm,post length=.5mm,pre length=1.5mm}] {}}
		node [left] {$\scriptstyle V(3)$}}
	child{node [fill,circle, inner sep=0,outer sep =0, minimum size = 1.5mm] {}
		child{node [fill,circle, inner sep=0,outer sep =0, minimum size = 1.5mm] {} 
			child{{} node [above] {$\scriptstyle A$}}
			child{{} node [above] {$\scriptstyle A$}}
			child{node [fill,circle, inner sep=0,outer sep =0, minimum size = 1.5mm] {} node [above] {$\scriptstyle V(0)$}}
			child{{} node [above] {$\scriptstyle A$}}
			node [right] {$\scriptstyle V(4)$}}
		node [right] {$\scriptstyle V(1)$}}
	child{node [fill,circle, inner sep=0,outer sep =0, minimum size = 1.5mm] {}
		child{edge from parent [decorate,decoration={snake,amplitude=.4mm,segment length=2mm,post length=.5mm,pre length=1.5mm}] {}}
		child{node [fill,circle, inner sep=0,outer sep =0, minimum size = 1.5mm] {} node [above] {$\scriptstyle V(0)$}}				
		node [right] {$\scriptstyle V(2)$}}
	child{edge from parent [decorate,decoration={snake,amplitude=.4mm,segment length=2mm,post length=.5mm,pre length=1.5mm}] {}}
	node [below right] {$\scriptstyle V(4)$}};
\end{tikzpicture}
\end{array}
\]
The maps $\mathcal O_A^1(n)\rightrightarrows\mathcal O_A^0(n)$ are defined as follows. The first one maps the factor $(T_0;T_1,\dots T_t)$ to the factor $T_0$ by iterated corolla contraction $T_i\r A$, $1\leq i\leq t$, e.g.~starting with the example in the previous diagram
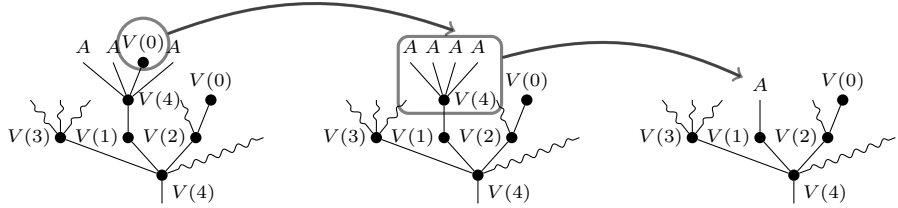
\[
\begin{tikzpicture}[level distance=5mm, sibling distance=9mm] 
\tikzstyle{level 3}=[sibling distance=4mm]
\tikzstyle{level 4}=[sibling distance=4mm]
\node {} [grow'=up]
child{node [fill,circle, inner sep=0,outer sep =0, minimum size = 1.5mm] {} 
	child{node [fill,circle, inner sep=0,outer sep =0, minimum size = 1.5mm] {}
		child{edge from parent [decorate,decoration={snake,amplitude=.4mm,segment length=2mm,post length=.5mm,pre length=1.5mm}] {}}
		child{edge from parent [decorate,decoration={snake,amplitude=.4mm,segment length=2mm,post length=.5mm,pre length=1.5mm}] {}}
		child{edge from parent [decorate,decoration={snake,amplitude=.4mm,segment length=2mm,post length=.5mm,pre length=1.5mm}] {}}
		node [left] {$\scriptstyle V(3)$}}
	child{node [fill,circle, inner sep=0,outer sep =0, minimum size = 1.5mm] {}
		child{node [fill,circle, inner sep=0,outer sep =0, minimum size = 1.5mm] {} 
			child{{} node [above] {$\scriptstyle A$}}
			child{{} node [above] {$\scriptstyle A$}}
			child{node [fill,circle, inner sep=0,outer sep =0, minimum size = 1.5mm] {} node (V00) [above] {$\scriptstyle V(0)$}}
			child{{} node [above] {$\scriptstyle A$}}
			node [right] {$\scriptstyle V(4)$}}
		node [left] {$\scriptstyle V(1)$}}
	child{node [fill,circle, inner sep=0,outer sep =0, minimum size = 1.5mm] {}
		child{edge from parent [decorate,decoration={snake,amplitude=.4mm,segment length=2mm,post length=.5mm,pre length=1.5mm}] {}}
		child{node [fill,circle, inner sep=0,outer sep =0, minimum size = 1.5mm] {} node [above] {$\scriptstyle V(0)$}}				
		node [left] {$\scriptstyle V(2)$}}
	child{edge from parent [decorate,decoration={snake,amplitude=.4mm,segment length=2mm,post length=.5mm,pre length=1.5mm}] {}}
	node [below right] {$\scriptstyle V(4)$}};
\tikzstyle{level 4}=[sibling distance=3mm]
\draw (4.2cm,0) node {} [grow'=up]
child{node [fill,circle, inner sep=0,outer sep =0, minimum size = 1.5mm] {} 
	child{node [fill,circle, inner sep=0,outer sep =0, minimum size = 1.5mm] {}
		child{edge from parent [decorate,decoration={snake,amplitude=.4mm,segment length=2mm,post length=.5mm,pre length=1.5mm}] {}}
		child{edge from parent [decorate,decoration={snake,amplitude=.4mm,segment length=2mm,post length=.5mm,pre length=1.5mm}] {}}
		child{edge from parent [decorate,decoration={snake,amplitude=.4mm,segment length=2mm,post length=.5mm,pre length=1.5mm}] {}}
		node [left] {$\scriptstyle V(3)$}}
	child{node [fill,circle, inner sep=0,outer sep =0, minimum size = 1.5mm] {}
		child{node [fill,circle, inner sep=0,outer sep =0, minimum size = 1.5mm] {} 
			child{{} node [above] (Al) {$\scriptstyle A$}}
			child{{} node [above] {$\scriptstyle A$}}
			child{{} node [above] (Am) {$\scriptstyle A$}}
			child{{} node [above] (Ar) {$\scriptstyle A$}}
			node [right] (V4) {$\scriptstyle V(4)$}}
		node [left] {$\scriptstyle V(1)$}}
	child{node [fill,circle, inner sep=0,outer sep =0, minimum size = 1.5mm] {}
		child{edge from parent [decorate,decoration={snake,amplitude=.4mm,segment length=2mm,post length=.5mm,pre length=1.5mm}] {}}
		child{node [fill,circle, inner sep=0,outer sep =0, minimum size = 1.5mm] {} node [above] {$\scriptstyle V(0)$}}				
		node [left] {$\scriptstyle V(2)$}}
	child{edge from parent [decorate,decoration={snake,amplitude=.4mm,segment length=2mm,post length=.5mm,pre length=1.5mm}] {}}
	node [below right] {$\scriptstyle V(4)$}};
\draw (8.4cm,0) node {} [grow'=up]
child{node [fill,circle, inner sep=0,outer sep =0, minimum size = 1.5mm] {} 
	child{node [fill,circle, inner sep=0,outer sep =0, minimum size = 1.5mm] {}
		child{edge from parent [decorate,decoration={snake,amplitude=.4mm,segment length=2mm,post length=.5mm,pre length=1.5mm}] {}}
		child{edge from parent [decorate,decoration={snake,amplitude=.4mm,segment length=2mm,post length=.5mm,pre length=1.5mm}] {}}
		child{edge from parent [decorate,decoration={snake,amplitude=.4mm,segment length=2mm,post length=.5mm,pre length=1.5mm}] {}}
		node [left] {$\scriptstyle V(3)$}}
	child{node [fill,circle, inner sep=0,outer sep =0, minimum size = 1.5mm] {}
		child{{} node [above] (A) {$\scriptstyle A$}}
		node [left] {$\scriptstyle V(1)$}}
	child{node [fill,circle, inner sep=0,outer sep =0, minimum size = 1.5mm] {}
		child{edge from parent [decorate,decoration={snake,amplitude=.4mm,segment length=2mm,post length=.5mm,pre length=1.5mm}] {}}
		child{node [fill,circle, inner sep=0,outer sep =0, minimum size = 1.5mm] {} node [above] {$\scriptstyle V(0)$}}				
		node [left] {$\scriptstyle V(2)$}}
	child{edge from parent [decorate,decoration={snake,amplitude=.4mm,segment length=2mm,post length=.5mm,pre length=1.5mm}] {}}
	node [below right] {$\scriptstyle V(4)$}};
\node[circle, fit=(V00), draw, opacity=.5, very thick, inner sep=-3] (E2) {};
\path[->] (E2) edge [opacity=.5, very thick, bend left=20] (Am.north);
	\node[rectangle, rounded corners, fit= (Al) (Ar) (V4), draw, opacity=.5, very thick, inner sep=-2] (E1) {};
	\path[opacity=.5, very thick, bend left = 20] (E1) edge [->] (A);
\end{tikzpicture}
\]
The second one maps via the identity the factor $(T_0;T_1,\dots T_t)$ to the factor of the tree obtained by grafting each $T_i$, $1\leq i\leq t$, into the $i^{\text{th}}$ straight leaf of $T_0$. The natural projection $\mathcal O_A^0(n)\rightarrow\mathcal O_A(n)$ is defined by iterated corolla contraction, until no corolla as in \eqref{newcc} remains, e.g. 
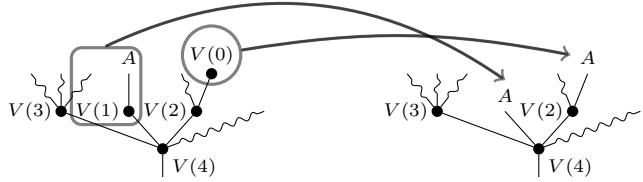
\[
\begin{tikzpicture}[level distance=5mm, sibling distance=9mm] 
\tikzstyle{level 3}=[sibling distance=4mm]
\node {} [grow'=up]
child{node [fill,circle, inner sep=0,outer sep =0, minimum size = 1.5mm] {} 
	child{node [fill,circle, inner sep=0,outer sep =0, minimum size = 1.5mm] {}
		child{edge from parent [decorate,decoration={snake,amplitude=.4mm,segment length=2mm,post length=.5mm,pre length=1.5mm}] {}}
		child{edge from parent [decorate,decoration={snake,amplitude=.4mm,segment length=2mm,post length=.5mm,pre length=1.5mm}] {}}
		child{edge from parent [decorate,decoration={snake,amplitude=.4mm,segment length=2mm,post length=.5mm,pre length=1.5mm}] {}}
		node [left] {$\scriptstyle V(3)$}}
	child{node [fill,circle, inner sep=0,outer sep =0, minimum size = 1.5mm] {}
		child{{} node [above] (A0) {$\scriptstyle A$}}
		node [left] (V1) {$\scriptstyle V(1)$}}
	child{node [fill,circle, inner sep=0,outer sep =0, minimum size = 1.5mm] {}
		child{edge from parent [decorate,decoration={snake,amplitude=.4mm,segment length=2mm,post length=.5mm,pre length=1.5mm}] {}}
		child{node [fill,circle, inner sep=0,outer sep =0, minimum size = 1.5mm] {} node [above] (V0) {$\scriptstyle V(0)$}}				
		node [left] {$\scriptstyle V(2)$}}
	child{edge from parent [decorate,decoration={snake,amplitude=.4mm,segment length=2mm,post length=.5mm,pre length=1.5mm}] {}}
	node [below right] {$\scriptstyle V(4)$}};
	\draw (5cm,0) node {} [grow'=up]
	child{node [fill,circle, inner sep=0,outer sep =0, minimum size = 1.5mm] {} 
		child{node [fill,circle, inner sep=0,outer sep =0, minimum size = 1.5mm] {}
			child{edge from parent [decorate,decoration={snake,amplitude=.4mm,segment length=2mm,post length=.5mm,pre length=1.5mm}] {}}
			child{edge from parent [decorate,decoration={snake,amplitude=.4mm,segment length=2mm,post length=.5mm,pre length=1.5mm}] {}}
			child{edge from parent [decorate,decoration={snake,amplitude=.4mm,segment length=2mm,post length=.5mm,pre length=1.5mm}] {}}
			node [left] {$\scriptstyle V(3)$}}
		child{{} node [above] (A1) {$\scriptstyle A$}}
		child{node [fill,circle, inner sep=0,outer sep =0, minimum size = 1.5mm] {}
			child{edge from parent [decorate,decoration={snake,amplitude=.4mm,segment length=2mm,post length=.5mm,pre length=1.5mm}] {}}
			child{{} node [above] (A2) {$\scriptstyle A$}}				
			node [left] {$\scriptstyle V(2)$}}
		child{edge from parent [decorate,decoration={snake,amplitude=.4mm,segment length=2mm,post length=.5mm,pre length=1.5mm}] {}}
		node [below right] {$\scriptstyle V(4)$}};
	\node[rectangle, rounded corners, fit= (A0) (V1), draw, opacity=.5, very thick, inner sep=-2] (E1) {};
	\path[opacity=.5, very thick, bend left = 30] (E1.north) edge [->] (A1.north);
	\node[circle, fit=(V0), draw, opacity=.5, very thick, inner sep=-2] (E2) {};
	\path[->] (E2) edge [opacity=.5, very thick, bend left = 10] (A2);
\end{tikzpicture}
\]
This is a split coequalizer. The arrows going backwards are defined as follows: $\mathcal O_A(n)\rightarrow\mathcal O_A^0(n)$ is a plain inclusion of coproduct factors, and $\mathcal O_A^0(n)\rightarrow\mathcal O_A^1(n)$ sends via the identity the factor corresponding to a tree $T$ with $n$ snaky leaves, to the factor $(T_0;T_1,\dots,T_t)$ obtained by pruning $T$ in such a way that the trees with straight leaves $T_1,\dots,T_t$ are as big as possible, e.g.~if $T$ is \eqref{ooo},
\[
\left(
\begin{array}{c}
\begin{tikzpicture}[level distance=5mm, sibling distance=8mm] 
\tikzstyle{level 3}=[sibling distance=3mm]
\node {} [grow'=up]
child{node [fill,circle, inner sep=0,outer sep =0, minimum size = 1.5mm] {} 
	child{node [fill,circle, inner sep=0,outer sep =0, minimum size = 1.5mm] {}
		child{edge from parent [decorate,decoration={snake,amplitude=.4mm,segment length=2mm,post length=.5mm,pre length=1.5mm}] {}}
		child{edge from parent [decorate,decoration={snake,amplitude=.4mm,segment length=2mm,post length=.5mm,pre length=1.5mm}] {}}
		child{edge from parent [decorate,decoration={snake,amplitude=.4mm,segment length=2mm,post length=.5mm,pre length=1.5mm}] {}}
		node [left] {$\scriptstyle V(3)$}}
	child{}
	child{node [fill,circle, inner sep=0,outer sep =0, minimum size = 1.5mm] {}
		child{edge from parent [decorate,decoration={snake,amplitude=.4mm,segment length=2mm,post length=.5mm,pre length=1.5mm}] {}}
		child{}				
		node [left] {$\scriptstyle V(2)$}}
	child{edge from parent [decorate,decoration={snake,amplitude=.4mm,segment length=2mm,post length=.5mm,pre length=1.5mm}] {}}
	node [below right] {$\scriptstyle V(4)$}};
\draw (1.5cm,.5cm) node {;};
\draw (2.7cm,0) node {} [grow'=up]
child{node [fill,circle, inner sep=0,outer sep =0, minimum size = 1.5mm] {} 
		child{{} node [above] {$\scriptstyle A$}}
		node [left] {$\scriptstyle V(1)$}};
\draw (3.5cm,.5cm) node {,};
\draw (4cm,0) node {} [grow'=up]
	child{node [fill,circle, inner sep=0,outer sep =0, minimum size = 1.5mm] {} node [above] {$\scriptstyle V(0)$}};	
\end{tikzpicture}
\end{array}
\right)
\]
\end{proof}

The following proposition can be similarly checked, using the explicit construction of the following push-out in $\operad{\C V}$ \cite[\S5]{htnso}, where the top map is a free operad map, rather than the structure of free operads,
\begin{equation}\label{pop}
\xymatrix{
	\mathcal F(U)\ar[r]^-{\mathcal F(f)}\ar[d]_g\ar@{}[rd]|{\text{push}}&\mathcal F(V)\ar[d]^{g'}\\
	\mathcal O\ar[r]_-{f'}&\mathcal P
}
\end{equation}
A $\mathcal P$-algebra $A$ is the same as an $\mathcal O$-algebra equipped with structure maps $z(V(n))\otimes A^{\otimes n}\r A$, that we regard as corolla contractions \eqref{newcc} additional to \eqref{treeops}, fitting in commutive squares, $n\geq 0$,
\begin{equation}\label{OPA}
\xymatrix@C=40pt{U(n)\otimes A^{\otimes n}\ar[r]^-{f(n)\otimes\id{}}\ar[d]_{\bar g(n)\otimes\id{}}&V(n)\otimes A^{\otimes n}\ar[d]\\
	\mathcal O(n)\otimes A^{\otimes n}\ar[r]&A
}
\end{equation}
Here $\bar g\colon U\r\mathcal O$ in $\C V^{\mathbb N}$ is the adjoint of $g$ in \eqref{pop}. 

Any map of operads $\phi\colon\mathcal O\r\mathcal P$ and $\mathcal P$-algebra $A$ give rise to a natural map of functor-operads $\phi_A\colon\mathcal O_A\r\mathcal P_A$ induced by the map of diagrams $\phi_A^i\colon\mathcal O^i_A\r\mathcal P^i_A$ defined by $\phi$ on inner vertices, $i=0,1$.

\begin{proposition}\label{horror}
	Given an operad push-out \eqref{pop} and a $\mathcal P$-algebra $A$, $f'_A\colon\mathcal O_A\r\mathcal P_A$ is the transfinite composition of a sequence 
	\[\mathcal O_{A}=\mathcal P_{A,0}\st{\Phi_{1}}\To \mathcal P_{A,1}\r\cdots\r\mathcal P_{A,t-1}\st{\Phi_{t}}\To \mathcal P_{A,t}\r\cdots\]
	in $\C C^{\C C^{(\mathbb N)}}$ such that $\Phi_{t}$, $t\geq 1$, is given by the push-out square
	$$\xymatrix@C=40pt{
		\bullet\ar[r]^{\tilde\Phi_{t}}\ar[d]_{\Psi_t}\ar@{}[rd]|{\text{push}}&\bullet\ar[d]^{\bar\Psi_t}\\
		\mathcal P_{A,t-1}\ar[r]_-{\Phi_{t}}&\mathcal P_{A,t}
	}$$
	where $\tilde\Phi_{t}(n)$ is the coproduct of all trees with $n$ leaves in even level, all of them snaky, no leaves in odd level, $t$ even inner vertices, labelled with $f$, odd inner vertices labelled with $\mathcal O_A$, and such that $T$ does not contain 
	\begin{equation}\label{corolla2}
	\begin{array}{c}
	\begin{tikzpicture}[level distance=5mm, sibling distance=5mm] 
	\draw (-4,0) node {} [grow'=up]
	child {node [fill,circle, inner sep=0,outer sep =0, minimum size = 1.5mm] (L) {} 
		child{[fill] circle (2pt) {} node [above] (A) {$\scriptstyle A$}}
		child{node [above] {$\cdots$} edge from parent[draw=none]}
		child{[fill] circle (2pt) {} node [above] (O) {$\scriptstyle A$}}
		node [below left] {$\scriptstyle f(n)$}};   
	\draw [decoration={brace}, decorate] ($(A)+(-.08,.15)$) -- node [above] {\scriptsize $n$ leaves} ($(O)+(.08,.15)$) ;
	\end{tikzpicture}
	\end{array}
	\end{equation}
	where $n\geq0$, e.g.~$(t=2, n=4)$
	
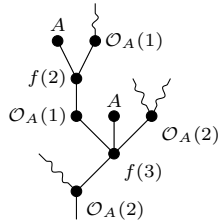
\[
	\begin{tikzpicture}[level distance=5mm, sibling distance=5mm] 
	\tikzstyle{level 2}=[sibling distance=10mm] 
	\tikzstyle{level 3}=[sibling distance=5mm] 
	\tikzstyle{level 4}=[sibling distance=5mm] 
	\node {} [grow'=up]
	child{[fill] circle (2pt)
		child{edge from parent [decorate,decoration={snake,amplitude=.4mm,segment length=2mm,post length=.5mm,pre length=1.5mm}] {}}
		child{[fill] circle (2pt)
			child{[fill] circle (2pt)
				child{[fill] circle (2pt)
					child{[fill] circle (2pt) node [above] {$\scriptstyle A$}}
					child{[fill] circle (2pt)
						child{edge from parent [decorate,decoration={snake,amplitude=.4mm,segment length=2mm,post length=.5mm,pre length=1.5mm}] {}}
						node [right] {$\scriptstyle \mathcal O_A(1)$}}
					node [left] {$\scriptstyle f(2)$}}
				node [left] {$\scriptstyle \mathcal O_A(1)$}}
			child{[fill] circle (2pt) node [above] {$\scriptstyle A$}}
			child{[fill] circle (2pt)
				child{edge from parent [decorate,decoration={snake,amplitude=.4mm,segment length=2mm,post length=.5mm,pre length=1.5mm}] {}}
				child{edge from parent [decorate,decoration={snake,amplitude=.4mm,segment length=2mm,post length=.5mm,pre length=1.5mm}] {}}
				node [below right] {$\scriptstyle \mathcal O_A(2)$}}
			node [below right] (U) {$\scriptstyle f(3)$}}
		node [below right] {$\scriptstyle \mathcal O_A(2)$}};
	\end{tikzpicture}
	\]
	The attaching map $\Psi_t$ is defined as the composition of three maps from the same trees, where now one even inner vertex is labelled with $U$ and the rest with $V$. The first map is induced by the composition of 
	$\bar g\colon U\r\mathcal O$ and the natural map $\mathcal O\r\mathcal O_A$, the second map is defined by inner edge contraction, and the third map is the restriction of $\bar{\Psi}_{t-1}$ if $t>1$ or the identity if $t=1$, e.g.
	\[
	\begin{tikzpicture}[level distance=5mm, sibling distance=5mm] 
	\tikzstyle{level 2}=[sibling distance=10mm] 
	\tikzstyle{level 3}=[sibling distance=5mm] 
	\tikzstyle{level 4}=[sibling distance=5mm] 
	\node {} [grow'=up]
	child{[fill] circle (2pt)
		child{  edge from parent [decorate,decoration={snake,amplitude=.4mm,segment length=2mm,post length=.5mm,pre length=1.5mm}] {}}
		child{[fill] circle (2pt)
			child{[fill] circle (2pt)
				child{[fill] circle (2pt)
					child{[fill] circle (2pt) node [above] {$\scriptstyle A$}}
					child{[fill] circle (2pt)
						child{  edge from parent [decorate,decoration={snake,amplitude=.4mm,segment length=2mm,post length=.5mm,pre length=1.5mm}] {}}
						node [right] {$\scriptstyle \mathcal O_A(1)$}}
					node [left] {$\scriptstyle V(2)$}}
				node [left] {$\scriptstyle \mathcal O_A(1)$}}
			child{[fill] circle (2pt) node [above] {$\scriptstyle A$}}
			child{[fill] circle (2pt)
				child{  edge from parent [decorate,decoration={snake,amplitude=.4mm,segment length=2mm,post length=.5mm,pre length=1.5mm}] {}}
				child{  edge from parent [decorate,decoration={snake,amplitude=.4mm,segment length=2mm,post length=.5mm,pre length=1.5mm}] {}}
				node [below right] {$\scriptstyle \mathcal O_A(2)$}}
			node [below right] (U) {$\scriptstyle U(3)$}}
		node [below right] {$\scriptstyle \mathcal O_A(2)$}};
	\draw (3.5,0) node {} [grow'=up]
	child{[fill] circle (2pt)
		child{  edge from parent [decorate,decoration={snake,amplitude=.4mm,segment length=2mm,post length=.5mm,pre length=1.5mm}] {}}
		child{[fill] circle (2pt)
			child{[fill] circle (2pt)
				child{[fill] circle (2pt)
					child{[fill] circle (2pt) node [above] {$\scriptstyle A$}}
					child{[fill] circle (2pt)
						child{  edge from parent [decorate,decoration={snake,amplitude=.4mm,segment length=2mm,post length=.5mm,pre length=1.5mm}] {}}
						node [right] {$\scriptstyle \mathcal O_A(1)$}}
					node [left] {$\scriptstyle V(2)$}}
				node [left] (O1) {$\scriptstyle \mathcal O_A(1)$}}
			child{[fill] circle (2pt) node [above] (A) {$\scriptstyle A$}}
			child{[fill] circle (2pt)
				child{  edge from parent [decorate,decoration={snake,amplitude=.4mm,segment length=2mm,post length=.5mm,pre length=1.5mm}] {}}
				child{  edge from parent [decorate,decoration={snake,amplitude=.4mm,segment length=2mm,post length=.5mm,pre length=1.5mm}] {}}
				node [below right] (O2) {$\scriptstyle \mathcal O_A(2)$}}
			node [below right] (U2) {$\scriptstyle \mathcal O_A(3)$}}
		node [below right] (root) {$\scriptstyle \mathcal O_A(2)$}};
	\node[rectangle, rounded corners, fit= (root) (O1) (A) (O2), draw, opacity=.5, very thick, inner sep=-2] (E) {};
	\path[->] (U) edge [opacity=.5, very thick, bend right = 10] (U2);
	\tikzstyle{level 2}=[sibling distance=9mm] 
		\tikzstyle{level 2}=[sibling distance=8mm] 
	\draw (8,0) node {} [grow'=up]
	child{[fill] circle (2pt)
		child{  edge from parent [decorate,decoration={snake,amplitude=.4mm,segment length=2mm,post length=.5mm,pre length=1.5mm}] {}}
		child{[fill] circle (2pt)
			child{[fill] circle (2pt) node [above] {$\scriptstyle A$}}
			child{[fill] circle (2pt)
				child{  edge from parent [decorate,decoration={snake,amplitude=.4mm,segment length=2mm,post length=.5mm,pre length=1.5mm}] {}}
				node [right] {$\scriptstyle \mathcal O_A(1)$}}
			node [right] {$\scriptstyle V(2)$}}
		child{  edge from parent [decorate,decoration={snake,amplitude=.4mm,segment length=2mm,post length=.5mm,pre length=1.5mm}] {}}
		child{  edge from parent [decorate,decoration={snake,amplitude=.4mm,segment length=2mm,post length=.5mm,pre length=1.5mm}] {}}
		node [below left] (root2) {$\scriptstyle \mathcal O_A(4)$}};
	\path[->] (E) edge [opacity=.5, very thick, bend left = 10] (root2);
	\draw (11,.5) node (P) {$\mathcal P_{A,1}(4)$};
	\path[->] (9,.5) edge [opacity=.5, very thick] node [above, opacity=1] {\scriptsize $\bar{\Psi}_1$} (P);
	\end{tikzpicture}
	\]
\end{proposition}

The only remarkable cosmetic difference between Propositions \ref{free} and \ref{horror} is that, in the former, we allow straight leaves but not corks (corks are corollas \eqref{newcc} for $n=0$) and, in the latter, we allow corks in odd levels (not in even  levels, since \eqref{corolla2} is an even  cork for $n=0$) but not straight leaves at all. This is just to simplify notation. We do not describe the functor-operad structure of $\mathcal P_{A}$ in Proposition \ref{horror}. It is given by grafting, then contracting the new inner edge and, if they appear, contracting corollas \eqref{corolla2}
%\[
%\begin{array}{c}
%	\begin{tikzpicture}[level distance=5mm, sibling distance=5mm] 
%	\draw (-4,0) node {} [grow'=up]
%	child {node [fill,circle, inner sep=0,outer sep =0, minimum size = 1.5mm] (L) {} 
%		child{[fill] circle (2pt) {} node [above] (A) {$\scriptstyle A$}}
%		child{node [above] {$\cdots$} edge from parent[draw=none]}
%		child{[fill] circle (2pt) {} node [above] (O) {$\scriptstyle A$}}
%		node [below left] {$\scriptstyle V(n)$}};   
%	\draw [decoration={brace}, decorate] ($(A)+(-.08,.15)$) -- node [above] {\scriptsize $n$ leaves} ($(O)+(.08,.15)$) ;
%	\end{tikzpicture}
%	\end{array}
%\]
by using the $\mathcal P$-algebra structure maps $z(V(n))\otimes A^{\otimes n}\r A$. To the best of our knowledge, Proposition \ref{horror} is new even in the symmetric case. In that case, enveloping functor-operads can be replaced with honest enveloping operads in the obvious way.

An $\mathcal O$-algebra map $\phi\colon A\r B$ induces a 
functor-operad map $\mathcal O_{\phi}\colon \mathcal O_{A}\r \mathcal O_{B}$ 
which is the coequalizer of maps $\mathcal O_{\phi}^i\colon \mathcal O_{A}^i\r \mathcal O_{B}^i$, $i=0,1$, induced by $\phi$ on straight leaves.

\begin{proposition}\label{transAcof}
If we have an $\mathcal O$-algebra push-out  \eqref{algebrapo}, $\mathcal O_{f'}$ is the transfinite composition of a sequence 
\[\mathcal O_{A}=\mathcal O_{B,0}\st{\Phi_{1}}\To \mathcal O_{B,1}\r\cdots\r\mathcal O_{B,t-1}\st{\Phi_{t}}\To \mathcal O_{B,t}\r\cdots\]
in $\C C^{\C C^{(\mathbb N)}}$ such that $\Phi_{t}$, $t\geq 1$, is given by the push-out square
	$$\xymatrix@C=40pt{
		\bullet\ar[r]^{\tilde\Phi_{t}}\ar[d]_{\Psi_t}\ar@{}[rd]|{\text{push}}&\bullet\ar[d]^{\bar\Psi_t}\\
		\mathcal O_{B,t-1}\ar[r]_-{\Phi_{t}}&\mathcal O_{B,t}
		}$$
where $\tilde\Phi_{t}(n)$ is the coproduct of all corollas with $t+n$ leaves, $t$ bumpy and $n$ snaky, inner vertex labelled with $\mathcal O_{A}$, and bumpy leaves with $f$, e.g.~$(t=3, n=2)$
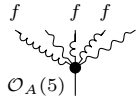
\[
\begin{tikzpicture}[level distance=5mm, sibling distance=4mm] 
\node {} [grow'=up]
child {[fill] circle (2pt) 
  child {{} node [above] {$\scriptstyle f$} edge from parent [decorate,decoration={bumps,amplitude=2,segment length=2mm,post length=.5mm,pre length=1.5mm}] {}}
  child {{} %node [above] {$\scriptstyle X_1$} 
  edge from parent [decorate,decoration={snake,amplitude=.4mm,segment length=2mm,post length=.5mm,pre length=1.5mm}] {}}
  child {{} node [above] {$\scriptstyle f$} edge from parent [decorate,decoration={bumps,amplitude=2,segment length=2mm,post length=.5mm,pre length=1.5mm}] {}}
  child {{} node [above] {$\scriptstyle f$} edge from parent [decorate,decoration={bumps,amplitude=2,segment length=2mm,post length=.5mm,pre length=1.5mm}] {}}
  child {{} %node [above] {$\scriptstyle X_2$} 
  edge from parent [decorate,decoration={snake,amplitude=.4mm,segment length=2mm,post length=.5mm,pre length=1.5mm}] {}}
node [below left] {$\scriptstyle \mathcal O_A(5)$}
};
\end{tikzpicture}
\]
The attaching map $\Psi_{t}$ is defined as the composition of three maps  from the same corollas, where  now one bumpy leaf is labelled with $Y$ and the rest with $Z$. The first map is given by $\bar g \colon Y\r A=\mathcal O_{A}(0)$ (which straightens and corks one bumpy leaf), the second map is defined by inner edge contraction, and the third map is the restriction of $\bar\Psi_{t-1}$ if $t>1$ and the identity if $t=1$, e.g.
\[
\begin{tikzpicture}[level distance=5mm, sibling distance=4mm] 
\node {} [grow'=up]
child {[fill] circle (2pt) 
  child {{} node [above] {$\scriptstyle Z$} edge from parent [decorate,decoration={bumps,amplitude=2,segment length=2mm,post length=.5mm,pre length=1.5mm}] {}}
  child {{} %node [above] {$\scriptstyle X_1$} 
edge from parent [decorate,decoration={snake,amplitude=.4mm,segment length=2mm,post length=.5mm,pre length=1.5mm}] {}}
  child {{} node [above] (Y) {$\scriptstyle Y$} edge from parent [decorate,decoration={bumps,amplitude=2,segment length=2mm,post length=.5mm,pre length=1.5mm}] {}}
  child {{} node [above] {$\scriptstyle Z$} edge from parent [decorate,decoration={bumps,amplitude=2,segment length=2mm,post length=.5mm,pre length=1.5mm}] {}}
  child {{} %node [above] {$\scriptstyle X_2$} 
edge from parent [decorate,decoration={snake,amplitude=.4mm,segment length=2mm,post length=.5mm,pre length=1.5mm}] {}}
node [below left] {$\scriptstyle \mathcal O_A(5)$}
};
\draw (3,0) node {} [grow'=up]
child {[fill] circle (2pt) 
  child {{} node [above] {$\scriptstyle Z$} edge from parent [decorate,decoration={bumps,amplitude=2,segment length=2mm,post length=.5mm,pre length=1.5mm}] {}}
  child {{} %node [above] {$\scriptstyle X_1$} 
edge from parent [decorate,decoration={snake,amplitude=.4mm,segment length=2mm,post length=.5mm,pre length=1.5mm}] {}}
  child {[fill] circle (2pt)  {} node [above] (A) {$\scriptstyle A$}}
  child {{} node [above] {$\scriptstyle Z$} edge from parent [decorate,decoration={bumps,amplitude=2,segment length=2mm,post length=.5mm,pre length=1.5mm}] {}}
  child {{} %node [above] {$\scriptstyle X_2$} 
edge from parent [decorate,decoration={snake,amplitude=.4mm,segment length=2mm,post length=.5mm,pre length=1.5mm}] {}}
node [below left] (OA5) {$\scriptstyle \mathcal O_A(5)$}
};
\draw (6,0) node {} [grow'=up]
child {[fill] circle (2pt) 
  child {{} node [above] {$\scriptstyle Z$} edge from parent [decorate,decoration={bumps,amplitude=2,segment length=2mm,post length=.5mm,pre length=1.5mm}] {}}
  child {{} %node [above] {$\scriptstyle X_1$} 
edge from parent [decorate,decoration={snake,amplitude=.4mm,segment length=2mm,post length=.5mm,pre length=1.5mm}] {}}
  child {{} node [above] {$\scriptstyle Z$} edge from parent [decorate,decoration={bumps,amplitude=2,segment length=2mm,post length=.5mm,pre length=1.5mm}] {}}
  child {{} %node [above] {$\scriptstyle X_2$} 
edge from parent [decorate,decoration={snake,amplitude=.4mm,segment length=2mm,post length=.5mm,pre length=1.5mm}] {}}
node [below left] (OA4) {$\scriptstyle \mathcal O_A(4)$}
};
\draw (8,.5) node (OB32) {$\mathcal O_{B,2}(2)$};
\node[circle, fit=(Y), draw, opacity=.5, very thick, inner sep=-2] (EY) {};
\path[opacity=.5, very thick, bend left = 20] (EY.north east) edge [->] node [above, opacity=1, text width=1.9cm, text centered] {$\scriptstyle \bar g$} (A);
\node[circle, fill, fit=(A), draw, opacity=.15, very thin, inner sep=-2.5] (EA) {};
\node[circle, fill, fit=(OA5), draw, opacity=.15, very thin, inner sep=-4] (EOA5) {};
\path[-] (EA) edge [opacity=.15, ultra thick] (EOA5);
\path[->] (EOA5.south east) edge [opacity=.15, ultra thick, bend right =10] node [above, opacity=1] {$\scriptstyle \circ_{3}$} (OA4);
\path[->] (6.3,.5) edge [opacity=.5, very thick] node [below, opacity=1] {$\scriptstyle \bar\Psi_2(2)$} (OB32);
\end{tikzpicture}
\]
\end{proposition}

\begin{proof}
By \cite[\S8]{htnso} (corrected version), we can define $\mathcal O_{B,t}$ as the coequalizer of the maps $\mathcal O_{B,t}^1\rightrightarrows \mathcal O_{B,t}^0$ defined as follows. The functor $\mathcal O_{B,t}^i(n)$ is a coproduct 
with the same indexing set as $\mathcal O_{A}^i(n)$. The factor corresponding to a tree with $n$ snaky leaves is the colimit of the functors obtained by labeling the inner vertices with $\mathcal O$ and the straight leaves with $B_{s_1},\dots, B_{s_r}$, $s_{1}+\cdots+s_{r}\leq t$, e.g.~$(i=0, n=2, r=3)$
$$
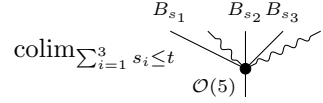
\begin{tikzpicture}[level distance=5mm, sibling distance=5mm] 
\node {} [grow'=up]
child {[fill] circle (2pt) 
  child {{} node [above] {$\scriptstyle B_{s_1}$}}
  child {{} %node [above] {$\scriptstyle X_1$} 
  edge from parent [decorate,decoration={snake,amplitude=.4mm,segment length=2mm,post length=.5mm,pre length=1.5mm}] {}}
  child {{} node [above] {$\scriptstyle B_{s_2}$}}
  child {{} node [above] {$\scriptstyle B_{s_3}$}}
  child {{} %node [above] {$\scriptstyle X_2$} 
  edge from parent [decorate,decoration={snake,amplitude=.4mm,segment length=2mm,post length=.5mm,pre length=1.5mm}] {}}
node [below left] {$\scriptstyle \mathcal O(5)$}
};
\draw (-2,.7) node {$\colim_{\sum_{i=1}^3s_i\leq t}$};
\end{tikzpicture}$$
The maps $\mathcal O_{B,t}^1\rightrightarrows \mathcal O_{B,t}^0$ are respectively given by inner edge contraction, as above, and by the following new kind of corolla contraction, % using the maps in \cite[Lemma 8.2]{htnso}, 
$$
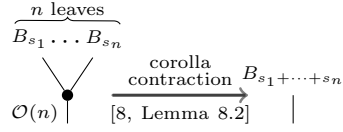
\begin{tikzpicture}[level distance=5mm, sibling distance=5mm] 
\node {} [grow'=up]
child {node [fill,circle, inner sep=0,outer sep =0, minimum size = 1.5mm] (L) {} 
  child{node [above] (A) {$\scriptstyle B_{s_1}$}}
  child{node [above] {$\cdots$} edge from parent[draw=none]}
  child{node [above] (O) {$\scriptstyle B_{s_n}$}}
node [below left] {$\scriptstyle \mathcal O(n)$}};   
\draw [decoration={brace}, decorate] ($(A)+(-.2,.2)$) -- node [above] {\scriptsize $n$ leaves} ($(O)+(.2,.2)$) ;
\draw (3,0) node {} [grow'=up] child{coordinate (R) node [above] {$\scriptstyle B_{s_{1}+\cdots+s_{n}}$}};
\path[->] ($(L)+(6mm,0)$) edge  [very thick, opacity=.5] node [above, opacity=1, text width=2.4cm, text centered] {\scriptsize corolla \\[-2mm] contraction} node [below, opacity=1, text width=2.4cm, text centered] {\scriptsize \cite[Lemma 8.2]{htnso}} ($(R)+(-6mm,0)$);
\end{tikzpicture}$$
The map $\Phi_{t}$ is the coequalizer of two maps $\Phi_{t}^1\rightrightarrows \Phi_{t}^0$ in $\mor{\C C}^{\C C^{(\mathbb N)}}$, 
\begin{equation}\label{son2}
\xymatrix{
\mathcal O_{B,t-1}^1\ar[r]^-{\Phi_{t}^{1}}\ar@<-.5ex>[d]\ar@<.5ex>[d]&\mathcal O_{B,t}^1\ar@<-.5ex>[d]\ar@<.5ex>[d]\\
\mathcal O_{B,t-1}^0\ar[r]^-{\Phi_{t}^{0}}&\mathcal O_{B,t}^0
}
\end{equation}
where $\Phi_{t}^{i}$ is defined by the bonding maps in \cite[Lemma \ref{pind2}]{htnso} (corrected version).

The map $\Phi^{i}_{t}$ in $\C C^{\C C^{(\mathbb N)}}$ fits into a push-out diagram
\begin{equation}\label{fii}
	\xymatrix{
		\bullet\ar[r]^-{\bar{\Phi}_t^i}\ar[d]\ar@{}[rd]|{\text{push}}&\bullet\ar[d]\\
		\mathcal O_{B,t-1}^i\ar[r]^-{\Phi_t^i}&\mathcal O_{B,t}^i
	}
\end{equation}
where $\bar{\Phi}_t^i$ is the coequalizer of the pair $\bar \Phi_t^{i1}\rightrightarrows \bar \Phi_t^{i0}$ in $\mor{\C C}^{\C C^{(\mathbb N)}}$ defined as follows. For $i=0$,  $\bar\Phi_t^{00}(n)$ is a coproduct of trees of height $3$ with $n$ leaves in level $2$, all of them snaky, and such that all level $3$ edges are leaves, including $t$ bumpy leaves (the rest straight). Inner edges are labeled with $\mathcal O$, straight leaves with $A$, and bumpy leaves with $f$, e.g.~$(n=2,t=2)$
$$
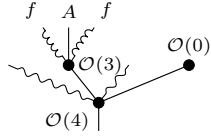
\begin{tikzpicture}[level distance=5mm, sibling distance=5mm] 
\tikzstyle{level 1}=[sibling distance=8mm] 
\tikzstyle{level 3}=[sibling distance=5mm] 
\node {} [grow'=up]
child {[fill] circle (2pt) 
  child {{} %node [above] {$\scriptstyle X_1$} 
  edge from parent [decorate,decoration={snake,amplitude=.4mm,segment length=2mm,post length=.5mm,pre length=1.5mm}] {}}
  child {[fill] circle (2pt) 
      child{node [above] {$\scriptstyle f$} edge from parent [decorate,decoration={bumps,amplitude=2,segment length=2mm,post length=.5mm,pre length=1.5mm}] {}}
      child{{} node [above] {$\scriptstyle A$}}
      child{node [above] {$\scriptstyle f$} edge from parent [decorate,decoration={bumps,amplitude=2,segment length=2mm,post length=.5mm,pre length=1.5mm}] {}}
    node [right] (O5) {$\scriptstyle \mathcal O(3)$}
  }
  child {{} %node [above] {$\scriptstyle X_2$} 
  edge from parent [decorate,decoration={snake,amplitude=.4mm,segment length=2mm,post length=.5mm,pre length=1.5mm}] {}}
  child {[fill] circle (2pt) node [above] (O0) {$\scriptstyle \mathcal O(0)$}}
node [below left] {$\scriptstyle \mathcal O(4)$}
};
\end{tikzpicture}
$$
Similarly, for $i=1$, $\bar \Phi_t^{01}(n)$ is a coproduct of trees of height $\leq 4$ with $n$ level $2$ leaves (snaky), $t$ level $3$ leaves (bumpy), and such that all level $4$ edges (if any) are straight leaves. Labels are as above, and the maps   $\bar \Phi_t^{01}\rightrightarrows \bar \Phi_t^{00}$ are defined by corolla contraction and level $3$ inner edge contraction, e.g.~$(n=2,t=2)$
$$
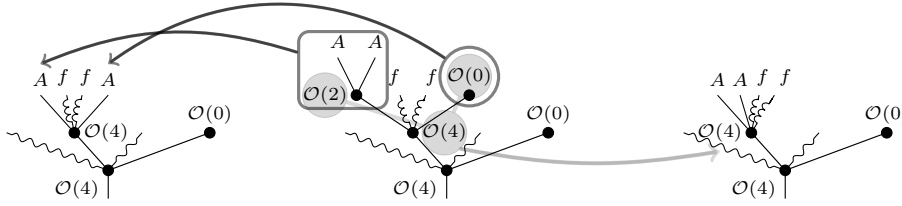
\begin{tikzpicture}[level distance=5mm, sibling distance=5mm] 
\tikzstyle{level 1}=[sibling distance=9mm] 
\tikzstyle{level 3}=[sibling distance=3mm] 
\draw (-4.5,0) node {} [grow'=up]
child {[fill] circle (2pt) 
  child {edge from parent [decorate,decoration={snake,amplitude=.4mm,segment length=2mm,post length=.5mm,pre length=1.5mm}] {}}
  child {[fill] circle (2pt) 
      child{{} node [above] (Al) {$\scriptstyle A$}}
      child{node [above] {$\scriptstyle f$} edge from parent [decorate,decoration={bumps,amplitude=2,segment length=2mm,post length=.5mm,pre length=1.5mm}] {}}
      child{node [above] {$\scriptstyle f$} edge from parent [decorate,decoration={bumps,amplitude=2,segment length=2mm,post length=.5mm,pre length=1.5mm}] {}}
      child{{} node [above] (Ar) {$\scriptstyle A$}}
    node [right] {$\scriptstyle \mathcal O(4)$}
  }
  child {edge from parent [decorate,decoration={snake,amplitude=.4mm,segment length=2mm,post length=.5mm,pre length=1.5mm}] {}}
  child {[fill] circle (2pt) node [above] {$\scriptstyle \mathcal O(0)$}}
node [below left] {$\scriptstyle \mathcal O(4)$}
};
\tikzstyle{level 3}=[sibling distance=5mm] 
\node {} [grow'=up]
child {[fill] circle (2pt) 
  child {edge from parent [decorate,decoration={snake,amplitude=.4mm,segment length=2mm,post length=.5mm,pre length=1.5mm}] {}}
  child {[fill] circle (2pt) 
      child{[fill] circle (2pt) 
	child{{} node [above] (AA) {$\scriptstyle A$}}
	child{{} node [above] (AAA) {$\scriptstyle A$}}
      node [left] (O2) {$\scriptstyle \mathcal O(2)$}} 
      child{node [above] {$\scriptstyle f$} edge from parent [decorate,decoration={bumps,amplitude=2,segment length=2mm,post length=.5mm,pre length=1.5mm}] {}}
      child{node [above] {$\scriptstyle f$} edge from parent [decorate,decoration={bumps,amplitude=2,segment length=2mm,post length=.5mm,pre length=1.5mm}] {}}
      child {[fill] circle (2pt) node [above] (O0t) {$\scriptstyle \mathcal O(0)$}}
    node [right] (O4t) {$\scriptstyle \mathcal O(4)$}
  }
  child {edge from parent [decorate,decoration={snake,amplitude=.4mm,segment length=2mm,post length=.5mm,pre length=1.5mm}] {}}
  child {[fill] circle (2pt) node [above] (O0) {$\scriptstyle \mathcal O(0)$}}
node [below left] {$\scriptstyle \mathcal O(4)$}
};
\tikzstyle{level 3}=[sibling distance=3mm] 
\draw (4.5,0) node {} [grow'=up]
child {[fill] circle (2pt) 
  child {edge from parent [decorate,decoration={snake,amplitude=.4mm,segment length=2mm,post length=.5mm,pre length=1.5mm}] {}}
  child {[fill] circle (2pt) 
  	child{{} node [above] {$\scriptstyle A$}}
	child{{} node [above] {$\scriptstyle A$}}
        child{node [above] {$\scriptstyle f$} edge from parent [decorate,decoration={bumps,amplitude=2,segment length=2mm,post length=.5mm,pre length=1.5mm}] {}}
      child{node [above] {$\scriptstyle f$} edge from parent [decorate,decoration={bumps,amplitude=2,segment length=2mm,post length=.5mm,pre length=1.5mm}] {}}
      node [left] (O4r) {$\scriptstyle \mathcal O(4)$}
  }
  child {edge from parent [decorate,decoration={snake,amplitude=.4mm,segment length=2mm,post length=.5mm,pre length=1.5mm}] {}}
  child {[fill] circle (2pt) node [above] {$\scriptstyle \mathcal O(0)$}}
node [below left] {$\scriptstyle \mathcal O(4)$}
};
\node[circle, fill, fit=(O4t), draw, opacity=.15, very thin, inner sep=-4] (OO4t) {};
\node[circle, fill, fit=(O2), draw, opacity=.15, very thin, inner sep=-4] (OO2) {};
\node[circle, fill, fit=(O0t), draw, opacity=.15, very thin, inner sep=-4] (OO0t) {};
\path[-] (OO2) edge [opacity=.15, ultra thick] (OO4t);
\path[-] (OO0t) edge [opacity=.15, ultra thick] (OO4t);
\path[->] (OO4t.south east) edge [opacity=.15, ultra thick, bend right =10] (O4r.south);
\node[rectangle, rounded corners, fit= (O2) (AA) (AAA), draw, opacity=.5, very thick, inner xsep=1.5, inner sep=-2] (E1) {};
\path[opacity=.5, very thick, bend right = 20] (E1) edge [->] (Al.north);
\node[circle, fit= (O0t), draw, opacity=.5, very thick, inner sep=-2] (E2) {};
\path[opacity=.5, very thick, bend right = 35] (E2) edge [->] (Ar.north);
\end{tikzpicture}
$$
Analogously, $\bar\Phi^{10}_r(n)$ is the coproduct of height $4$ trees with $n$ level $2$ leaves (snaky), no level $3$ leaves, and such that all level $4$ edges are leaves ($t$ of them bumpy), $\bar\Phi^{11}_r(n)$ is the coproduct of trees of height $\leq 5$ with $n$ level $2$ leaves (snaky), no level $3$ leaves, $t$ level $4$ leaves (bumpy), and such that all level $5$ edges (if any) are straight leaves. Labels are again as above, and the maps $\bar \Phi_t^{11}\rightrightarrows \bar \Phi_t^{10}$ are defined by corolla contraction and level $4$ inner edge contraction, respectively. The attaching maps in \eqref{fii} are derived from the attaching maps in \cite[Lemma \ref{pind2}]{htnso} (corrected version).

There are parallel maps $\bar{\Phi}_t^1\rightrightarrows\bar{\Phi}_t^0$ in $\mor{\C C}^{\C C^{(\mathbb N)}}$ compatible with $\Phi_{t}^1\rightrightarrows \Phi_{t}^0$ in \eqref{son2}  via the attaching and characteristic maps in \eqref{fii},
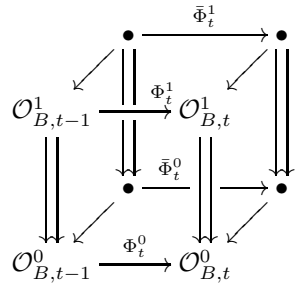
\[
\xymatrix@!=5pt{
	&\bullet\ar[rr]^{\bar \Phi_t^1}
	\ar@<-.5ex>[dd]|-{\hole}\ar@<.5ex>[dd]|-{\hole}\ar[ld]
	&&\bullet\ar@<-.5ex>[dd]\ar@<.5ex>[dd]\ar[ld]\\
	\mathcal O_{B,t-1}^1\ar[rr]^<(.6){\Phi_{t}^{1}}\ar@<-.5ex>[dd]\ar@<.5ex>[dd]&&\mathcal O_{B,t}^1\ar@<-.5ex>[dd]\ar@<.5ex>[dd]\\
		&\bullet\ar[rr]^<(.2){\bar \Phi_t^0}|-{\hole\hole}\ar[ld]&&\bullet\ar[ld]\\
	\mathcal O_{B,t-1}^0\ar[rr]^-{\Phi_{t}^{0}}&&\mathcal O_{B,t}^0
}
\]
The maps $\bar{\Phi}_t^1\rightrightarrows\bar{\Phi}_t^0$ are obtained by taking horizontal coequalizers in the following diagram,
\begin{equation}\label{cuadrado}
\xymatrix{\bar \Phi_t^{11}\ar@<-.5ex>[r]\ar@<.5ex>[r]\ar@<-.5ex>[d]\ar@<.5ex>[d]&\bar \Phi_t^{10}\ar@<-.5ex>[d]\ar@<.5ex>[d]\\
	\bar \Phi_t^{01}\ar@<-.5ex>[r]\ar@<.5ex>[r]&\bar \Phi_t^{00}}
\end{equation}
Here, the right (resp.~left) vertical arrow in each vertical pair is defined by contraction of level $2$ (resp.~$3$) inner edges, e.g.~the vertical pair on the right,
$$
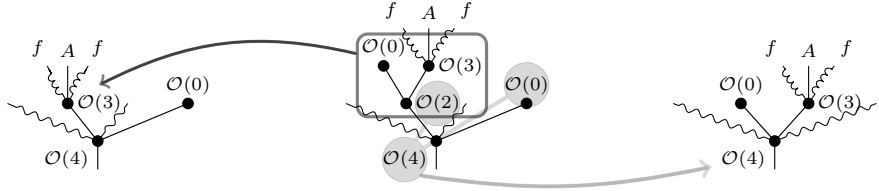
\begin{tikzpicture}[level distance=5mm, sibling distance=7mm] 
\tikzstyle{level 1}=[sibling distance=8mm] 
\tikzstyle{level 3}=[sibling distance=4mm] 
\node {} [grow'=up]
child {[fill] circle (2pt) 
  child {edge from parent [decorate,decoration={snake,amplitude=.4mm,segment length=2mm,post length=.5mm,pre length=1.5mm}] {}}
  child {[fill] circle (2pt) 
      child{node [above] {$\scriptstyle f$} edge from parent [decorate,decoration={bumps,amplitude=2,segment length=2mm,post length=.5mm,pre length=1.5mm}] {}}
      child{{} node [above] {$\scriptstyle A$}}
      child{node [above] {$\scriptstyle f$} edge from parent [decorate,decoration={bumps,amplitude=2,segment length=2mm,post length=.5mm,pre length=1.5mm}] {}}
    node [right] (O5) {$\scriptstyle \mathcal O(3)$}
  }
  child {edge from parent [decorate,decoration={snake,amplitude=.4mm,segment length=2mm,post length=.5mm,pre length=1.5mm}] {}}
  child {[fill] circle (2pt) node [above] (O0) {$\scriptstyle \mathcal O(0)$}}
node [below left] {$\scriptstyle \mathcal O(4)$}
};   
\tikzstyle{level 1}=[sibling distance=8mm] 
\tikzstyle{level 3}=[sibling distance=6mm] 
\tikzstyle{level 4}=[sibling distance=5mm] 
\draw (4.5,0) node {} [grow'=up]
child {[fill] circle (2pt) 
  child {edge from parent [decorate,decoration={snake,amplitude=.4mm,segment length=2mm,post length=.5mm,pre length=1.5mm}] {}}
  child {[fill] circle (2pt) 
    child {[fill] circle (2pt) node [above] (OOO) {$\scriptstyle \mathcal O(0)$}}
    child {
      child{node [above] {$\scriptstyle f$} edge from parent [decorate,decoration={bumps,amplitude=2,segment length=2mm,post length=.5mm,pre length=1.5mm}] {}}
      child{{} node [above] {$\scriptstyle A$}}
      child{node [above] {$\scriptstyle f$} edge from parent [decorate,decoration={bumps,amplitude=2,segment length=2mm,post length=.5mm,pre length=1.5mm}] {}}
    [fill] circle (2pt) node [right] (O22) {$\scriptstyle \mathcal O(3)$}}
    node [right] (O2) {$\scriptstyle \mathcal O(2)$}
  }
  child {edge from parent [decorate,decoration={snake,amplitude=.4mm,segment length=2mm,post length=.5mm,pre length=1.5mm}] {}}
  child {[fill] circle (2pt) node [above] (O02) {$\scriptstyle \mathcal O(0)$}}
node [below left] (O3) {$\scriptstyle \mathcal O(4)$}
};   
\node[rectangle, rounded corners, fit= (O2) (O22) (OOO), draw, opacity=.5, very thick, inner xsep=1.5, inner sep=-2] (E1) {};
\path[opacity=.5, very thick, bend right = 20] (E1) edge [->] (O5.north);
%\node[circle, fit= (O02), draw, opacity=.5, very thick, inner sep=-2] (E2) {};
%\path[opacity=.5, very thick, bend right = 35] (E2) edge [->] (O0.north);
\tikzstyle{level 1}=[sibling distance=13mm]
\tikzstyle{level 2}=[sibling distance=9mm] 
\tikzstyle{level 3}=[sibling distance=5mm] 
\draw (9,0) node {} [grow'=up]
child {[fill] circle (2pt) 
  child { edge from parent [decorate,decoration={snake,amplitude=.4mm,segment length=2mm,post length=.5mm,pre length=1.5mm}] {}}
    child {[fill] circle (2pt) node [above] {$\scriptstyle \mathcal O(0)$}}
    child {
      child{node [above] {$\scriptstyle f$} edge from parent [decorate,decoration={bumps,amplitude=2,segment length=2mm,post length=.5mm,pre length=1.5mm}] {}}
      child{{} node [above] {$\scriptstyle A$}}
      child{node [above] {$\scriptstyle f$} edge from parent [decorate,decoration={bumps,amplitude=2,segment length=2mm,post length=.5mm,pre length=1.5mm}] {}}
    [fill] circle (2pt) node [right] {$\scriptstyle \mathcal O(3)$}}
  child {edge from parent [decorate,decoration={snake,amplitude=.4mm,segment length=2mm,post length=.5mm,pre length=1.5mm}] {}}
node [below left] (O4) {$\scriptstyle \mathcal O(4)$}
}; 
\node[circle, fill, fit=(O3), draw, opacity=.15, very thin, inner sep=-4] (OO3) {};
\node[circle, fill, fit=(O2), draw, opacity=.15, very thin, inner sep=-4] (OO2) {};
\node[circle, fill, fit=(O02), draw, opacity=.15, very thin, inner sep=-4] (OO02) {};
\path[-] (OO2) edge [opacity=.15, ultra thick] (OO3);
\path[-] (OO02) edge [opacity=.15, ultra thick] (OO3);
\path[->] (OO3.south east) edge [opacity=.15, ultra thick, bend right =10] (O4);
\end{tikzpicture}
$$
Hence $\Phi_{t}$ is the push-out of $\bar{\Phi}_t$, the coequalizer of $\bar{\Phi}_t^1\rightrightarrows\bar{\Phi}_t^0$, 
\begin{equation*}
\xymatrix{
	\bullet\ar[r]^-{\bar{\Phi}_t}\ar[d]\ar@{}[rd]|{\text{push}}&\bullet\ar[d]\\
	\mathcal O_{B,t-1}\ar[r]^-{\Phi_t}&\mathcal O_{B,t}
}
\end{equation*}

The map $\bar{\Phi}_t$ can also be obtained by taking first vertical and then horizontal coequalizers in \eqref{cuadrado}. Let $\tilde{\Phi}_t^j$ be the coequalizer of $\bar\Phi_t^{1j}\rightrightarrows\bar\Phi_t^{0j}$. The functor $\tilde{\Phi}_t^0(n)$ is the coproduct of all corollas with $t$ bumpy leaves, $n$ snaky leaves, and an undetermined number of straight leaves, with labels as above. The natural projection is defined by inner edge contraction, e.g.
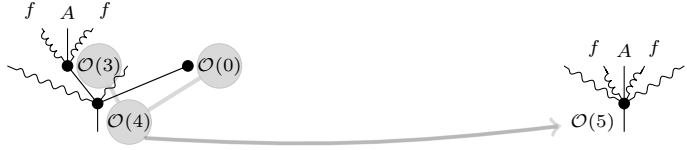
$$\begin{tikzpicture}[level distance=5mm, sibling distance=5mm] 
\tikzstyle{level 1}=[sibling distance=8mm] 
\tikzstyle{level 3}=[sibling distance=5mm] 
\node {} [grow'=up]
child {[fill] circle (2pt) 
	child {{} edge from parent [decorate,decoration={snake,amplitude=.4mm,segment length=2mm,post length=.5mm,pre length=1.5mm}] {}}
	child {[fill] circle (2pt) 
		child{node [above] {$\scriptstyle f$} edge from parent [decorate,decoration={bumps,amplitude=2,segment length=2mm,post length=.5mm,pre length=1.5mm}] {}}
		child{{} node [above] {$\scriptstyle A$}}
		child{node [above] {$\scriptstyle f$} edge from parent [decorate,decoration={bumps,amplitude=2,segment length=2mm,post length=.5mm,pre length=1.5mm}] {}}
		node [right] (O5) {$\scriptstyle \mathcal O(3)$}
	}
	child {{} edge from parent [decorate,decoration={snake,amplitude=.4mm,segment length=2mm,post length=.5mm,pre length=1.5mm}] {}}
	child {[fill] circle (2pt) node [right] (O0) {$\scriptstyle \mathcal O(0)$}}
	node [below right] (O4) {$\scriptstyle \mathcal O(4)$}
};   
\tikzstyle{level 1}=[sibling distance=4mm] 
\draw (7,0) node {} [grow'=up]
child {[fill] circle (2pt) 
	child {{} edge from parent [decorate,decoration={snake,amplitude=.4mm,segment length=2mm,post length=.5mm,pre length=1.5mm}] {}}
	child{node [above] {$\scriptstyle f$} edge from parent [decorate,decoration={bumps,amplitude=2,segment length=2mm,post length=.5mm,pre length=1.5mm}] {}}
	child{{} node [above] {$\scriptstyle A$}}
	child{node [above] {$\scriptstyle f$} edge from parent [decorate,decoration={bumps,amplitude=2,segment length=2mm,post length=.5mm,pre length=1.5mm}] {}}
	child {{} edge from parent [decorate,decoration={snake,amplitude=.4mm,segment length=2mm,post length=.5mm,pre length=1.5mm}] {}}
	node [below left] (O7) {$\scriptstyle \mathcal O(5)$}
}; 
\node[circle, fill, fit=(O4), draw, opacity=.15, very thin, inner sep=-4] (OO4) {};
\node[circle, fill, fit=(O5), draw, opacity=.15, very thin, inner sep=-4] (OO5) {};
\node[circle, fill, fit=(O0), draw, opacity=.15, very thin, inner sep=-4] (OO0) {};
\path[-] (OO4) edge [opacity=.15, ultra thick] (OO5);
\path[-] (OO4) edge [opacity=.15, ultra thick] (OO0);
\path[->] (OO4.south east) edge [opacity=.15, ultra thick, bend right =5] (O7);
\end{tikzpicture}
$$
This is indeed a split coequalizer. The arrows going backwards are defined by subdivision of straight and bumpy leaves, e.g.
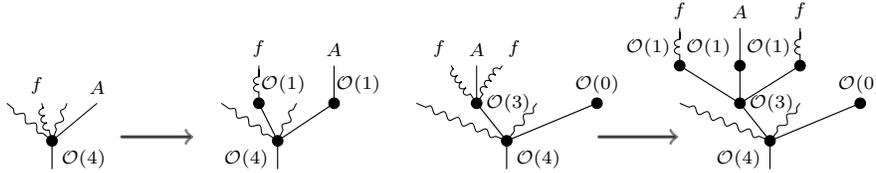
$$\hspace{-7pt}\begin{array}{cc}\begin{tikzpicture}[level distance=5mm, sibling distance=5mm] 
\tikzstyle{level 1}=[sibling distance=4mm] 
\tikzstyle{level 3}=[sibling distance=5mm] 
\node {} [grow'=up]
child {[fill] circle (2pt) 
	child {edge from parent [decorate,decoration={snake,amplitude=.4mm,segment length=2mm,post length=.5mm,pre length=1.5mm}] {}}
	child{node [above] {$\scriptstyle f$} edge from parent [decorate,decoration={bumps,amplitude=2,segment length=2mm,post length=.5mm,pre length=1.5mm}] {}}
	child {edge from parent [decorate,decoration={snake,amplitude=.4mm,segment length=2mm,post length=.5mm,pre length=1.5mm}] {}}
	child{{} node [above] {$\scriptstyle A$}}
	node [below right] (O7) {$\scriptstyle \mathcal O(4)$}
}; 
\tikzstyle{level 1}=[sibling distance=5mm] 
\draw (3,0) node {} [grow'=up]
child {[fill] circle (2pt) 
	child {edge from parent [decorate,decoration={snake,amplitude=.4mm,segment length=2mm,post length=.5mm,pre length=1.5mm}] {}}
	child{[fill] circle (2pt) child{node [above] {$\scriptstyle f$} edge from parent [decorate,decoration={bumps,amplitude=2,segment length=2mm,post length=.5mm,pre length=1.5mm}] {}} node [above right] {$\scriptstyle \!\!\mathcal O(1)$}}
	child {edge from parent [decorate,decoration={snake,amplitude=.4mm,segment length=2mm,post length=.5mm,pre length=1.5mm}] {}}
	child{[fill] circle (2pt) child{{} node [above] {$\scriptstyle A$}} node [above right] {$\scriptstyle \!\!\mathcal O(1)$}}
	node [below left] (O72) {$\scriptstyle \mathcal O(4)$}
}; 
\path[->] ($(O7)+(.5,.3)$) edge [very thick, opacity=.5] ($(O72)+(-.7,.3)$);
\end{tikzpicture}
&
\begin{tikzpicture}[level distance=5mm, sibling distance=5mm] 
\tikzstyle{level 2}=[sibling distance=8mm] 
\tikzstyle{level 3}=[sibling distance=5mm] 
\node {} [grow'=up]
child {[fill] circle (2pt) 
	child {edge from parent [decorate,decoration={snake,amplitude=.4mm,segment length=2mm,post length=.5mm,pre length=1.5mm}] {}}
	child {
		child{node [above] {$\scriptstyle f$} edge from parent [decorate,decoration={bumps,amplitude=2,segment length=2mm,post length=.5mm,pre length=1.5mm}] {}}
		child{{} node [above] {$\scriptstyle A$}}
		child{node [above] {$\scriptstyle f$} edge from parent [decorate,decoration={bumps,amplitude=2,segment length=2mm,post length=.5mm,pre length=1.5mm}] {}}
		[fill] circle (2pt) node [right] {$\scriptstyle \mathcal O(3)$}}
	child {edge from parent [decorate,decoration={snake,amplitude=.4mm,segment length=2mm,post length=.5mm,pre length=1.5mm}] {}}
	child {[fill] circle (2pt) node [above] {$\scriptstyle \mathcal O(0)$}}
	node [below right] (O4) {$\scriptstyle \mathcal O(4)$}
}; 
\tikzstyle{level 2}=[sibling distance=8mm] 
\tikzstyle{level 3}=[sibling distance=8mm] 
\draw (3.5,0) node {} [grow'=up]
child {[fill] circle (2pt) 
	child {edge from parent [decorate,decoration={snake,amplitude=.4mm,segment length=2mm,post length=.5mm,pre length=1.5mm}] {}}
	child {
		child{[fill] circle (2pt) child{node [above] {$\scriptstyle f$} edge from parent [decorate,decoration={bumps,amplitude=2,segment length=2mm,post length=.5mm,pre length=1.5mm}] {}} node [above left] {$\scriptstyle \mathcal O(1)$}}
		child{[fill] circle (2pt) child{{} node [above] {$\scriptstyle A$}} node [above left] {$\scriptstyle \mathcal O(1)$}}
		child{[fill] circle (2pt) child{node [above] {$\scriptstyle f$} edge from parent [decorate,decoration={bumps,amplitude=2,segment length=2mm,post length=.5mm,pre length=1.5mm}] {}} node [above left] {$\scriptstyle \mathcal O(1)$}}
		[fill] circle (2pt) node [right] {$\scriptstyle \mathcal O(3)$}}
	child {edge from parent [decorate,decoration={snake,amplitude=.4mm,segment length=2mm,post length=.5mm,pre length=1.5mm}] {}}
	child {[fill] circle (2pt) node [above] {$\scriptstyle \mathcal O(0)$}}
	node [below left] (O42) {$\scriptstyle \mathcal O(4)$}
};
\path[->] ($(O4)+(.8,.3)$) edge [very thick, opacity=.5] ($(O42)+(-.8,.3)$);
\end{tikzpicture}
\end{array}
$$
One can similarly check that $\tilde{\Phi}_t^1(n)$ is the coproduct of all trees of heigth $\leq 3$  with $t+n$ leaves at level $2$, $t$ bumpy and $n$ snaky,  such that all level $3$ edges (if any) are straight leaves, with the usual labeling. Moreover, the parallel arrows $\tilde{\Phi}_t^1\rightrightarrows\tilde{\Phi}_t^0$ obtained by taking vertical coequalizers in \eqref{cuadrado} are defined by inner edge contraction and corolla contraction, respectively. Therefore $\bar\Phi_t$, regarded as the coequalizer of  $\tilde{\Phi}_t^1\rightrightarrows\tilde{\Phi}_t^0$, coincides with the arrow $\tilde\Phi_t$ in the statement.
\end{proof}  

Unlike Proposition \ref{free}, Proposition \ref{transAcof} was previously known in the symmetric case, at least implicitly. It follows from the proof of \cite[Proposition 5.4]{ahto} and from the construction of push-outs of free operad maps in \cite[\S5]{htnso}.

\section{Proofs}\label{proofs}\label{s3}

We start with a generalization of the Reedy model structure on the category of diagrams indexed by $\dos^n$, see  \cite[\S5.1]{hmc} and \cite[\S15.3]{hirschhorn}. 

\begin{proposition}
		If $\C M$ is a model category and $S\subset\{1,\dots,n\}$, there is a model structure $\C M_S^{\dos^n}$ on the diagram category $\C M^{\dos^n}$ such that a map $\tau\colon F\r G$ is
		\begin{itemize}
			\item a fibration if $\tau(x_1,\dots, x_n)\colon F(x_1,\dots, x_n)\r G(x_1,\dots, x_n)$ is a fibration in $\C M$ for all $(x_1,\dots, x_n)\in\dos^n$,
			\item a weak equivalence if $\tau(x_1,\dots, x_n)$ is a weak equivalence in $\C M$ for all $(x_1,\dots, x_n)\in\dos^n$ with $x_i=0$ if $i\in S$,
			\item and a cofibration if the relative latching map of $\tau$ at any $(x_1,\dots, x_n)\in\dos^n$ is a cofibration, and moreover a trivial cofibration if $x_i= 1$ for some $i\in S$.
		\end{itemize}
\end{proposition}

\begin{proof}
	For $n=0$  and $S=\varnothing$ (the only choice) we recover the given model structure on $\C M=\C M^{\dos^0}$. For $n=1$, $\C M^\dos_\varnothing$ is the Reedy model structure. Let us check that $\C M^\dos_{\{1\}}$ satisfies the axioms of model categories \cite[Definition 1.1.3]{hmc}. Only the parts of the factorization and lifting axioms involving a cofibration and a trivial fibration are not completely trivial. 
	
	A map $X\r Y$ in $\C M^\dos_{\{1\}}$ can be factored as $X\into Z\st{\sim}\onto Y$ in the following way. The map $X\r Y$ is the outer square in the following commutative diagram,\vspace{15pt}
	\[
	\xymatrix@C=10pt{
		X_0\ar[d]\ar@{>->}[rrr]\ar@/^20pt/[rrrrrr]\ar@{}[rrd]|{\text{push}}&&&Z_0\ar@{->>}[rrr]^\sim\ar[d]\ar[dl]&&&Y_0\ar[d]\\
		X_1\ar@{>->}[rr]\ar@/_20pt/[rrrrrr]&&P\ar@{>->}[r]^-\sim\ar@/_10pt/[rrrr]&Z_1\ar@{->>}[rrr]&&&Y_1
		}\vspace{15pt}
	\]
	Here, we first factor $X_0\r Y_0$ as a cofibration followed by a trivial fibration in $\C M$. Then we factor the induced map from the push-out $P=Z_0\cup_{X_0}X_1\r Y_1$ as a trivial cofibration followed by a fibration.
	
	A diagram
	\[\xymatrix{A\ar[r]\ar@{>->}[d]&X\ar@{->>}[d]^\sim\\B\ar[r]&Y}\]
	in $\C M^\dos_{\{1\}}$ is the same as a commutative cube 
	
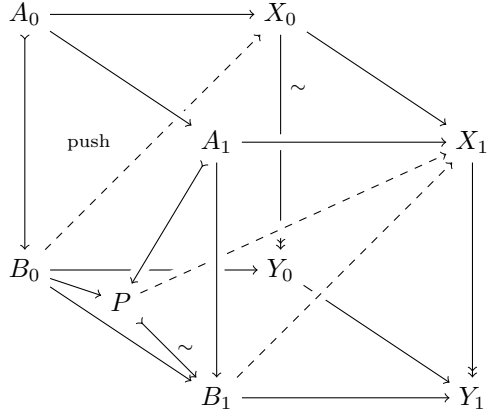
\[
	\begin{tikzpicture}[scale=1.7]
	\node (A0) {$A_0$};
	\draw (2,0) node (X0) {$X_0$};
	\draw (0,-2) node (B0) {$B_0$};
	\draw (2,-2) node (Y0) {$Y_0$};
		\draw (1.5,-1) node (A1) {$A_1$};
		\draw (3.5,-1) node (X1) {$X_1$};
		\draw (1.5,-3) node (B1) {$B_1$};
		\draw (3.5,-3) node (Y1) {$Y_1$};
			\draw (.75,-2.25) node (P) {$P$};
				\draw (.5,-1) node {\scriptsize push};
	\path (A0) edge [->] (X0) edge [>->] (B0);
	\path[->>] (X0) edge node [right, near start] {$\scriptstyle\sim$} (Y0);
	\path[->] (B0) edge (Y0);
			\path[->] (B0) edge [style=dashed] (X0);
		\path[->] (A0) edge [-,line width=6pt,draw=white] (A1) edge (A1);
		\path[->] (B0) edge (B1);
		\path[->] (X0) edge (X1);
		\path[->] (Y0) edge (Y1);
			\path (P) edge [-,line width=6pt,draw=white] (X1) edge [->, style=dashed] (X1);
			\path (B1) edge [-,line width=6pt,draw=white] (X1) edge [->, style=dashed] (X1);
		\path[->] (A1) edge [-,line width=6pt,draw=white] (X1) edge (X1) edge [-,line width=6pt,draw=white] (B1) edge (B1);
		\path[->>] (X1) edge (Y1);
		\path[->] (B1) edge (Y1);
		\path (A1) edge [-,line width=6pt,draw=white] (P) edge [>->] (P);
		\path[->] (B0) edge (P);
		\path[>->] (P) edge node [right] {$\scriptstyle\sim$} (B1);
	\end{tikzpicture}	
	\]
	We can construct a lifting $B\r X$ as follows. We first take a lifting $B_0\r X_0$ in $\C M$ of the rear commutative square. This lifting and the universal property of a push-out induce a map $P=B_0\cup_{A_0}A_1\r X_1$, which is the top arrow of a commutative square containing also $B_1\r Y_1$. We can take a lifting $B_1\r X_1$ of this square in $\C M$ too. These two liftings in $\C M$ define the lifting in $\C M^\dos_{\{1\}}$.
	
	For $n>1$, using the exponential law $\C M^{\dos^n}=\C M^{\dos^{n-1}\times\dos}=(\C M^{\dos^{n-1}})^\dos$, it is easy to see that $\C M^{\dos^n}_S=(\C M^{\dos^{n-1}}_S)^\dos_\varnothing$ if $n\notin S$ and $\C M^{\dos^n}_S=(\C M^{\dos^{n-1}}_{S\setminus\{n\}})^\dos_{\{1\}}$ if $n\in S$, so the proposition follows by induction.
\end{proof}

\begin{remark}\label{yyy}
	Notice that fibrations in $\C M^{\dos^n}_S$ are independent of $S$, they are the same as in the Reedy model structure $\C M^{\dos^n}_\varnothing$, hence the same holds for trivial cofibrations (they are the maps whose relative latching maps are trivial cofibrations in $\C M$). This means that $\C M^{\dos^n}_S$ is a right Bousfield localization of $\C M^{\dos^n}_\varnothing$.
	
	We remind the reader that cofibrant objects in $\C M^{\dos^n}_\varnothing$ are functors $F$ whose latching maps are cofibrations. They take values in cofibrant objects in $\C M$ and have cofibrant latching objects. Moreover, any weak equivalence between cofibrant functors induces weak equivalences between latching objects.	
\end{remark}

Given model categories $\C M$ and $\C N$, we introduce some naive homotopical notions in big functor categories.

\begin{definition}\label{loco}
A map $\tau \colon F\r G$ in $\C M^{\C N^{n}}$ is a weak equivalence, fibration or cofibration if, given cofibrations between cofibrant objects $g_1,\dots,g_n$ in $\C N$, $\tau(g_1,\dots, g_n)$ is so in $\C M^{\dos^n}_S$ for any $S\subset\{1,\dots,n\}$ such that $g_i$ is a trivial cofibration if $i\in S$. A map in $\C M^{\C N^{(\mathbb N)}}$ is a weak equivalence, fibration or cofibration if it is so aritywise. Cofibrant objects are defined in the usual way.
\end{definition}

\begin{remark}\label{after}
	Notice that $S=\varnothing$ is always the smallest subset satisfying the assumptions in the previous definition. For this choice, we obtain the strongest conditions on weak equivalences. Indeed, $\tau \colon F\r G$ is a weak equivalence in $\C M^{\C N^{n}}$ if and only if $\tau (X_1,\dots, X_n)\colon F(X_1,\dots, X_n)\r G(X_1,\dots, X_n)$ is a weak equivalence in $\C M$ for any cofibrant objects $X_1,\dots, X_n$ in $\C N$. This choice also says that cofibrations in $\C M^{\C N^{n}}$ yield Reedy cofibrations when evaluated at cofibrations between cofibrant objects in $\C N$, but they must satisfy extra conditions obtained for the biggest choice of $S$. Similarly for cofibrant functors, which preserve cofibrant objects. Notice that for $n=0$ we recover the original notions in $\C M^{\C N^{0}}=\C M$.

	The identity functor is cofibrant in $\C M^{\C M}$. Suppose that $\C M$ is a monoidal model category. The $n$-fold tensor product is cofibrant in $\C M^{\C M^n}$ by the push-out product axiom, $n\geq 2$. More generally, if $Y$ is a cofibrant object in $\C M$, the functor $(X_1,\dots,X_n)\mapsto Y\otimes\bigotimes_{i=1}^nX_i$ is cofibrant in $\C M^{\C M^n}$, and if $f$ is a  a weak equivalence between cofibrant objects or a (trivial) cofibration  in $\C M$ then so is the natural transformation $(X_1,\dots,X_n)\mapsto f\otimes\bigotimes_{i=1}^nX_i$  in $\C M^{\C M^n}$, $n\geq 0$.
\end{remark}

\begin{proposition}\label{mensgen}
	Let $\mathcal O$ be an operad in $\C V$ such that $\mathcal O(n)$ is cofibrant in $\C C$ for all $n\geq 0$. Then for any cofibrant $A$ in $\algebra{\mathcal O}{\C C}$, 
	$\mathcal O_A$ is cofibrant in $\C C^{\C C^{(\mathbb N)}}$. Moreover, for any cofibration with cofibrant source $f'\colon A\r B$ in $\algebra{\mathcal O}{\C C}$, $\mathcal O_{f'}\colon\mathcal O_A\r\mathcal O_B$ is a cofibration in $\C C^{\C C^{(\mathbb N)}}$.
\end{proposition}

\begin{proof}
	Generating cofibrations in $\algebra{\mathcal O}{\C C}$ are free $\mathcal O$-algebra maps on cofibrations in $\C C$, see \cite[\S9]{htnso}. Hence, by the usual argument involving transfinite compositions and retracts, it is enough to notice that the first statement holds for the initial $\mathcal O$-algebra $A=z(\mathcal O(0))$ in $\C C$, see Proposition \ref{envinitial} and Remark \ref{after}, 	
	and then check the second statement when $f'$ fits into a push-out \eqref{algebrapo} with $f$ a cofibration in $\C C$, assuming that the first statement holds for $A$. Let us check this.
	
	We can replace $f$ with the bottom map in the following push-out in $\C C$,
	\[\xymatrix{
		X\ar[d]_{\bar g}\ar@{}[rd]|{\text{push}}\ar@{>->}[r]^-f&Y\ar[d]\\
		A\ar@{>->}[r]&A\cup_XY
		}\]
	Compare the trick at the beginning of the proof of \cite[Proposition 4.2]{htnso2}. Therefore, since $A=\mathcal O_A(0)$ is cofibrant in $\C C$, we can suppose that $X$, and hence $Y$, is cofibrant. The factors of the coproduct $\tilde\Phi_t(n)$ in Proposition \ref{transAcof} are cofibrations with cofibrant source and target in $\C C^{\C C^n}$ for all $n\geq 0$ and $t\geq 1$. This follows from the facts that $f$ is a cofibration between cofibrant objects and $\mathcal O_A(n)$ is cofibrant in  $\C C^{\C C^{n}}$ for $n\geq 1$. Hence $\Phi_t$ is a cofibration for all $t\geq 1$, and $\mathcal O_{f'}$ is a cofibration since it is a transfinite composition of cofibrations. This also proves that $\mathcal O_{B,t}$ is cofibrant in $\C C^{\C C^{(\mathbb N)}}$, $t\geq 0$. This fact will be used later.
\end{proof}

\begin{proposition}\label{masgen}
Let $\phi\colon\mathcal O\st{\sim}\r\mathcal P$ be a weak equivalence in $\operad{\C V}$. Assume that the objects $\mathcal O(n)$ and $\mathcal P(n)$ are cofibrant in $\C V$ for all $n\geq 0$. Given a cofibrant $\mathcal O$-algebra $A$ in $\C C$, the map $\phi_{\eta_A}\colon \mathcal O_A\r\mathcal P_{\phi_*A}$ induced by $\phi$ and by the unit $\eta_A\colon A\r\phi^{*}\phi_{*}A$ of the change of operad adjunction $\phi_{*}\dashv\phi^{*}$ \cite[(1)]{htnso} is a weak equivalence in $\C C^{\C C^{(\mathbb N)}}$.
\end{proposition}

\begin{proof}
As in the proof of Proposition \ref{mensgen}, it is enough to check the statement for $A=z(\mathcal O(0))$ the initial $\mathcal O$-algebra, and then for $B$ in \eqref{algebrapo} assuming that it holds for $A$ and that $f$ is a cofibration between cofibrant objects.	

For the initial $\mathcal O$-algebra, $\phi_{\eta_A}$ is induced by the sequence of weak equivalences between cofibrant objects $z(\phi)$, see Proposition \ref{envinitial} and Remark \ref{after}, so it is a weak equivalence.

In the situiation of \eqref{algebrapo}, we have another push-out square
\[\xymatrix{
	\mathcal F_{\mathcal P}(Y)\ar[r]^-{\mathcal F_{\mathcal P}(f)}\ar[d]_{\phi_*g}\ar@{}[rd]|{\text{push}}&\mathcal F_{\mathcal P}(Z)\ar[d]^{\phi_*g'}\\
	\phi_*A\ar[r]_-{\phi_*f'}&\phi_*B
}\]
where $\phi_*g$ and $\phi_*g'$ are adjoints to $\bar g$ and $\bar g'$, respectively. 
Notice that $\phi_*A$ is a cofibrant $\mathcal P$-algebra, since $\phi_*$ is a left Quillen functor. 
Using Proposition \ref{transAcof}, $\phi_{\eta_B}$ is the colimit of a sequence
\[
\xymatrix@C=10pt{
	\mathcal O_A=\mathcal O_{B,0}\ar[rr]^-{\Phi_1^{\mathcal O}}
	\ar@<-4.5ex>[d]_{\phi_{\eta_A}}^\sim	\ar@<3.5ex>[d]^{\phi_{\eta_{B,0}}}
	&&\mathcal O_{B,1}\ar[r]\ar[d]^{\phi_{\eta_{B,1}}}&\cdots\ar[r]&\mathcal O_{B,t-1}\ar[rr]^-{\Phi_t^{\mathcal O}}\ar[d]^{\phi_{\eta_{B,t-1}}}&&\mathcal O_{B,t}\ar[r]\ar[d]^{\phi_{\eta_{B,t}}}&\cdots\\
	\mathcal P_{\phi_*A}=\mathcal P_{\phi_*B,0}\ar[rr]^-{\Phi_1^{\mathcal P}}&&\mathcal P_{\phi_*B,1}\ar[r]&\cdots\ar[r]&\mathcal P_{\phi_*B,t-1}\ar[rr]^-{\Phi_t^{\mathcal P}}&&\mathcal P_{\phi_*B,t}\ar[r]&\cdots
	}
\]
where all objects are cofibrant and horizontal maps are cofibrations, see  Proposition \ref{mensgen} and its proof. The map $\phi_{\eta_{B,t}}$ is obtained by taking horizontal push-outs in the following commutative diagram
\[
\xymatrix{
	\mathcal O_{B,t-1}\ar[d]_{\phi_{\eta_{B,t-1}}}&\bullet\ar[l]_-{\Psi_t^{\mathcal O}}\ar[r]^-{\tilde{\Phi}_t^{\mathcal O}}\ar[d]&\bullet\ar[d]\\
	\mathcal P_{\phi_*B,t-1}&\bullet\ar[l]_-{\Psi_t^{\mathcal P}}\ar[r]^-{\tilde{\Phi}_t^{\mathcal P}}&\bullet
	}
\]
All objects in this diagram are cofibrant and the two arrows pointing $\rightarrow$ are cofibrations, see the proof of Proposition \ref{mensgen} again. The square on the right is aritywise the factor-preserving map between coproducts $\tilde{\Phi}_t^{\mathcal O}(n)\r\tilde{\Phi}_t^{\mathcal P}(n)$ in $\mor{\C C}^{\C C^{n}}$ induced by $\phi_{\eta_A}$ on inner vertices, e.g.~$(n=2)$
\[
\begin{tikzpicture}[level distance=5mm, sibling distance=3mm] 
\node {} [grow'=up]
child {[fill] circle (2pt) 
	child {{} node [above] {$\scriptstyle f$} edge from parent [decorate,decoration={bumps,amplitude=2,segment length=2mm,post length=.5mm,pre length=1.5mm}] {}}
	child {{} %node [above] {$\scriptstyle X_1$} 
		edge from parent [decorate,decoration={snake,amplitude=.4mm,segment length=2mm,post length=.5mm,pre length=1.5mm}] {}}
	child {{} node [above] {$\scriptstyle f$} edge from parent [decorate,decoration={bumps,amplitude=2,segment length=2mm,post length=.5mm,pre length=1.5mm}] {}}
	child {{} node [above] {$\scriptstyle f$} edge from parent [decorate,decoration={bumps,amplitude=2,segment length=2mm,post length=.5mm,pre length=1.5mm}] {}}
	child {{} %node [above] {$\scriptstyle X_2$} 
		edge from parent [decorate,decoration={snake,amplitude=.4mm,segment length=2mm,post length=.5mm,pre length=1.5mm}] {}}
	node [below right] (L) {$\scriptstyle \mathcal O_A(5)$}
};
\draw (4cm,0) node {} [grow'=up]
child {[fill] circle (2pt) 
	child {{} node [above] {$\scriptstyle f$} edge from parent [decorate,decoration={bumps,amplitude=2,segment length=2mm,post length=.5mm,pre length=1.5mm}] {}}
	child {{} %node [above] {$\scriptstyle X_1$} 
		edge from parent [decorate,decoration={snake,amplitude=.4mm,segment length=2mm,post length=.5mm,pre length=1.5mm}] {}}
	child {{} node [above] {$\scriptstyle f$} edge from parent [decorate,decoration={bumps,amplitude=2,segment length=2mm,post length=.5mm,pre length=1.5mm}] {}}
	child {{} node [above] {$\scriptstyle f$} edge from parent [decorate,decoration={bumps,amplitude=2,segment length=2mm,post length=.5mm,pre length=1.5mm}] {}}
	child {{} %node [above] {$\scriptstyle X_2$} 
		edge from parent [decorate,decoration={snake,amplitude=.4mm,segment length=2mm,post length=.5mm,pre length=1.5mm}] {}}
	node [below left] (R) {$\scriptstyle \mathcal P_{\phi_*A}(5)$}
};
\path[opacity=.5, very thick] (L) edge [->] node [above, opacity=1] {$\scriptstyle \phi_{\eta_{A}}(5)$} (R);
\end{tikzpicture}
\]
This map is a weak equivalence since $f$ is a cofibration between cofibrant objects and $\phi_{\eta_A}$ is a weak equivalence (by assumption) between cofibrant objects (by Proposition \ref{mensgen}). The starting map $\phi_{\eta_{B,0}}=\phi_{\eta_{A}}$ is a weak equivalence, hence we deduce by induction, using the cube lemma \cite[Lemma 5.2.6]{hmc}, that the maps $\phi_{\eta_{B,t}}$ are weak equivalences for all $t\geq 0$. Since $\phi_{\eta_{B}}$ is the colimit of weak equivalences between sequences of cofibrations with cofibrant starting objects, we deduce that $\phi_{\eta_{B}}$  is also a weak equivalence  \cite[Proposition 15.10.12 (1)]{hirschhorn}. Notice that \cite[Proposition 15.10.12 (1)]{hirschhorn} is also true for continuous sequences indexed by an arbitrary ordinal, compare \cite[Corollary 5.1.6]{hirschhorn}. This generalization is used to check the limit steps in the transfinite induction.
\end{proof}

Let $L\colon\C M\rightleftarrows \C N\colon R$ be a Quillen pair. Left composition with $L$ induces a `functor' $\C M^{\C M^n}\r \C N^{\C M^n}$ which preserves (trivial) cofibrations, $n\geq 0$. Right composition with $L^{\times n}=L\times\st{n}\cdots\times L$ gives rise to a `functor' $\C N^{\C N^n}\r \C N^{\C M^n}$ which preserves all homotopical notions in Definition \ref{loco}, $n\geq 0$. We will denote $L^{\times(\mathbb N)}=\{L^{\times n}\}_{n\geq 0}$.
	
Let us place ourselves in of \cite[\S7]{htnso2}. There is a natural map in $\C D^{\C C^{(\mathbb N)}}$, 
$$\chi_{\mathcal O,A} \colon\bar F\mathcal O_A \To F^{\oper}(\mathcal O)_{\bar F_{\mathcal O}(A)} \bar F^{\times(\mathbb N)}$$
defined by taking vertical coequalizers in the following diagram
\[
\xymatrix{
	\bar F\mathcal O_A^1 \ar[r]^-{\chi_{\mathcal O,A}^1 }\ar@<-.5ex>[d]\ar@<.5ex>[d]& F^{\oper}(\mathcal O)_{\bar F_{\mathcal O}(A)}^1 \bar F^{\times(\mathbb N)}\ar@<-.5ex>[d]\ar@<.5ex>[d]\\
	\bar F\mathcal O_A^0 \ar[r]^-{\chi_{\mathcal O,A}^0 }& F^{\oper}(\mathcal O)_{\bar F_{\mathcal O}(A)}^0 \bar F^{\times(\mathbb N)}
	}
\]
The functor $\bar F$ preserves coproducts since it is a left adjoint, and $\chi_{\mathcal O,A}^0(n)$ is factorwise defined as illustrated in the following picture $(n=2)$
\[
\bar F\left(\!\!\!\begin{array}{c}
\begin{tikzpicture}[level distance=5mm, sibling distance=5mm] 
\node {} [grow'=up]
child {[fill] circle (2pt) 
	child {{} node [above] {$\scriptstyle A$}}
	child {edge from parent [decorate,decoration={snake,amplitude=.4mm,segment length=2mm,post length=.5mm,pre length=1.5mm}] {}}
	child {{} node [above] {$\scriptstyle A$}}
	child {{} node [above] {$\scriptstyle A$}}
	child {edge from parent [decorate,decoration={snake,amplitude=.4mm,segment length=2mm,post length=.5mm,pre length=1.5mm}] {}}
	node [below left] {$\scriptstyle \mathcal O(5)$}
};
\end{tikzpicture}
\end{array}\right)
\begin{array}{c}
\begin{tikzpicture}[level distance=5mm, sibling distance=7mm] 
\node {} [grow'=up]
child {[fill] circle (2pt) 
	child {{} node [above] {$\scriptstyle \bar F(A)$}}
	child {{} node [above] {$\scriptstyle \bar F(-)$} edge from parent [decorate,decoration={snake,amplitude=.4mm,segment length=2mm,post length=.5mm,pre length=1.5mm}] {}}
	child {{} node [above] {$\scriptstyle \bar F(A)$}}
	child {{} node [above] {$\scriptstyle \bar F(A)$}}
	child {{} node [above] {$\scriptstyle \bar F(-)$} edge from parent [decorate,decoration={snake,amplitude=.4mm,segment length=2mm,post length=.5mm,pre length=1.5mm}] {}}
	node [below right] (O1) {$\scriptstyle F(\mathcal O(5))$}
};
\draw (4.5cm,0) node {} [grow'=up]
child {[fill] circle (2pt) 
	child {{} node [above] {$\scriptstyle \bar F(A)$}}
	child {{} node [above] {$\scriptstyle \bar F(-)$} edge from parent [decorate,decoration={snake,amplitude=.4mm,segment length=2mm,post length=.5mm,pre length=1.5mm}] {}}
	child {{} node [above] {$\scriptstyle \bar F(A)$}}
	child {{} node [above] {$\scriptstyle \bar F(A)$}}
	child {{} node [above] {$\scriptstyle \bar F(-)$} edge from parent [decorate,decoration={snake,amplitude=.4mm,segment length=2mm,post length=.5mm,pre length=1.5mm}] {}}
	node [below left] (O2) {$\scriptstyle F^{\oper}(\mathcal O(5))$}
};
\path[opacity=.5, very thick] (-2.5,.8) edge [->] (-1.5,.8);
\node[rectangle, rounded corners, fit=(O1), draw, opacity=.5, very thick, inner sep=-2] (E) {};
\path[->] (E) edge [opacity=.5, very thick] (O2);
\end{tikzpicture}
\end{array}
\]
Here, the first arrow is defined by the comultiplication of $\bar F$ and by the natural transformation $\tau \colon \bar Fz_{\C C}\r z_{\C D}F$, and the second one is defined by the map  $\chi\colon F(\mathcal O)\r F^{\oper}(\mathcal O)$ of sequences in $\C V$ \cite[(4-1)]{htnso2}. Similarly $\chi_{\mathcal O,A}^1(n)$.

\begin{proposition}\label{masgen1/2}
If $\bar F\dashv \bar G$ is a weak monoidal Quillen adjunction, $\C V$ and $\C W$ have cofibrant tensor units, $\mathcal O$ is a  cofibrant operad in $\C V$, and $A$ is a cofibrant $\mathcal O$-algebra in $\C C$, then $\chi_{\mathcal O,A}$ is a weak equivalence in $\C C^{\C C^{(\mathbb N)}}$.
\end{proposition}

\begin{proof}
	It is enough to prove the statement just in the two cases described at the beginning of the poof of Proposition \ref{masgen}.
	
	If $A$ is the initial $\mathcal O$-algebra, then $\bar F_{\mathcal O}(A)$ is the initial $F^{\oper}(\mathcal O)$-algebra and $\chi_{\mathcal O,A}$ 
	is defined like $\chi_{\mathcal O,A}^{0}$ above on corollas withouht straight leaves. 
	Comultiplication is a weak equivalence when evaluated at cofibrant objects, since $\bar F$ is a weak monoidal left Quillen functor \cite[Definition 2.6]{htnso2}. The natural transformation $\tau$ is also a weak equivalence between cofibrant objects when evaluated at $\mathcal O(n)$, since it is cofibrant in $\C V$ \cite[Corollary 3.8]{htnso2}. Moreover, $\chi_{\mathcal O}$ is a sequence of weak equivalences between cofibrant objects by \cite[Proposition 4.2]{htnso2}. This shows that $\chi_{\mathcal O,A}$ is a weak equivalence in this case.
 
	Assume we have a push-out \eqref{algebrapo} such that $f$ is a cofibration between cofibrant objects and $A$ satisfies the proposition. Then $\chi_{\mathcal O,B}$ is the colimit of a sequence where all objects are cofibrant and horizontal maps are cofibrations, see Proposition \ref{mensgen} and its proof,
	\[
	\xymatrix@C=10pt{
		\cdots\ar[r]&\bar F\mathcal O_{B,t-1}\ar[rrrr]^-{\bar F\Phi_t^{\mathcal O}}\ar[d]^{\chi_{\mathcal O,{B,t-1}}}&&&&\bar F\mathcal O_{B,t}\ar[r]\ar[d]^{\chi_{\mathcal O,{B,t}}}&\cdots\\
		\cdots\ar[r]& F^{\oper}(\mathcal O)_{\bar F_{\mathcal O}(B),t-1}\bar F^{\times (\mathbb N)}\ar[rrrr]^-{\Phi_t^{F^{\oper}(\mathcal O)}\bar F^{\times (\mathbb N)}}&&&& F^{\oper}(\mathcal O)_{\bar F_{\mathcal O}(B),t}\bar F^{\times (\mathbb N)}\ar[r]&\cdots
	}
	\]
	This sequence starts with the weak equivalence $\chi_{\mathcal O,B,0}=\chi_{\mathcal O,A}$. The map $\chi_{\mathcal O,B,t}$ is
	obtained by taking horizontal push-outs in the following commutative diagram 
	\[
	\xymatrix@C=70pt{
		\bar F\mathcal O_{B,t-1}\ar[d]_{\phi_{\eta_{B,t-1}}}&\bullet\ar[l]_-{\bar F\Psi_t^{\mathcal O}}\ar[r]^-{\bar F\tilde{\Phi}_t^{\mathcal O}}\ar[d]&\bullet\ar[d]\\
		F^{\oper}(\mathcal O)_{\bar F_{\mathcal O}(B),t-1}\bar F^{\times (\mathbb N)}&\bullet\ar[l]_-{\Psi_t^{F^{\oper}(\mathcal O)}\bar F^{\times (\mathbb N)}}\ar[r]^-{\tilde{\Phi}_t^{F^{\oper}(\mathcal O)}\bar F^{\times (\mathbb N)}}&\bullet
	}
	\]
	The two arrows pointing $\rightarrow$ are cofibrations between cofibrant objects, see again the proof of Proposition  \ref{mensgen}. The square on the right is aritywise the factor-preserving map between coproducts in $\mor{\C C}^{\C C^n}$ defined by $\chi_{\mathcal O,A}$ on inner vertices, e.g.~$(n=2)$
	\[
	\begin{tikzpicture}[level distance=5mm, sibling distance=3mm] 
	\node {} [grow'=up]
	child {[fill] circle (2pt) 
		child {{} node [above] {$\scriptstyle f$} edge from parent [decorate,decoration={bumps,amplitude=2,segment length=2mm,post length=.5mm,pre length=1.5mm}] {}}
		child {{} %node [above] {$\scriptstyle X_1$} 
			edge from parent [decorate,decoration={snake,amplitude=.4mm,segment length=2mm,post length=.5mm,pre length=1.5mm}] {}}
		child {{} node [above] {$\scriptstyle f$} edge from parent [decorate,decoration={bumps,amplitude=2,segment length=2mm,post length=.5mm,pre length=1.5mm}] {}}
		child {{} node [above] {$\scriptstyle f$} edge from parent [decorate,decoration={bumps,amplitude=2,segment length=2mm,post length=.5mm,pre length=1.5mm}] {}}
		child {{} %node [above] {$\scriptstyle X_2$} 
			edge from parent [decorate,decoration={snake,amplitude=.4mm,segment length=2mm,post length=.5mm,pre length=1.5mm}] {}}
		node [below right] (L) {$\scriptstyle \bar F\mathcal O_A(5)$}
	};
	\draw (6cm,0) node {} [grow'=up]
	child {[fill] circle (2pt) 
		child {{} node [above] {$\scriptstyle f$} edge from parent [decorate,decoration={bumps,amplitude=2,segment length=2mm,post length=.5mm,pre length=1.5mm}] {}}
		child {{} %node [above] {$\scriptstyle X_1$} 
			edge from parent [decorate,decoration={snake,amplitude=.4mm,segment length=2mm,post length=.5mm,pre length=1.5mm}] {}}
		child {{} node [above] {$\scriptstyle f$} edge from parent [decorate,decoration={bumps,amplitude=2,segment length=2mm,post length=.5mm,pre length=1.5mm}] {}}
		child {{} node [above] {$\scriptstyle f$} edge from parent [decorate,decoration={bumps,amplitude=2,segment length=2mm,post length=.5mm,pre length=1.5mm}] {}}
		child {{} %node [above] {$\scriptstyle X_2$} 
			edge from parent [decorate,decoration={snake,amplitude=.4mm,segment length=2mm,post length=.5mm,pre length=1.5mm}] {}}
		node [below left] (R) {$\scriptstyle F^{\oper}(\mathcal O)_{\bar F_{\mathcal O}(A)}(5)\bar F^{\times 5}$}
	};
	\tikzstyle{level 2}=[sibling distance=7mm] 
	\draw (10cm,0) node {} [grow'=up]
		child {[fill] circle (2pt) 
			child {{} node [above] {$\scriptstyle \bar F(f)$} edge from parent [decorate,decoration={bumps,amplitude=2,segment length=2mm,post length=.5mm,pre length=1.5mm}] {}}
			child {{} node [above] {$\scriptstyle \bar F(-)$} 
				edge from parent [decorate,decoration={snake,amplitude=.4mm,segment length=2mm,post length=.5mm,pre length=1.5mm}] {}}
			child {{} node [above] {$\scriptstyle \bar F(f)$} edge from parent [decorate,decoration={bumps,amplitude=2,segment length=2mm,post length=.5mm,pre length=1.5mm}] {}}
			child {{} node [above] {$\scriptstyle \bar F(f)$} edge from parent [decorate,decoration={bumps,amplitude=2,segment length=2mm,post length=.5mm,pre length=1.5mm}] {}}
			child {{} node [above] {$\scriptstyle \bar F(-)$}
				edge from parent [decorate,decoration={snake,amplitude=.4mm,segment length=2mm,post length=.5mm,pre length=1.5mm}] {}}
			node [below left] {$\scriptstyle F^{\oper}(\mathcal O)_{\bar F_{\mathcal O}(A)}(5)$}
		};
	\draw (7,.5) node {$=$};	
	\path[opacity=.5, very thick] (L) edge [->] node [above, opacity=1] {$\scriptstyle \chi_{\mathcal O,A}(5)$} (R);
	\end{tikzpicture}
	\]
Since $A$ satisfies the proposition, this is a weak equivalence, and the result follows as in the last paragraph of the proof of Proposition \ref{masgen}.
\end{proof}

We now abandon the context of \cite[\S7]{htnso2}.

\begin{definition}\label{excellent}
An operad $\mathcal O$ in $\C V$ is \emph{excellent} if:
\begin{enumerate}
\item Given an $\mathcal O$-algebra $A$ in $\C C$ with underlying cofibrant object, $\mathcal O_A$ is cofibrant in $\C C^{\C C^{(\mathbb N)}}$.
\item For any weak equivalence $\varphi\colon A\r C$ between $\mathcal O$-algebras in $\C C$ with underlying cofibrant objects, $\mathcal O_{\varphi}\colon\mathcal O_{A}\r\mathcal O_{C}$ is a weak equivalence in $\C C^{\C C^{(\mathbb N)}}$.
\end{enumerate}
\end{definition}

\begin{proposition}\label{sonex1}
The initial operad,  $\mathtt{Ass}^{\C V}$, and $\mathtt{uAss}^{\C V}$, are excellent.
\end{proposition}

This follows from Propositions \ref{uass}, \ref{ass}, and \ref{anillos}, together with the push-out product axiom.

\begin{proposition}\label{marrown}
If $\mathcal O$ is an excellent operad in $\C V$ and $\psi\colon A\r B$  is a cofibration of $\mathcal O$-algebras in $\C C$ such that $A$ has underlying cofibrant object, then $\psi$ is a cofibration in $\C C$. In particular $B$ also has an underlying cofibrant object
\end{proposition}
 
\begin{proof}
As in the proof of Proposition \ref{mensgen},  we can assume that $\psi=f'$ in \eqref{algebrapo} with $f$ a cofibration between cofibrant objects in $\C C$. By \cite[\S8]{htnso} (corrected version), $A\r B$ is a transfinite composition of push-outs of the maps $\mathcal O_A(t)(f,\dots,f)$. These maps are cofibrations by Definition \ref{excellent} (1).
\end{proof}

\begin{theorem}\label{cuandoex}
If $\mathcal O$ is an excellent operad in $\C V$ and 
$$\xymatrix{A\ar@{ >->}[r]^-{\psi}\ar[d]_-{\varphi}^{\sim}\ar@{}[rd]|{\text{push}}&B\ar[d]^-{\varphi'}\\
C\ar@{ >->}[r]_-{\psi'}&C\cup_{A}B}$$
is a push-out of $\mathcal O$-algebras in $\C C$ such that the underlying objects of $A$ and $C$ are cofibrant, then $\varphi'$ is a weak equivalence. 
\end{theorem}

\begin{proof}
As in the proof of Proposition \ref{mensgen}, it is enough to consider the case where $\psi=f'$ is a push-out \eqref{algebrapo} with $f$ a cofibration between cofibrant objects. In this case, the map $\psi'$ fits into the following push-out
\[\xymatrix{
	\mathcal F_{\mathcal O}(Y)\ar[r]^-{\mathcal F_{\mathcal O}(f)}\ar[d]_{\text{adjoint of }\varphi\bar g}\ar@{}[rd]|{\text{push}}&\mathcal F_{\mathcal O}(Z)\ar[d]^{\text{adjoint of }\varphi'\bar g'}\\
	C\ar[r]_-{\psi'}&C\cup_AB
}\]
and $\varphi'$ is the colimit of a sequence where all objects are cofibrant and horizontal maps are cofibrations, see Proposition \ref{marrown} and its proof, 
\[
\xymatrix@C=10pt{
	A=B_{0}\ar[rr]
	\ar@<-3ex>[d]_{\varphi}^\sim	\ar@<2.5ex>[d]^{\varphi'_{{0}}}
	&&B_{1}\ar[r]\ar[d]^{\varphi'_{{1}}}&\cdots\ar[r]&B_{t-1}\ar[rr]\ar[d]^{\varphi'_{{t-1}}}&&B_{t}\ar[r]\ar[d]^{\varphi'_{{t}}}&\cdots\\
	C=D_{0}\ar[rr]&&D_{1}\ar[r]&\cdots\ar[r]&D_{t-1}\ar[rr]&&D_{t}\ar[r]&\cdots
}
\]
The map $\varphi'_{t}$ is obtained by taking horizontal push-outs in the following commutative diagram of cofibrant objects
\[
\xymatrix{
	B_{t-1}\ar[d]_{\varphi_{t-1}'}&\bullet\ar[l]_-{\psi_t^{A}}\ar[rr]^-{\mathcal O_A(t)(f,\dots,f)}\ar[d]&&\bullet\ar[d]\\
	D_{t-1}&\bullet\ar[l]_-{\psi_t^{C}}\ar[rr]^-{\mathcal O_C(t)(f,\dots,f)}&&\bullet
}
\]
Here, the two arrows pointing $\r$ are cofibrations between cofibrant objects since $\mathcal O_A$ and $\mathcal O_C$ are cofibrant in $\C C^{\C C^{(\mathbb N)}}$ and $f$ is a cofibration between cofibrant objects. The second square is the map $\mathcal O_\varphi(t)(f,\dots,f)\colon \mathcal O_A(t)(f,\dots,f)\r \mathcal O_C(t)(f,\dots,f)$ in $\mor{\C C}$. This map is a weak equivalence because $\mathcal O_\varphi$ is a weak equivalence between cofibrant objects. Now the result follows as in the proof of Proposition \ref{masgen}.
\end{proof}

Cofibrant functors preserve weak equivalences between cofibrant objects.

\begin{lemma}\label{condon}
Any functor $F\colon\C N^n\r\C M$ which is cofibrant in $\C M^{\C N^{n}}$ takes a weak equivalence between cofibrant objects in $\C N^n$, with respect to the product model structure, to a weak equivalence between cofibrant objects in $\C M$.
\end{lemma}

\begin{proof}
Given trivial cofibrations between cofibrant objects $g_i\colon Y_i\r Z_i$ in $\C N$, $1\leq i\leq n$, the natural map $F(Y_{1},\dots, Y_{n})\r F(g_{1},\dots, g_{n})$ in $\C M^{\dos^{n}}$ from the constant diagram is a trivial cofibration since the latching map of $F(g_1,\dots,g_n)$ at any object different from $(0,\dots,0)$ is a trivial cofibration. This trivial cofibration in $\C M^{\dos^{n}}$, evaluated at $(1,\dots,1)\in\dos^{n}$, induces a trivial cofibration $F(Y_{1},\dots, Y_{n})\r F(Z_{1},\dots, Z_{n})$ in $\C M$ \cite[Proposition 15.3.11 (1)]{hirschhorn}. This proves that $F$ preserves trivial cofibrations between cofibrant objects. The result now follows from Ken Brown's lemma \cite[Lemma 1.1.12]{hmc}.
\end{proof}

As a consequence, weak equivalences between cofibrant objects in big functor categories are closed under horizontal composition.

\begin{corollary}\label{condon2}
Given weak equivalences between cofibrant objects $\tau\colon F\r G$ in $\C M^{\C M^{p}}$ and $\tau'\colon F'\r G'$ in $\C M^{\C M^{q}}$, the composite $\tau(\st{i-1}\dots,\tau',\st{p-i}\dots)\colon F(\st{i-1}\dots,F',\st{p-i}\dots)\r G(\st{i-1}\dots,G',\st{p-i}\dots)$ is a weak equivalence in $\C M^{\C M^{p+q-1}}$ for any $1\leq i\leq p$ and $q\geq 0$. 
\end{corollary}

Colimit preserving cofibrant objects in big functor categories are closed under composition. 

\begin{lemma}\label{compcof}
Given cofibrant objects $F$  in $\C M^{\C M^{p}}$ and $G$ in $\C M^{\C M^{p}}$, if $F$ preserves colimits in the $i^{\text{th}}$ variable then $F(\st{i-1}\dots, G,\st{p-i}\dots)$ is cofibrant in $\C M^{\C M^{p+q-1}}$, $1\leq i\leq p$, $q\geq 0$.
\end{lemma}

\begin{proof}
This follows form the fact that, given maps $g_{j}\colon Y_{j}\r Z_{j}$ in $\C M$, $1\leq j\leq p+q-1$,  the latching map of $F(g_{1},\dots, g_{i-1},G(g_{i},\dots, g_{i+q-1}),g_{i+q},\dots, g_{p+q-1})$ at  $x=(x_{1},\dots, x_{p+q-1})\in\dos^{p+q-1}$ is the latching map of
$F(g_{1},\dots, g_{i-1},h,g_{i+q},\dots, g_{p+q-1})$
at $(x_{1},\dots, x_{i-1},1,x_{i+q},\dots, x_{p+q-1})\in\dos^{p}$, where $h$ denotes the latching map of $G(g_{i},\dots, g_{i+q-1})$ at $(x_{i},\dots, x_{i+q-1})\in\dos^{q}$.
\end{proof}

\begin{proposition}\label{sonex2}
If $f'\colon\mathcal O\into\mathcal P$ is a cofibration in $\operad{\C V}$ and $\mathcal O$ is an excellent operad such that $\mathcal O(n)$ is cofibrant for all $n\geq 0$, then so is $\mathcal P$.
\end{proposition}

\begin{proof}
Like in the proof of Proposition \ref{mensgen}, it is enough to assume that $f'$ fits into a push-out diagram \eqref{pop} where $f$ is a sequence of cofibrations between cofibrant objects. The objects $\mathcal P(n)$ are cofibrant by \cite[Corollary 3.7]{htnso2}. If $A$ is a $\mathcal P$-algebra with underlying cofibrant object then, by Proposition \ref{horror}, $\mathcal P_A$ is the colimit of 
\[
\xymatrix@C=10pt{
	\mathcal O_{A}=\mathcal P_{A,0}\ar[rr]^-{\Phi_1}&&\mathcal P_{A,1}\ar[r]&\cdots\ar[r]&\mathcal P_{A,t-1}\ar[rr]^-{\Phi_t}&&\mathcal P_{A,t}\ar[r]&\cdots
}
\]
Here $\mathcal O_A$ is cofibrant in $\C C^{\C C^{(\mathbb N)}}$ since $\mathcal O$ is excellent. 
The map $\Phi_t$ is a push-out of $\tilde \Phi_t$, which is a cofibration between cofibrant objects in $\C C^{\C C^{(\mathbb N)}}$ by Lemma \ref{compcof} and the push-out product axiom, since $f$ is a sequence of cofibrations between cofibrant objects and $\mathcal O_A$ is cofibrant. We deduce that $\Phi_t$ is a cofibration for all $t\geq 1$, $\mathcal P_{A,t}$ is cofibrant for $t\geq 0$, the transfinite composition of the previous diagram is also a cofibration, and $\mathcal P_A$ is cofibrant.

If $\varphi\colon A\st{\sim}\r C$ is a weak equivalence of $\mathcal P$-algebras with underlying cofibrant objects in $\C C$, then $\mathcal P_\varphi$ is the colimit of a diagram, 
\[
\xymatrix@C=10pt{
	\mathcal O_A=\mathcal P_{A,0}\ar[rr]^-{\Phi_1^{A}}
	\ar@<-4.5ex>[d]_{\mathcal O_\varphi}^\sim	\ar@<3.5ex>[d]^{\mathcal P_{\varphi,{0}}}
	&&\mathcal P_{A,1}\ar[r]\ar[d]^{\mathcal P_{\varphi,{1}}}&\cdots\ar[r]&\mathcal P_{A,t-1}\ar[rr]^-{\Phi_t^{A}}\ar[d]^{\mathcal P_{\varphi,{t-1}}}&&\mathcal P_{A,t}\ar[r]\ar[d]^{\mathcal P_{\varphi,{t}}}&\cdots\\
	\mathcal O_{C}=\mathcal P_{C,0}\ar[rr]^-{\Phi_1^{C}}&&\mathcal P_{C,1}\ar[r]&\cdots\ar[r]&\mathcal P_{C,t-1}\ar[rr]^-{\Phi_t^{C}}&&\mathcal P_{C,t}\ar[r]&\cdots
}
\]
such that $\mathcal P_{\varphi,t}$ is obtained by taking horizontal push-outs in the commutative diagram
\[
\xymatrix{
	\mathcal P_{A,t-1}\ar[d]_{\mathcal P_{\varphi,t-1}}&\bullet\ar[l]_-{\Psi_t^{A}}\ar[r]^-{\tilde{\Phi}_t^{A}}\ar[d]&\bullet\ar[d]\\
	\mathcal P_{C,t-1}&\bullet\ar[l]_-{\Psi_t^{C}}\ar[r]^-{\tilde{\Phi}_t^{C}}&\bullet
}
\]

Aritywise, the commutative square on the right is
the factor-preserving map between coproducts $\tilde{\Phi}_t^{A}(n)\r \tilde{\Phi}_t^{C}(n)$ in $\mor{\C C}^{\C C^n}$  defined by $\mathcal O_\varphi$ on odd inner vertices, e.g.
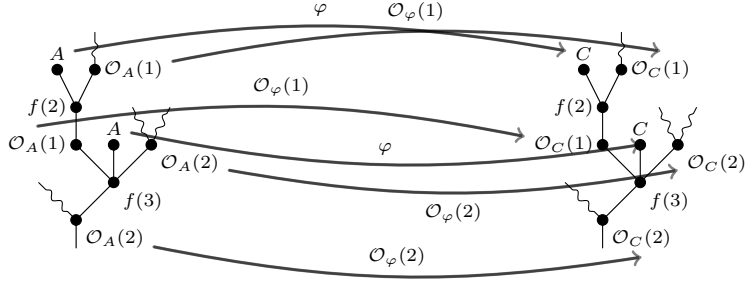
	\[
	\begin{tikzpicture}[level distance=5mm, sibling distance=5mm] 
	\tikzstyle{level 2}=[sibling distance=10mm] 
	\tikzstyle{level 3}=[sibling distance=5mm] 
	\tikzstyle{level 4}=[sibling distance=5mm] 
	\node {} [grow'=up]
	child{[fill] circle (2pt)
		child{edge from parent [decorate,decoration={snake,amplitude=.4mm,segment length=2mm,post length=.5mm,pre length=1.5mm}] {}}
		child{[fill] circle (2pt)
			child{[fill] circle (2pt)
				child{[fill] circle (2pt)
					child{[fill] circle (2pt) node [above] (Au) {$\scriptstyle A$}}
					child{[fill] circle (2pt)
						child{edge from parent [decorate,decoration={snake,amplitude=.4mm,segment length=2mm,post length=.5mm,pre length=1.5mm}] {}}
						node [right] (OA1u) {$\scriptstyle \mathcal O_A(1)$}}
					node [left] {$\scriptstyle f(2)$}}
				node [left] (OA1d) {$\scriptstyle \mathcal O_A(1)$}}
			child{[fill] circle (2pt) node [above] (Ad) {$\scriptstyle A$}}
			child{[fill] circle (2pt)
				child{edge from parent [decorate,decoration={snake,amplitude=.4mm,segment length=2mm,post length=.5mm,pre length=1.5mm}] {}}
				child{edge from parent [decorate,decoration={snake,amplitude=.4mm,segment length=2mm,post length=.5mm,pre length=1.5mm}] {}}
				node [below right] (OA2u) {$\scriptstyle \mathcal O_A(2)$}}
			node [below right] (U) {$\scriptstyle f(3)$}}
		node [below right] (OA2d) {$\scriptstyle \mathcal O_A(2)$}};
		\draw (7,0) node {} [grow'=up]
		child{[fill] circle (2pt)
			child{edge from parent [decorate,decoration={snake,amplitude=.4mm,segment length=2mm,post length=.5mm,pre length=1.5mm}] {}}
			child{[fill] circle (2pt)
				child{[fill] circle (2pt)
					child{[fill] circle (2pt)
						child{[fill] circle (2pt) node [above] (Bu) {$\scriptstyle C$}}
						child{[fill] circle (2pt)
							child{edge from parent [decorate,decoration={snake,amplitude=.4mm,segment length=2mm,post length=.5mm,pre length=1.5mm}] {}}
							node [right] (OB1u) {$\scriptstyle \mathcal O_C(1)$}}
						node [left] {$\scriptstyle f(2)$}}
					node [left] (OB1d) {$\scriptstyle \mathcal O_C(1)$}}
				child{[fill] circle (2pt) node [above] (Bd) {$\scriptstyle C$}}
				child{[fill] circle (2pt)
					child{edge from parent [decorate,decoration={snake,amplitude=.4mm,segment length=2mm,post length=.5mm,pre length=1.5mm}] {}}
					child{edge from parent [decorate,decoration={snake,amplitude=.4mm,segment length=2mm,post length=.5mm,pre length=1.5mm}] {}}
					node [below right] (OB2u) {$\scriptstyle \mathcal O_C(2)$}}
				node [below right] (U) {$\scriptstyle f(3)$}}
			node [below right] (OB2d) {$\scriptstyle \mathcal O_C(2)$}};
	\path[opacity=.5, very thick,bend left =10] (OA1u) edge [->] node [above, opacity=1] {$\scriptstyle \mathcal O_\varphi(1)$} (OB1u.north);
	\path[opacity=.5, very thick,bend left =10] (OA1d.north) edge [->] node [above, opacity=1] {$\scriptstyle \mathcal O_\varphi(1)$} (OB1d);
		\path[opacity=.5, very thick,bend left =10] (Au) edge [->] node [above, opacity=1] {$\scriptstyle \varphi$} (Bu);
		\path[opacity=.5, very thick,bend right =10] (Ad) edge [->] node [above, opacity=1] {$\scriptstyle \varphi$} (Bd.south);
	\path[opacity=.5, very thick,bend right =10] (OA2u) edge [->] node [below, opacity=1] {$\scriptstyle \mathcal O_\varphi(2)$} (OB2u);
	\path[opacity=.5, very thick,bend right =10] (OA2d) edge [->] node [above, opacity=1] {$\scriptstyle \mathcal O_\varphi(2)$} (OB2d.south);
	\end{tikzpicture}
	\]
	Hence $\tilde{\Phi}_t^{A}\r \tilde{\Phi}_t^{C}$ is a weak equivalence by Corollary \ref{condon2}. Now this result follows as in the last paragraph of Proposition \ref{masgen}.
\end{proof}

\begin{remark}\label{ende}
In some results above, we have assumed that tensor units are cofibrant or closely related things, e.g.~that cofibrant operads have underlying cofibrant sequences. These hypotheses can be relaxed in the following way. Proofs are essentially the same, mutatis mutandis, using the results on pseudo-cofibrant and $\unit$-cofibrant objects in \cite[Appendices A and B]{htnso2}, which show that these kinds of objects share many properties with cofibrant objects, assuming certain axioms when necessary. 	Cofibrant and $\unit$-cofibrant objects are always pseudo-cofibrant.  Cofibrant and pseudo-cofibrant objects coincide if and only if the tensor unit is cofibrant \cite[Lemma A.7]{htnso2}.
	We recently learnt that pseudo-cofibrant objects were previously introduced in \cite{mmmc}, where they are called semicofibrant.

	If $\C M$ is a monoidal model category, an object $F$ in $\C M^{\dos^n}_S$ is said to be \emph{$\unit$-cofibrant} or \emph{pseudo-cofibrant} if there is a cofibration $X\r F$ from the constant diagram on an $\unit$-cofibrant or pseudo-cofibrant object $X$ in $\C M$. This is equivalent to say that $F(0,\dots,0)$ is $\unit$-cofibrant or pseudo-cofibrant and that the relative latching map of $F$ at any $(x_1,\dots, x_n)\in \dos^n$, $(x_1,\dots, x_n)\neq (0,\dots,0)$, is a cofibration, and moreover a trivial cofibration if $x_i=1$ for some $i\in S$. The values and latching objects of $F$ are $\unit$-cofibrant or pseudo-cofibrant in $\C M$. 
	Under the strong unit axiom, any weak equivalence in $\C M^{\dos^n}_\varnothing$  between pseudo-cofibrant objects induces weak equivalences between latching objects. We define $\unit$-cofibrant and pseudo-cofibrant objects in big functor categories  $\C M^{\C N^n}$, $n\geq 0$, and $\C M^{\C N^{(\mathbb N)}}$ as in Definition \ref{loco}.
	
	We can define \emph{strong} homotopical notions in big functor categories, with $\C N$ a monoidal model category, by just requiring the sources (and hence the targets) of the cofibrations $g_i$ in Definition \ref{loco} to be cofibrant or $\unit$-cofibrant. Moreover, we can define \emph{very strong} homotopical notions 
	by more generally allowing the sources (and hence the targets) of the cofibrations $g_i$ to be pseudo-cofibrant. For instance, a map is a strong (resp.~very strong) weak equivalence in $\C M^{\C N^n}$ if and only if it yields a weak equivalence in $\C M$ when evaluated at $n$ ($\unit$-)cofibrant (resp.~pseudo-cofibrant) objects in $\C N$. 
	
	If $\C M$ is a monoidal model category, the identity functor is strongly $\unit$-cofibrant and very strongly pseudo-cofibrant in $\C M^{\C M}$. These properties are shared by the $n$-fold tensor product in $\C M^{\C M^n}$, $n\geq 2$, by the push-out product axiom. Moreover, if $Y$ is cofibrant/$\unit$-cofibrant/pseudo-cofibrant in $\C M$ then the functor $(X_1,\dots,X_n)\mapsto Y\otimes\bigotimes_{i=1}^nX_i$ is very strongly cofibrant/strongly $\unit$-cofibrant/very strongly pseudo-cofibrant in $\C M^{\C M^n}$, $n\geq 0$. Furthermore, if $f$ is a (trivial) cofibration in $\C M$  then the natural transformation $(X_1,\dots,X_n)\mapsto f\otimes\bigotimes_{i=1}^nX_i$ is a very strong (trivial) cofibration, $n\geq 0$. If $f$ is a weak equivalence between pseudo-cofibrant objects in $\C M$ and $\C M$ satisfies the strong unit axiom, then the previous natural transformation is a very strong weak equivalence.
	
	In Proposition \ref{mensgen}, if we assume that the objects $z(\mathcal O(n))$ are pseudo-cofibrant in $\C C$, $n\geq 0$, we derive that the enveloping functor-operad $\mathcal O_A$ of a cofibrant $\mathcal O$-algebra $A$ is very strongly pseudo-cofibrant in $\C C^{\C C^{(\mathbb N)}}$, and that any cofibration $A\r B$ between cofibrant $\mathcal O $-algebras induces a very strong cofibration between the enveloping functor-operads $\mathcal O_A\r \mathcal O_B$. Moreover, if $A\r B$ fits into a push-out \eqref{algebrapo} with $f$ a cofibration between pseudo-cofibrant objects then the objects $\mathcal O_{B,t}$ in Proposition \ref{transAcof} are also pseudo-cofibrant. 	
	A sufficient condition is that $\mathcal O$ is a sequence of ($\unit$-)cofibrant objects in $\C V$, since $z$ is a strong monoidal left Quillen functor. If, moreover, $\mathcal O(n)$ is cofibrant (resp.~$\unit$-cofibrant) for a certain $n\geq 0$ then $\mathcal O_A(n)$ is very strongly cofibrant (resp.~strongly $\unit$-cofibrant) in $\C C^{\C C^n}$. This applies to the associative operad, which is cofibrant in arity $0$ and $\unit$-cofibrant in higher arities. It also applies to any cofibrant operad $\mathcal O$, since $\mathcal O(1)$ is $\unit$-cofibrant and $\mathcal O(n)$ is cofibrant for all $n\neq 1$, see \cite[Corollary C.3]{htnso2}. %We recall that $z$ preserves ($\unit$-)cofibrant objects, since it is a strong monoidal left Quillen functor. Therefore, if $\mathcal O$ is a sequence of ($\unit$-)cofibrant objects, then so is $z(\mathcal O)$, and ($\unit$-)cofibrant objects are pseudo-cofibrant.
	
	In Proposition \ref{masgen}, $\phi_{\eta_{A}}$ is actually a very strong weak equivalence. Moreover, if $\C C$ satisfies the strong unit axiom then it is enough  to assume that  $z(\phi)\colon z(\mathcal O)\r z(\mathcal P)$ is a sequence of weak equivalences between pseudo-cofibrant objects. This condition on $z(\phi)$ holds if $\phi\colon\mathcal O\r\mathcal P$ is a weak equivalence between operads with underlying ($\unit$-)cofibrant sequences and $z$ satisfies the $\unit$-cofibrant axiom (which says that $z$ preserves weak equivalences between ($\unit$-)cofibrant objects).
	
	The transfer of (very) strong homotopical notions in big functor categories along Quillen pairs $L\colon\C M\rightleftarrows \C N\colon R$ between monoidal model categories works as follows. The `functor' $\C M^{\C M^n}\rightarrow\C N^{\C M^n}$ induced by left composition with $L$ preserves (very) strong cofibrations and trivial cofibrations. This only uses that $L$ is a left Quillen functor. If $L$ satisfies the pseudo-cofibrant axiom, then $L$ takes $\unit$-cofibrant objects to pseudo-cofibrant objects, so left composition with $L$ takes (very) strong $\unit$-cofibrant objects to (very) strong pseudo-cofibrant objects. Moreover, under the pseudo-cofibrant axiom,  right composition with cartesian powers of $L$, $\C N^{\C N^n}\r\C N^{\C M^n}$, takes each very strong homotopical notion to the corresponding strong homotopical notion. This is enough to ensure that Proposition \ref{masgen1/2} remains true if we replace the cofibrancy condition on the tensor units of $\C V$ and $\C W$ with the requirement that $\C V$, $\C W$, $\C C$ and $\C D$ satisfy the strong unit axiom and $F$, $\bar F$, $z_{\C C}$ and $z_{\C D}$ satisfy the pseudo-cofibrant and $\unit$-cofibrant axioms. We actually obtain that $\chi_{\mathcal O,A}$ is a strong weak equivalence.
	
	An operad $\mathcal O$ is \emph{very strongly pseudo-excellent} if the enveloping functor-operad $\mathcal O_A$ of an $\mathcal O$-algebra $A$ with underlying pseudo-cofibrant object is very strongly pseudo-cofibrant in $\C C^{\C C^{(\mathbb N)}}$ and any weak equivalence between such $\mathcal O$-algebras $\varphi\colon A\r C$ induces a very strong weak equivalence $\mathcal O_\varphi\colon\mathcal O_A\r\mathcal O_C$ in $\C C^{\C C^{(\mathbb N)}}$. This notion is neither stronger nor weaker than the notion of excellent operad. Under the strong unit axiom, the operads in Proposition \ref{sonex1} are also very strongly pseudo-excellent. Proposition \ref{marrown} holds if $\mathcal O$ is very stongly pseudo-excellent  and the $\mathcal O$-algebra $A$ has an underlying pseudo-cofibrant object ($B$ would just be pseudo-cofibrant in this case).  Theorem \ref{cuandoex} is true if $\C C$ satisfies the strong unit axiom, $\mathcal O$ is very strongly pseudo-excellent, and the underlying objects of $A$ and $C$ are just pseudo-cofibrant. 
	
	The pseudo-cofibrant version of Lemma \ref{condon} is more involved. If $\C C$ satisfies the strong unit axiom, $F\colon\C N^n\r\C M$ is very strongly pseudo-cofibrant in $\C M^{\C N^n}$ and, for $\tilde{\unit}$ a cofibrant replacement of the tensor unit, we have a natural isomorphism $F(\tilde{\unit}\otimes X_1,\dots,\tilde{\unit}\otimes X_n)\cong \tilde{\unit}^{\otimes n}\otimes F(X_1,\dots,X_n)$, then $F$ takes weak equivalences between pseudo-cofibrant objects $g_i\colon Y_i\r Z_i$, $1\leq i\leq n$, to a weak equivalence between pseudo-cofibrant objects $F(Y_1,\dots, Y_n)\r F(Z_1,\dots, Z_n)$. These hypotheses are satisfied by the components of the enveloping functor-operad of any $\mathcal O$-algebra $A$ with underlying pseudo-cofibrant object in $\C C$, provided the operad $\mathcal O$ is very strongly pseudo-excellent. Let us check the previous statement. The functor $\tilde{\unit}\otimes F$ is very strongly cofibrant in $\C M^{\C N^n}$, in particular it is cofibrant, so $\tilde{\unit}\otimes F(\tilde{\unit}\otimes g_1,\dots,\tilde{\unit}\otimes g_n)\cong \tilde{\unit}^{\otimes (n+1)}\otimes F(g_1,\dots,g_n)$ is a weak equivalence. Since $\tilde{\unit}^{\otimes (n+1)}$ is also a cofibrant replacement of the tensor unit, we deduce that $F(g_1,\dots,g_n)$ is also a weak equivalence. Its source and target are pseudo-cofibrant because $F$ is very strongly pseudo-cofibrant, so it takes pseudo-cofibrant values when evaluated at pseudo-cofibrant objects. The pseudo-version of Corollary \ref{condon2} also holds for very strong homotopical notions under these extra assumptions.
	
	The very strongly pseudo-cofibrant version of Lemma \ref{compcof} is obviously valid. Finally, Proposition \ref{sonex2} also holds for very strongly pseudo-excellent operads provided $\C C$ satisfies the strong unit axiom. It actually suffices that the objects $\mathcal O(n)$ are ($\unit$-)cofibrant, $n\geq 0$.
\end{remark}

%\bibliographystyle{amsplain}
%\bibliography{../Fernando}

\providecommand{\bysame}{\leavevmode\hbox to3em{\hrulefill}\thinspace}
\providecommand{\MR}{\relax\ifhmode\unskip\space\fi MR }
% \MRhref is called by the amsart/book/proc definition of \MR.
\providecommand{\MRhref}[2]{%
  \href{http://www.ams.org/mathscinet-getitem?mr=#1}{#2}
}
\providecommand{\href}[2]{#2}

\end{document}